\newcommand{\sqin}{\mathrel{\vphantom{\sqsubset}\text{\mathsurround=0pt\ooalign{$\sqsubset$\cr$-$\cr}}}}
\newcommand{\sqni}{\mathrel{\vphantom{\sqsupset}\text{\mathsurround=0pt\ooalign{$\sqsupset$\cr$-$\cr}}}}
\newcommand{\sqnii}{\mathrel{\vphantom{\sqsupset}\text{\mathsurround=0pt\ooalign{$\sqsupset$\cr$=$\cr}}}}

\newcommand{\PP}{\mathbb{P}}
\newcommand*{\QQ}{\mathbb{Q}}
\newcommand*{\RR}{\mathbb{R}}
\newcommand*{\maps}{\colon}
\newcommand*{\im}[1]{[#1]}

\newcommand*{\Spec}[1]{\mathsf{S}#1}

\newcommand*{\id}{\operatorname{id}}

\documentclass[oneside]{amsart}

\usepackage[colorlinks=true, urlcolor=black, citecolor=black, linkcolor=black, hyperfootnotes=true]{hyperref}
\usepackage{amssymb}
\usepackage{aliascnt}%for theorem counters
\usepackage{mathrsfs}
\usepackage{mathscinet}

\usepackage{enumitem} % resuming enumerations and better customization
\setlist[enumerate, 1]{label=\upshape (\arabic*)}

\usepackage{graphicx}%for Hasse diagram
\usepackage{tikz}%for Hasse diagram
\usetikzlibrary{arrows.meta} % for better default arrowtips
\usetikzlibrary{graphs} % for graph edge forming syntax
\usetikzlibrary{decorations.pathreplacing} % for long braces with labels

\numberwithin{equation}{section}

\newtheorem{thm}{Theorem}[section]

\newaliascnt{prp}{thm}
\newtheorem{prp}[prp]{Proposition}
\aliascntresetthe{prp}

\newaliascnt{lemma}{thm}
\newtheorem{lemma}[lemma]{Lemma}
\aliascntresetthe{lemma}

\newaliascnt{claim}{thm}

\aliascntresetthe{claim}

\newaliascnt{cor}{thm}
\newtheorem{cor}[cor]{Corollary}
\aliascntresetthe{cor}

\theoremstyle{definition}

\newaliascnt{dfn}{thm}
\newtheorem{dfn}[dfn]{Definition}
\aliascntresetthe{dfn}

\newaliascnt{xpl}{thm}
\newtheorem{xpl}[xpl]{Example}
\aliascntresetthe{xpl}

\newaliascnt{rmk}{thm}
\newtheorem{rmk}[rmk]{Remark}
\aliascntresetthe{rmk}

\author{Adam Barto\v{s}}
\address{Institute of Mathematics of the Czech Academy of Sciences, \v{Z}itn\'a 25, Prague}
\email{bartos@math.cas.cz}
\thanks{Research of Adam Bartoš was supported by GAČR project 20-31529X and RVO: 67985840.}

\author{Tristan Bice}
\thanks{Research of Tristan Bice was supported by GAČR project 22-07833K and RVO: 67985840.}
\address{Institute of Mathematics of the Czech Academy of Sciences, \v{Z}itn\'a 25, Prague}
\email{bice@math.cas.cz}

\author{Alessandro Vignati}
\address{
Institut de Math\'ematiques de Jussieu - Paris Rive Gauche (IMJ-PRG)\\
Universit\'e Paris Cit\'e\\
B\^atiment Sophie Germain\\
8 Place Aur\'elie Nemours \\ 75013 Paris, France}
\email{ale.vignati@gmail.com}
\urladdr{http://www.automorph.net/avignati}
\thanks{Research of Alessandro Vignati was supported by an Emergence en recherche grant from Universit\'e Paris Cit\'e and the ANR AGRUME}

\keywords{Posets, Filters, Bases, Compacta, Continua, Stone Duality}
\subjclass[2020]{06A07, 54D70, 54D80, 54E45, 54H10}

\title{Constructing Compacta from Posets}

\begin{document}

\begin{abstract}
We develop a simple method of constructing topological spaces from countable posets with finite levels, one which applies to all second countable $\mathsf{T}_1$ compacta.  This results in a duality amenable to building such spaces from finite building blocks, essentially an abstract analog of classical constructions defining compacta from progressively finer open covers.
\end{abstract}

\maketitle

\tableofcontents

\section*{Introduction}

\subsection*{Background}

Connections between topology and order theory have been central to a large body of mathematical research over the past century.  The idea behind much of this is to study abstract order structures like Boolean algebras, distributive lattices and semilattices etc. by representing them as families of subsets of topological spaces.  Stone was the first to initiate this line of research with his classic dualities in the 30's (see \cite{Stone1936} and \cite{Stone1938}) which have since been reformulated and extended in various ways by people such as Priestley \cite{Priestley1970}, Gr\"atzer \cite{Gratzer1978} and Celani--Gonzalez \cite{CelaniGonzalez2020}, just to name a few.  However, the spaces involved in these dualities typically have many compact open sets, which makes them quite different from the connected spaces more commonly considered in analysis.

In the opposite direction, other work has been motivated by the idea that topological spaces, particularly compacta, can be analysed from a more order theoretic perspective via (semi)lattices consisting of open sets.  This line of research was initiated by Wallman \cite{Wallman1938} and continued in various forms by people such as Shirota \cite{Shirota1952}, de Vries \cite{deVries1962}, Hofmann--Lawson \cite{HofmannLawson1978} and Jung--S\"underhauf \cite{JungSunderhauf1995}, with recent efforts to unify and extend these results also appearing in \cite{vanGool2012}, \cite{BiceStarling2018}, \cite{Bice2021GHLJS}, \cite{Kawai2021} and \cite{BiceKubis2020}.\footnote{In truth, Wallman worked lattices of closed sets, only in subsequent work did people consider lattices of open sets instead.  However, translating between open and closed sets is just a matter of reversing the order and taking complements where appropriate, as shown explicitly in \cite{BiceKubis2020}.}  In contrast to the work above, these dualites do encompass connected spaces.  However, so far they have not found many applications in actually building such spaces, like those considered in continuum theory.

One reason for this is that the order structures involved in these dualities are not so easily built from finite substructures.  In contrast, classical constructions of continua often proceed by building them up from finitary approximations, e.g. coming from simplicial complexes or finite open covers.  For example, the famous pseudoarc (see \cite{Bing1948} and \cite{Moise1948}) is usually built from successively finer chains of open subsets in $\mathbb{R}^2$, each chain being `crooked' in the previous chain.  Our work stems from the simple observation that the ambient space $\mathbb{R}^2$ here is essentially irrelevant, what really matters is just the poset arising from the inclusion relation between the links in the chains.  More precisely, the covers of the space are completely determined by the levels of the poset and these, in turn, determine the points of the space.  Indeed, points can be identified with their neighbourhood filters, which are nothing more than subsets of the poset `selecting' at least one element from each cover.

This leads us to consider a general class of posets formed from sequences of finite levels.  From any such poset, we construct a space of selectors, resulting in a $\mathsf{T}_1$ compactum on which the levels of the poset get represented as open covers.  Moreover, we will see that all second countable $\mathsf{T}_1$ compacta arise in this way.  Thus, at least in theory, it should be possible to construct any such space from a sequence of finite sets defining the levels of such a poset.  We further show that continuous functions between the resulting spaces can be completely described by certain relations between the original posets.  In this way we obtain a duality\footnote{Or rather an equivalence of categories, as we chose the direction of our relations so that the relevant functors are covariant (the term `duality' is often reserved for contravariant functors).} of a somewhat different flavour to those described above, one which has more potential applications to building spaces like the pseudoarc from finitary approximations.

\newpage
\subsection*{Outline}

To motivate our construction we first embark on a detailed analysis of bases of $\mathsf{T}_1$ compacta and the posets they form (when ordered by the usual inclusion relation $\subseteq$).  In particular, we examine special subsets of a poset as analogs of open covers, namely \emph{bands} and more general \emph{caps}.  On the one hand, caps are always covers, by \autoref{CapsCovers}.  Conversely, it is always possible to choose a basis of any second countable $\mathsf{T}_1$ compactum so that covers are caps.  We can also ensure that the basis forms an \emph{$\omega$-poset} where ranks and levels are always well-defined and finite.  Further order-topological properties of the resulting \emph{$\omega$-cap-bases} are also explored in \autoref{BasesAsPosets}, e.g. showing how they are simply characterised in metric compacta as the bases whose diameters converge to zero (see \autoref{OmegaCapBasesInMetricCompacta}).

In \autoref{TheSpectrum}, we show how to reverse this process, representing any $\omega$-poset $\mathbb{P}$ as an $\omega$-cap-basis $\mathbb{P}_\mathsf{S}$ of a suitably defined $\mathsf{T}_1$ compactum, namely its \emph{spectrum} $\mathsf{S}\mathbb{P}$ (which is then second countable, as $\mathbb{P}$ is countable).  Topological properties of $\mathsf{S}\mathbb{P}$ are thus determined by the order structure of $\mathbb{P}$.  Most notably, $\mathsf{S}\mathbb{P}$ is Hausdorff precisely when $\mathbb{P}$ is \emph{regular}, as shown in \autoref{RegularImpliesHausdorff}.  Subcompacta and subcontinua of $\mathsf{S}\mathbb{P}$ are also determined by special subsets of $\mathbb{P}$, as we show in \autoref{Subcompacta} and \autoref{Subcontinua}.  With an eye to our primary motivating example of the pseudoarc, we even show how to characterise hereditary indecomposability of $\mathsf{S}\mathbb{P}$ via \emph{tangled} refinements in $\mathbb{P}$ which, modulo regularity, generalise the original crooked refinements of Bing.

Finally, in \autoref{Functoriality}, we show how to encode continuous maps between spectra by certain relations between the posets we call \emph{refiners}.  A single continuous map can come from various different refiners and this flexibility yields homeomorphisms between spectra under some fairly general conditions explored in \autoref{Homeomorphisms}.  To obtain a more precise equivalence of categories, we turn our attention to \emph{strong refiners} in \autoref{StarComposition} under an appropriate \emph{star-composition}, thus yielding a combinatorial equivalent $\mathbf{S}$ of the category $\mathbf{K}$ of metrisable compacta.

\subsection*{Future Work}

Naturally, the next step would be to construct the posets themselves (as well as the refiners between them) in a more combinatorial way.  The basic idea would be to consider categories of finite graphs, much like in the work of Irwin--Solecki \cite{IrwinSolecki2006} and D\polhk{e}bski--Tymchatyn \cite{DebskiTymchatyn2018}, except with more general relational morphisms.  Sequences of such relations determine the levels of a graded $\omega$-poset, which then yield $\mathsf{T}_1$ compacta from the work presented here.  In particular, Fra\"iss\'e sequences in appropriate categories should yield canonical constructions of well-known compacta like the pseudoarc and Lelek fan.  Classical properties of these spaces relating to uniqueness and homogeneity could then be derived in a more canonical Fra\"iss\'e theoretic way, as we hope to demonstrate in future work.

\section{Bases as Posets}\label{BasesAsPosets}

Here we analyse bases of topological spaces, viewed as posets ordered by inclusion.  In particular, we explore how to characterise covers order theoretically and how to construct well-behaved bases satisfying certain order theoretic properties.

\subsection{Preliminaries}\label{Preliminaries}

We begin with some basic terminology and notation. We view any $\sqsubset\ \subseteq A\times B$ as a relation `from $B$ to $A$'.
We call $\sqsubset$
\begin{enumerate}
\item \emph{a function} if every $b\in B$ is related to exactly one $a\in A$.
\item \emph{surjective} if every $a\in A$ is related to at least one $b\in B$.
\item \emph{injective} if, for every $b\in B$, we have some $a\in A$ which is only related to $b$.
\end{enumerate}
These notions of surjectivity and injectivity for relations generalise the usual notions for functions.  The prefix `co' will be used to refer to the opposite/inverse relation $\sqsubset^{-1}\ =\ \sqsupset\ \subseteq B\times A$ (where $b\sqsupset a$ means $a\sqsubset b$), e.g. we say $\sqsupset$ is co-injective to mean that $\sqsubset$ is injective.  For example, one can note that every co-injective relation is automatically surjective, and the converse also holds for functions.

\begin{rmk}
    While this version of injectivity for relations may not be the most obvious generalisation from functions, it is the one we need for our work, being closely related to minimal covers -- see \autoref{prop:basicrefinement} below.  It is also natural from a categorical point of view, as the monic morphisms in the category of relations between sets are exactly those that are injective in this sense.  It also corresponds to injectivity of the image map $C\mapsto C^\sqsupset$ on subsets $C\subseteq B$ defined below, i.e. $\sqsubset$ is injective precisely when $C^\sqsupset=D^\sqsupset$ implies $C=D$, for all $C,D\subseteq B$.
\end{rmk}

 The motivating situation we have in mind is where $\sqsubset$ is the inclusion relation $\subseteq$ between covers $A$ and $B$ of a set $X$. In this case, $\sqsubset$ is surjective precisely when $A$ refines $B$ in the usual sense (we will also generalise refinement soon below). If $B$ is even a minimal cover, then $\sqsubset$ will also be injective, as we now show.

Let us denote the \emph{power set} of any set $X$ by
\[\mathsf{P}X=\{A:A\subseteq X\}.\]
To say $A\subseteq\mathsf{P}X$ \emph{covers} $X$ of course means $X=\bigcup A$.

\begin{prp}\label{prop:basicrefinement}
If $A,B\subseteq\mathsf{P}X$ cover $X$ and $\sqsubset\ =\ \subseteq$ on $A\times B$ then
\[\text{$B$ is minimal and $\sqsubset$ is surjective}\qquad\Rightarrow\qquad\text{$\sqsubset$ is injective}.\]
\end{prp}

\begin{proof}
If $B$ is a minimal cover of $X$ then every $b\in B$ must contain some $x\in X$ which is not in any other element of $B$, i.e. $x\in b\setminus\bigcup(B\setminus\{b\})$. If $A$ also covers $X$ then we must have some $a\in A$ containing $x$. If $\sqsubset$ is also surjective then we have some $c\in B$ with $x\in a\subseteq c$ and hence $c=b$. This shows that $a$ is only related to $b$, which in turn shows that $\sqsubset$ is injective.
\end{proof}

Again take a relation $\sqsubset\ \subseteq A\times B$. The \emph{preimage} of any $S\subseteq A$ is given by
\begin{align}
\tag{Preimage}S^\sqsubset=\mathop{\sqsupset}[S]=\{b\in B:\exists s\in S\ (s\sqsubset b)\}.\\
\intertext{Likewise, the \emph{image} of any $T\subseteq B$ is the preimage of the opposite relation $\sqsupset$, i.e.}
\tag{Image}T^\sqsupset=\mathop{\sqsubset}[T]=\{a\in A:\exists t\in T\ (a\sqsubset t)\}.
\end{align}
We say $S\subseteq A$ \emph{refines} $T\subseteq B$ if it is contained in its image, i.e. $S\subseteq T^\sqsupset$. Equivalently, $S$ refines $T$ when the restriction of $\sqsubset$ to $S\times T$ is surjective. The resulting refinement relation will also be denoted by $\sqsubset$, i.e. for any $S\subseteq A$ and $T\subseteq B$,
\[S\sqsubset T\qquad\Leftrightarrow\qquad S\subseteq T^\sqsupset\qquad\Leftrightarrow\qquad\forall s\in S\ \exists t\in T\ (s\sqsubset t).\]
Likewise, the \emph{corefinement} relation will also be denoted by $\sqsupset$, i.e.
\[T\sqsupset S\qquad\Leftrightarrow\qquad T\subseteq S^\sqsubset\qquad\Leftrightarrow\qquad\forall t\in T\ \exists s\in S\ (s\sqsubset t).\]
(so refinement and corefinement are not inverses, i.e. $S\sqsubset T$ does not mean $T\sqsupset S$).  Here again the motivating situation we have in mind is when $\sqsubset$ is the inclusion relation or, more generally, some partial order or even preorder (recall that a \emph{preorder} is a reflexive transitive relation, while a \emph{partial order} is an antisymmetric preorder).

Given a preorder $\leq$ on a set $\mathbb{P}$, we define ${<}={\leq}\cap{\neq}$, i.e.
\[p<q\qquad\Leftrightarrow\qquad p\leq q\quad\text{and}\quad p\neq q.\]
The \emph{antichains}\footnote{Note that these are more general than the \emph{strong antichains} usually considered by set theorists (which are defined to be subsets $A$ of $\mathbb{P}$ in which no pair in $A$ has a common lower bound in $\mathbb{P}$).} of $\mathbb{P}$ will be denoted by
\[\mathsf{A}\mathbb{P}=\{A\subseteq\mathbb{P}:\forall q,r\in A\ (q\nless r\text{ and }r\nless q)\}.\]

\begin{prp}
If $\leq$ is a preorder on $\mathbb{P}$ then so is the refinement relation on $\mathsf{P}\mathbb{P}$. If $\leq$ is a partial order then so is the refinement relation when restricted to $\mathsf{A}\mathbb{P}$.
\end{prp}

\begin{proof}
If $\leq$ is reflexive on $\mathbb{P}$ and $Q\subseteq\mathbb{P}$ then $q\leq q$, for all $q\in Q$, showing that $Q\leq Q$, i.e. $\leq$ is also reflexive on $\mathsf{P}\mathbb{P}$. On the other hand, if $Q\leq R\leq S$ then, for any $q\in Q$, we have $r\in R$ with $q\leq r$, which in turn yields $s\in S$ with $r\leq s$. If $\leq$ is transitive on $\mathbb{P}$ then $q\leq s$, showing that $Q\leq S$, i.e. $\leq$ is also transitive on $\mathsf{P}\mathbb{P}$.

Finally, say $\leq$ is also antisymmetric on $\mathbb{P}$ and $Q\leq R\leq Q$, for some antichains $Q,R\in\mathsf{A}\mathbb{P}$. For all $q\in Q$, we thus have $r\in R$ with $q\leq r$, which in turn yields $q'\in Q$ with $q\leq r\leq q'$. Thus $q=q'$, as $Q$ is an antichain, and hence $q=r$, as $\leq$ is antisymmetric on $\mathbb{P}$. This shows that $Q\subseteq R$, while $R\subseteq Q$ follows dually.
\end{proof}

We will also need to compose relations, which we do in the usual way, i.e. if $\sqsubset\ \subseteq A\times B$ and $\sqin\ \subseteq B\times C$ then $\sqsubset\circ\sqin\ \subseteq A\times C$ is defined by
\[\tag{Composition}\label{Composition}a\sqsubset\circ\sqin c\qquad\Leftrightarrow\qquad\exists b\in B\ (a\sqsubset b\sqin c).\]
Note this is consistent with the usual composition of functions as we are taking the domain of a function to correspond to the right coordinate not the left, i.e. a function $f:B\rightarrow A$ from $B$ to $A$ is a subset of $A\times B$ (not $B\times A$).

As in \cite{Bing1948} (see also \cite{Moise1948}), we say that $B$ \emph{consolidates} $A$ when $A$ refines $B$ and every $b\in B$ is a union of elements of $A$, i.e. $b=\bigcup(b^\supseteq\cap A)=\bigcup\{a\in A:a\subseteq b\}$.

\begin{prp}
Take $A,B,C\subseteq\mathsf{P}X$ with $\sqsubset\ \subseteq A\times B$ and $\sqin\ \subseteq B\times C$ defined to be restrictions of the inclusion relation $\subseteq$ on $\mathsf{P}X$. For any $a\in A$ and $c\in C$,
\[a\sqsubset\circ\sqin c\qquad\Rightarrow\qquad a\subseteq c.\]
Conversely, if $A$ is a minimal cover, $B$ consolidates $A$ and $C$ consolidates $B$ then
\[a\subseteq c\qquad\Rightarrow\qquad a\sqsubset\circ\sqin c.\]
\end{prp}

\begin{proof}
Certainly $a\subseteq b\subseteq c$ implies $a\subseteq c$. Conversely, say $a\subseteq c$ and $A$ is a minimal cover so we have $x\in a\setminus\bigcup(A\setminus\{a\})$. If $c=\bigcup c^{\sqni}$ then we have $b\in B$ with $x\in b\subseteq c$. If $b=\bigcup b^\sqsupset$ too then we have $a'\in A$ with $x\in a'\subseteq b$ and hence $a=a'$, i.e. $a\subseteq b\subseteq c$ and hence $a\sqsubset\circ\sqin c$.
\end{proof}

\subsection{Bands and Caps}

Let us denote the finite subsets of a set $X$ by
\[\mathsf{F}X=\{F\subseteq X:|F|<\infty\}.\]

The following special subsets of our poset $\mathbb{P}$ form the key order theoretic analogs of open covers that are fundamental to our work.

\begin{dfn}\label{BandsAndCaps}
Take a poset $(\mathbb{P},\leq)$.
\begin{enumerate}
\item We call $B\in\mathsf{F}\mathbb{P}$ a \emph{band} if each $p\in\mathbb{P}$ is comparable to some $b\in B$.
\item We call $C\in\mathsf{P}\mathbb{P}$ a \emph{cap} if $C$ is refined by some band.
\end{enumerate}
\end{dfn}

\begin{rmk}
    There is also the related notion of a \emph{cutset} from \cite{SauerWoodrow1984}, which is a subset $C$ of $\mathbb{P}$ overlapping (i.e. intersecting) every maximal chain in $\mathbb{P}$.  Put another way, these are precisely the transversals of maximal cliques of the comparability graph of $\mathbb{P}$, as studied in \cite{BellGinsburg1984}. Similarly, bands are the finite dominating subsets of the comparability graph.  By Kuratowski--Zorn, every element of a poset is contained in a maximal chain and hence every finite cutset is a band.  However, the converse can fail, e.g. if $\mathbb{P}=\{a,b,c,d\}$ with $<\ =\{(a,c),(b,c),(b,d)\}$ then $\{a,d\}$ is a band but not a cutset, as it fails to overlap the maximal chain $\{b,c\}$.  Although in graded $\omega$-posets, every level is a cutset and so in this case every band and hence every cap is at least refined by a finite cutset, thanks to \autoref{CapsRefineLevels} below.
\end{rmk}

We denote the bands and caps of $\mathbb{P}$ by
\begin{align}
\tag{Bands}\mathsf{B}\mathbb{P}&=\{B\in\mathsf{F}\mathbb{P}:\mathbb{P}=B^\leq\cup B^\geq\}.\\
\tag{Caps}\mathsf{C}\mathbb{P}&=\{C\in\mathsf{P}\mathbb{P}:\exists B\in\mathsf{B}\mathbb{P}\ (B\leq C)\}.
\end{align}
The primary example we have in mind is when $\mathbb{P}$ is a basis of some topological space $X$ ordered by inclusion $\subseteq$.  In this case, caps are meant to correspond to covers of the space $X$.  More precisely, we have the following.

\begin{prp}\label{CapsCovers}
If $\mathbb{P}$ is a basis of non-empty open sets of some $\mathsf{T}_1$ topological space $X$ ordered by inclusion $($i.e. $\leq\ =\ \subseteq)$ then every cap covers $X$, i.e.
\begin{equation}\label{CapsAreCovers}
C\in\mathsf{C}\mathbb{P}\qquad\Rightarrow\qquad\bigcup C=X.
\end{equation}
\end{prp}

\begin{proof}
Note that if $B$ refines $C$ and $\bigcup B=X$ then $\bigcup C=X$. Thus it is enough to show that $\bigcup B=X$ whenever $B$ is a band. Take a band $B$ and suppose that we have $x\in X\setminus\bigcup B$. For each $b\in B$, let $x_b$ be a point in $b$. As $X$ is $\mathsf{T}_1$, $c=X\setminus\{x_b:b\in B\}$ is an open set containing $x$. As $\mathbb P$ is a basis, there is $d\in\mathbb P$ such that $x\in d\subseteq c$. For each $b\in B$, note $x_b\in b\setminus d$ so $b\nsubseteq d$ while $x\in d\setminus b$ so $d\nsubseteq b$. This shows that $B$ is not a band, a contradiction.
\end{proof}

The converse of \eqref{CapsAreCovers}, however, can fail. We can even show that there is no way to identify the covers of a space purely from the inclusion order on an arbitrary basis. Indeed, in the following two examples we have bases of different compact Hausdorff spaces which are isomorphic as posets but have different covers.  Specifically, the bases are both isomorphic to the unique countable atomless pseudo-Boolean algebra (the \emph{atoms} of a poset $\mathbb{P}$ are its minimal elements and $\mathbb{P}$ is \emph{atomless} and if it has no atoms, while a \emph{pseudo-Boolean algebra} is a poset $\mathbb{P}$ formed from a Boolean algebra $\mathbb{B}$ minus its bottom element $0$, i.e. $\mathbb{P}=\mathbb{B}\setminus\{0\}$).

\begin{xpl}
The interval $X=[0,1]$ in its usual topology has a basis $\mathbb{P}$ consisting of non-empty regular open sets which are unions of finitely many intervals with rational endpoints (note regularity disqualifies sets like $(0,\frac{1}{2})$ and $(\frac{1}{4},\frac{1}{2})\cup(\frac{1}{2},\frac{3}{4})$, only the interior of their closures $[0,\frac{1}{2})$ and $(\frac{1}{4},\frac{3}{4})$ lie in $\mathbb{P}$).  One immediately sees that $\mathbb{P}$ is then a countable atomless pseudo-Boolean algebra with respect to the inclusion ordering.  We also see that $p,q\in\mathbb{P}$ are disjoint precisely when they have no lower bound in $\mathbb{P}$, and no such $p$ and $q$ cover $X$.
\end{xpl}

\begin{xpl}
The Cantor space $X=\{0,1\}^\omega$ has a basis $\mathbb{P}$ consisting of all non-empty clopen sets. Again $\mathbb{P}$ is a countable atomless pseudo-Boolean algebra and $p,q\in\mathbb{P}$ are disjoint precisely when they have no lower bound in $\mathbb{P}$.  However, this time there are many disjoint $p,q\in\mathbb{P}$ that cover $X$.
\end{xpl}

In fact, if $\mathbb{P}$ is the countable atomless pseudo-Boolean algebra then its bands and caps are all trivial in that they must contain the top element. This poset does, however, have a subposet isomorphic to the full countable binary tree $2^{<\omega}$, which is still isomorphic to a basis of the Cantor space (but not the unit interval anymore).  In this case, caps of $2^{<\omega}$ do indeed correctly identify the covers of the Cantor space.  This suggests that we might be able to ensure covers of other spaces are also caps by choosing the basis more carefully.  In other words, we might be able to find `cap-bases' or even `band-bases' in the following sense.

\begin{dfn}\label{CapBasis}
We call a basis $\mathbb{P}$ of a topological space $X$ a
\begin{enumerate}
    \item \emph{band-basis} if $\mathsf{B}\mathbb{P}=\{B\in\mathsf{F}\mathbb{P}:X=\bigcup B\}$.
    \item \emph{cap-basis} if $\mathsf{C}\mathbb{P}=\{C\in\mathsf{P}\mathbb{P}:X=\bigcup C\}$.
\end{enumerate}
\end{dfn}

Note every element of a cap-basis or band-basis $\mathbb{P}$ of a non-empty space $X$ must also be non-empty -- otherwise $\emptyset$ would be a minimum of $\mathbb{P}$ and hence a band of $\mathbb{P}$ which does not cover $X$, contradicting the definition.  Further observe that, as every cap contains a finite subcap, every space with a cap-basis is automatically compact.  And every band-basis of a compact space is a cap-basis, as every cover has a finite subcover which is then a band and hence a cap.  Also, to verify that a basis of non-empty open sets of a $\mathsf{T}_1$ space is a band/cap-basis, it suffices to show that covers are bands/caps, as the converse follows from \eqref{CapsAreCovers}.

\begin{prp}\label{CoversAreCaps}
Every second countable compact $\mathsf{T}_1$ space has a cap-basis.
\end{prp}

\begin{proof}
To start with, take any countable basis $B$ of a compact $\mathsf{T}_1$ space $X$ and let $(C_n)_{n\in\omega}$ enumerate all finite minimal covers of $X$ from $B$.  Recursively define $(n_k)_{k\in\omega}$ as follows. Let $n_0$ be arbitrary. If $n_k$ has been defined then note that, for any $x\in X$, we have $p\in C_{n_k}$ and $q\in C_k$ with $x\in p\cap q$. As $B$ is a basis, we thus have $b\in B$ with $x\in b\subseteq p\cap q$. By compactness, $X$ has a finite minimal cover of such $b$'s. This means we have $n_{k+1}\in\omega$ such that $C_{n_{k+1}}$ refines both $C_{n_k}$ and $C_k$.

Set $B_k=C_{n_k}$ and $\mathbb{P}=\bigcup_{k\in\omega}B_k$. First note that $\mathbb{P}$ is still a basis for $X$. Indeed, if $x\in b\in B$ then, as $X$ is $\mathsf{T}_1$, we can cover $X\setminus b$ with elements of $B$ avoiding $x$. Compactness then yields a finite minimal subcover, i.e. we have some $k\in\omega$ with $b\in C_k$ and $x\notin\bigcup(C_k\setminus\{b\})$. Taking $c\in B_{k+1}$ with $x\in c$, it follows that $c\subseteq b$, as $B_{k+1}$ refines $C_k$ and $b$ is the only element of $C_k$ containing $x$. In particular, we have found $c\in\mathbb{P}$ with $x\in c\subseteq b$, showing that $\mathbb{P}$ is a basis for $X$.

By definition, $B_{k+1}$ refines $B_k$. We claim $B_k$ also corefines $B_{k+1}$, i.e. $B_k\subseteq B_{k+1}^\subseteq$. Indeed, as $B_k$ is a minimal cover, for every $p\in B_k$, we have $x\in p\setminus\bigcup(B_k\setminus\{p\})$. Taking $q\in B_{k+1}$ with $x\in q$, we see that $q\subseteq p$, as $B_{k+1}$ refines $B_k$ and no other element of $B_k$ contains $x$. This proves the claim and hence each $B_k$ is a band of $\mathbb{P}$. As every cover of $X$ from $B$ (and, in particular, $\mathbb{P}$) is refined by some $B_k$, it follows that every cover of $X$ from $\mathbb{P}$ is a cap of $\mathbb{P}$, i.e. $\mathbb{P}$ is a cap-basis.
\end{proof}

Note the cap-bases in the above proof are Noetherian, which have also been studied independently (see e.g. \cite{Grabner1983}). In general, we call a poset $\mathbb{P}$ \emph{Noetherian} if every subset of $\mathbb{P}$ has a maximal element or, equivalently, if $\mathbb{P}$ has no strictly increasing sequences. Put another way, this is saying that $>$ (where $a>b$ means $a\geq b\neq a$) is \emph{well-founded} in the sense of \cite[Definition I.6.21]{Kunen2011}. Like in \cite[\S I.9]{Kunen2011}, we then recursively define the \emph{rank} $\mathsf{r}(p)$ of any $p\in\mathbb{P}$ as the ordinal given by
\[\mathsf{r}(p)=\sup_{q>p}(\mathsf{r}(q)+1).\]
So maximal elements of $\mathbb{P}$ have rank $0$, maximal elements among the remaining subset have rank $1$ and so on. For any ordinal $\alpha$, we denote the $\alpha^\mathrm{th}$ \emph{cone} of $\mathbb{P}$ by
\[\mathbb{P}^\alpha=\{p\in\mathbb{P}:\mathsf{r}(p)\leq\alpha\}.\]
The atoms of the $\alpha^\mathrm{th}$ cone form the $\alpha^\mathrm{th}$ \emph{level} of $\mathbb{P}$, denoted by
\[\mathbb{P}_\alpha=\{p\in\mathbb{P}^\alpha:p^>\cap\mathbb{P}^\alpha=\emptyset\}\]
Note $\mathsf{r}^{-1}\{\alpha\}\subseteq\mathbb{P}_\alpha$, i.e. the $\alpha^\mathrm{th}$ level contains all elements of rank $\alpha$. But this inclusion can be strict, i.e. the $\alpha^\mathrm{th}$ level can also contain elements of smaller rank (e.g. the $\alpha^\mathrm{th}$ level always contains all atoms of $\mathbb{P}$ of rank smaller than $\alpha$).

If some level $\mathbb{P}_\alpha$ of a Noetherian poset $\mathbb{P}$ has finitely many elements then it is immediately seen to be a band.  If $\mathbb{P}$ here is a basis of non-empty sets of a $\mathsf{T}_1$ space $X$ then it follows that $\mathbb{P}_\alpha$ covers $X$, by \autoref{CapsCovers}.  In fact, even if $\mathbb{P}_\alpha$ has infinitely many elements, it will still cover $X$ as long as $\alpha$ is finite.
 
\begin{prp}\label{LevelsCovers}
    If $\mathbb{P}$ is a Noetherian basis for a $\mathsf{T}_1$ space $X$ then $X=\bigcup\mathbb{P}_n$, for all $n\in\omega$.
\end{prp}

\begin{proof}
    Every $x\in X$ lies in some $p\in\mathbb{P}$, as $\mathbb{P}$ is a basis.  As $\mathbb{P}$ is Noetherian, we then have $p_0\in\mathbb{P}_0$ with $p\subseteq p_0$ and hence $x\in p_0$ too.  This shows that $\mathbb{P}_0$ covers $X$.  Now say $\mathbb{P}_n$ covers $X$.  This means any $x\in X$ lies in some $p\in\mathbb{P}_n$.  If $p=\{x\}$ then $p$ is an atom of $\mathbb{P}$ and hence $p\in\mathbb{P}_{n+1}$ too.  Otherwise, we have $y\in p\setminus\{x\}$ and hence $p\setminus\{y\}$ is an open neighbourhood of $x$, as $X$ is $\mathsf{T}_1$.  Then we have $q\in\mathbb{P}$ with $x\in q\subseteq p\setminus\{y\}$, as $\mathbb{P}$ is a basis, necessarily with $\mathsf{r}(q)>\mathsf{r}(p)=n$.  Thus we have $r\in\mathbb{P}_{n+1}$ with $x\in q\subseteq r$, showing that $\mathbb{P}_{n+1}$ also covers $X$.  By induction, $\mathbb{P}_n$ thus covers $X$, for all $n\in\omega$.
\end{proof}

\subsection{\texorpdfstring{$\omega$}{Omega}-Posets}

We call a poset $\mathbb{P}$ an \emph{$\omega$-poset} if every principal filter $p^\leq$ is finite and the number of principal filters of size $n$ is also finite, for any $n\in\omega$.  Equivalently, an $\omega$-poset is a Noetherian poset in which both the rank of each element of $\mathbb{P}$ and the size of each level (or cone) of $\mathbb{P}$ is finite.  For example, taking any $\omega$-tree in the sense of \cite[III.5.7]{Kunen2011} and replacing $\leq$ with $\geq$ yields an $\omega$-poset.

The nice thing about $\omega$-posets is that their levels determine the caps, specifically caps are precisely the subsets refined by some level.  Put another way, the levels are coinitial with respect to refinement within the family of all caps (and even bands).

\begin{prp}\label{CapsRefineLevels}
If $\mathbb{P}$ is an $\omega$-poset then its levels $(\mathbb{P}_n)$ are coinitial in $\mathsf{B}\mathbb{P}$.
\end{prp}

\begin{proof}
First note that each level $\mathbb{P}_n$ is a band. Indeed, if $\mathsf{r}(p)\leq n$ then $p$ must be above some minimal element of $\mathbb{P}^n$, i.e. some element of $\mathbb{P}_n$. On the other hand, if $\mathsf{r}(p)\geq n$ then $p$ is below some element of rank $n$, which must again lie in $\mathbb{P}_n$.

Conversely, say $B\subseteq\mathbb{P}$ is a band and let $n=\max_{b\in B}\mathsf{r}(b)$ so $B\subseteq\mathbb{P}^n$. It follows that no atom of $\mathbb{P}^n$ can be strictly above any element of $B$. Thus every element of $\mathbb{P}_n$ must be below some element of $B$, as $B$ is a band, i.e. $\mathbb{P}_n\leq B$.
\end{proof}

In particular, the bands and caps of any $\omega$-poset are (downwards) directed with respect to refinement.  The fact that the levels here are finite is crucial, i.e. there are simple examples of Noetherian posets for which this fails.

\begin{xpl}
Take a poset $\mathbb{P}$ consisting of two incomparable $q,r\in\mathbb{P}$ together with infinitely many incomparable elements which all lie below both $q$ and $r$, i.e. $\mathbb{P}\setminus\{q,r\}=q^>=r^>$ is infinite and $s\nleq t$, for all distinct $s,t\in\mathbb{P}\setminus\{q,r\}$.  This poset is Noetherian with two levels, although only the top level is finite. Note $\{q,r\}$ is a band of $\mathbb{P}$ but the only other bands of $\mathbb{P}$ contain at least one element of both $\{q,r\}$ and $\mathbb{P}\setminus\{q,r\}$, while the caps of $\mathbb{P}$ are precisely those subsets containing $q$ and/or $r$. In particular, no cap refines both the singleton caps $\{q\}$ and $\{r\}$.
\end{xpl}

Another simple observation about caps in $\omega$-posets is the following.

\begin{prp}\label{CapAntichains}
    If $\mathbb{P}$ is an $\omega$-poset then no infinite cap is an antichain, i.e.
    \[\mathsf{A}\mathbb{P}\cap\mathsf{C}\mathbb{P}\subseteq\mathsf{F}\mathbb{P}.\]
\end{prp}

\begin{proof}
    If $C$ is an infinite cap then $B\leq C$, for some band $B$.  In particular, $B$ is finite so we must have $c\in C$ with $\mathsf{r}(c)>\max_{b\in B}\mathsf{r}(b)$.  As $B$ is a band, we then have $b\in B\cap c^<$.  As $B\leq C$, we then have $c'\in C$ with $c<b\leq c'$, showing that $C$ is not an antichain.
\end{proof}

We are particularly interested in $\omega$-posets arising from bases.

\begin{dfn}
    A (band/cap-)basis that is also an $\omega$-poset (w.r.t. inclusion $\subseteq$) will be called an \emph{$\omega$-$($band/cap-$)$basis}.
\end{dfn}

The proof of \autoref{CoversAreCaps} shows that every second countable $\mathsf{T}_1$ compactum has an $\omega$-cap-basis.  Further note that if the space there is Hausdorff then it is metrisable.  In compact metric spaces, we can actually characterise $\omega$-cap-bases as precisely the countable bases whose diameters converge to zero.

\begin{prp}\label{OmegaCapBasesInMetricCompacta}
If $X$ is a compact metric space with a countable basis $\mathbb{P}$ of non-empty open sets then, for any enumeration $(p_n)$ of $\mathbb{P}$,
\[\mathbb{P}\text{ is an $\omega$-cap-basis}\qquad\Leftrightarrow\qquad\mathrm{diam}(p_n)\rightarrow0.\]
\end{prp}

\begin{proof}
For $\varepsilon\in(0,1)$, let $\mathbb{P}_\varepsilon=\{p\in\mathbb{P}:\mathrm{diam}(p)<\varepsilon\}$, so what we want to show is
\[\mathbb{P}\text{ is an $\omega$-cap-basis}\qquad\Leftrightarrow\qquad\mathbb{P}\setminus\mathbb{P}_\varepsilon\text{ is finite, for all }\varepsilon>0.\]
First say $\mathbb{P}\setminus\mathbb{P}_\varepsilon$ is infinite, for some $\varepsilon>0$.  Assuming $\mathbb{P}$ is an $\omega$-poset (otherwise we are already done), this means that every level of $\mathbb{P}$ contains a set with diameter at least $\varepsilon$.  By \autoref{CapsRefineLevels}, the same is true of all caps.  This means that the cover $\mathbb{P}_\varepsilon$ of $X$ can not be a cap and hence $\mathbb{P}$ is not a cap-basis.

Conversely, say $\mathbb{P}\setminus\mathbb{P}_\varepsilon$ is finite, for all $\varepsilon>0$.  As $p\subseteq q$ implies $\mathrm{diam}(p)\leq\mathrm{diam}(q)$, $\mathbb{P}$ is Noetherian and the rank of each element is finite.

We claim every level $\mathbb{P}_n$ of $\mathbb{P}$ covers $X$.  To see this, take any $x\in X$.  If $x$ is not isolated then we must have a sequence in $\mathbb{P}$ of neighbourhoods of $x$ which is strictly decreasing with respect to inclusion.  There are then sets in $\mathbb{P}$ of arbitrary rank containing $x$, in particular we have some $p_x\in\mathbb{P}$ with $x\in p_x$ and $\mathsf{r}(p_x)=n$ and hence $p_x\in\mathbb{P}_n$.  On the other hand, if $x$ is isolated then either $\{x\}\in\mathbb{P}^n$ and hence we may take $p_x=\{x\}\in\mathbb{P}_n$, or $\mathsf{r}(\{x\})>n$ and hence we again have $p_x\in\mathbb{P}$ with $x\in p_x$ and $\mathsf{r}(p_x)=n$ so $p_x\in\mathbb{P}_n$.  Then $\{p_x:x\in X\}\subseteq\mathbb{P}_n$ covers $X$, as claimed.

If $\mathbb{P}$ were not an $\omega$-poset then $\mathbb{P}$ would have some infinite level $\mathbb{P}_n$.  By the claim just proved, $\mathbb{P}_n$ would then cover $X$ and hence have some finite subcover $F\subseteq\mathbb{P}_n$.  By the Lebesgue number lemma (see \cite[Lemma 27.5]{Munkres2000}), any cover of a compact metric space is uniform, i.e. we have some $\varepsilon>0$ such that every subset of diameter at most $\varepsilon$ is contained in some set in the cover.  In particular, we have some $\varepsilon>0$ such that $\mathbb{P}_\varepsilon$ refines $F$.  As $\mathbb{P}_n$ is infinite, we can take some $p\in\mathbb{P}_n\setminus F$ with $\mathrm{diam}(p)<\varepsilon$.  But then $p\subseteq f$, for some $f\in F$, contradicting the fact that elements in the same level are incomparable.  Thus $\mathbb{P}$ is indeed an $\omega$-poset.

For any $\varepsilon>0$, we next claim that $\mathbb{P}_\varepsilon$ contains a band.  To see this first note that $\mathbb{P}_\varepsilon$ is still a basis for $X$.  In particular, $\mathbb{P}_\varepsilon$ covers $X$ and hence we have a finite subcover $F\subseteq\mathbb{P}_\varepsilon$.  Again, we have some $
\delta>0$ such that $\mathbb{P}_\delta$ refines $F$, i.e. $\mathbb{P}_\delta\leq F$.  As $\mathbb{P}\setminus\mathbb{P}_\delta$ is finite, we also have finite $E\subseteq\mathbb{P}_\varepsilon$ with $\mathbb{P}\setminus\mathbb{P}_\delta\geq E$.  Thus $E\cup F$ is a band of $\mathbb{P}$ contained in $\mathbb{P}_\varepsilon$, proving the claim.

Now take any cover $C\subseteq\mathbb{P}$ of $X$.  Again $C$ is uniform and is thus refined by $\mathbb{P}_\varepsilon$, for some $\varepsilon>0$, and hence by some band $B\subseteq\mathbb{P}_\varepsilon$, i.e. $C$ is a cap.  Conversely, caps are covers, by \eqref{CapsAreCovers}, so $\mathbb{P}$ is indeed a cap-basis.
\end{proof}

Note that if $U$ is an up-set of an $\omega$-poset $\mathbb{P}$, i.e. $U^\leq\subseteq U$, then $U$ is again an $\omega$-poset in the induced ordering $\leq_U\ =\ \leq\cap\ (U\times U)$.  Indeed, $U$ being an up-set implies that the rank within $U$ of any element of $U$ is the same as its rank within the original $\omega$-poset $\mathbb{P}$.  As long as $U$ does not contain any extra atoms, the caps of $U$ will also all come from caps of $\mathbb{P}$ in a canonical way.

\begin{prp}\label{UpsetCaps}
If $\mathbb{P}$ is an $\omega$-poset then, for all $U\subseteq\mathbb{P}$,
\begin{equation}\label{CUsub}
\mathsf{C}U\subseteq\{C\cap U:C\in\mathsf{C}\mathbb{P}\}.
\end{equation}
Moreover, equality holds if $U$ is an up-set whose atoms are all already atoms in $\mathbb{P}$.
\end{prp}

\begin{proof}
First note that, for any finite $F\subseteq\mathbb{P}$, we can find a level $\mathbb{P}_n$ whose overlap with $F$ consists entirely of atoms.  Indeed, as $F$ is finite, we can find a cone $\mathbb{P}^n$ overlapping $f^>$, for each $f\in F$ that is not an atom.  This means non-atomic elements of $F$ are never minimal in $\mathbb{P}^n$ and hence $\mathbb{P}_n$ is the required level.  In particular, if $F$ contains no atoms at all then it is disjoint from $\mathbb{P}_n$.

Now take any $C\in\mathsf{C}U$, which is refined by some $B\in\mathsf{B}U$.  We thus have a level $\mathbb{P}_n$ disjoint from $B^<$.  As $B$ is a band of $U$, for any $u\in\mathbb{P}_n\cap U$, we have some comparable $b\in B$.  As $u\notin B^<$, it follows that $u\leq b$.  This shows that $\mathbb{P}_n\cap U$ refines $B$ and hence $C$.  Thus $\mathbb{P}_n$ refines $C\cup(\mathbb{P}_n\setminus U)$, which is thus a cap of $\mathbb{P}$ whose intersection with $U$ is the original $C$.  This proves \eqref{CUsub}.

Conversely, take $C\in\mathsf{C}\mathbb{P}$, which is refined by some $B\in\mathsf{B}\mathbb{P}$.  If $U$ is an up-set and hence an $\omega$-poset in its own right then we have some level $U_n$ of $U$ such that $B^<\cap U_n$ consists entirely of atoms of $U$.  However, $B^<$ does not contain any atoms of $\mathbb{P}$.  If all atoms of $U$ are already atoms of $\mathbb{P}$, this implies $B^<\cap U_n=\emptyset$.  As $B$ is a band of $\mathbb{P}$, for each $u\in U_n$, we have some comparable $b\in B$.  As $u\notin B^<$, this means $u\leq b$ and hence $u\leq c$, for some $c\in C$, which is necessarily also in $U$.  This shows that $U_n$ refines $C\cap U$, which is thus a cap of $U$, i.e. $\{C\cap U:C\in\mathsf{C}\mathbb{P}\}\subseteq\mathsf{C}U$.
\end{proof}

The following result and its corollary show how to identify levels of an $\omega$-poset.

\begin{prp}\label{PosetLevels}
    The levels of a Noetherian poset $\mathbb{P}$ in which each element has finite rank are the unique antichains $(A_n)\subseteq\mathsf{A}\mathbb{P}$ covering $\mathbb{P}$ such that, for all $n\in\omega$, $A_{n+1}\setminus A_n$ refines $A_n\setminus A_{n-1}$ $($taking $A_{-1}=\emptyset)$ and $A_n$ corefines $A_{n+1}$.
\end{prp}

\begin{proof}
    If $A_{n+1}\setminus A_n$ refines $A_n\setminus A_{n-1}$ then, in particular, $A_{n+1}$ refines $A_n$ and hence $A_m$ refines $A_n$, for all $m\geq n$.  From this we can already show that $A_n\subseteq\mathbb{P}^n$, for all $n\in\omega$.  Indeed, this follows immediately from the fact that
    \begin{equation}\label{mpqn}
        A_m\ni p<q\in A_n\qquad\Rightarrow\qquad m>n.
    \end{equation}
    To see this, just note if $A_m\ni p<q\in A_n$ then $m\leq n$ would imply $A_n\leq A_m$ and hence we would have $p'\in A_m$ with $p<q\leq p'$, contradicting $A_m\in\mathsf{A}\mathbb{P}$.

    Returning to the fact that $A_{n+1}\setminus A_n$ refines $A_n\setminus A_{n-1}$, it now follows by induction that $A_n\setminus A_{n-1}\subseteq\mathsf{r}^{-1}\{n\}$, for all $n\in\omega$.  Indeed, $A_0\subseteq\mathbb{P}^0=\mathsf{r}^{-1}\{0\}$ is immediate from what we just showed.  And if every $p\in A_{n+1}\setminus A_n$ is (strictly) below some $q\in A_n\setminus A_{n-1}\subseteq\mathsf{r}^{-1}\{n\}$ then $n+1\geq\mathsf{r}(p)>\mathsf{r}(q)=n$ and hence $\mathsf{r}(p)=n+1$, showing that $A_{n+1}\setminus A_n\subseteq\mathsf{r}^{-1}\{n+1\}$.  If the $(A_n)$ cover $\mathbb{P}$ then so do the sets $(A_n\setminus A_{n-1})$ and hence the inclusion must actually be an equality, i.e. for all $n\in\omega$,
    \[\mathsf{r}^{-1}\{n\}=A_n\setminus A_{n-1}.\]
    For all $n\in\omega$, it follows that $\mathbb{P}^n=\bigcup_{k\leq n}A_k$ and hence $A_n\subseteq\mathbb{P}_n$, by \eqref{mpqn}.

    If $A_n$ also corefines $A_{n+1}$, for all $n\in\omega$, then it again follows by induction that $A_n=\mathbb{P}_n$.  Indeed, we already know $\mathbb{P}_0=\mathbb{P}^0=A_0$.  And if $\mathbb{P}^n\geq A_n\geq A_{n+1}$ then $A_{n+1}$ must contain all atoms of $\mathbb{P}^n\cup A_{n+1}=\mathbb{P}^{n+1}$, i.e. $\mathbb{P}_{n+1}\subseteq A_{n+1}\subseteq\mathbb{P}_{n+1}$.
\end{proof}

Let us call an $\omega$-poset $\mathbb{P}$ \emph{weakly graded} if consecutive levels share only atoms of $\mathbb{P}$.  Equivalently, this is saying that every non-atomic $p\in\mathbb{P}$ has a lower bound $q$ with $\mathsf{r}(q)=\mathsf{r}(p)+1$.  Moreover, we immediately see that the following are equivalent.
\begin{enumerate}
    \item $\mathbb{P}$ is atomless and weakly graded.
    \item Every $p\in\mathbb{P}$ has a lower bound $q$ with $\mathsf{r}(q)=\mathsf{r}(p)+1$.
    \item The levels of $\mathbb{P}$ are disjoint.
    \item $\mathbb{P}_n=\mathsf{r}^{-1}\{n\}$, for all $n\in\omega$.
\end{enumerate}
\autoref{PosetLevels} has the following corollary for weakly graded $\omega$-posets.

\begin{cor}\label{AtomicPosetLevels}
    If $\mathbb{P}$ is a poset covered by finite antichains $(A_n)_{n\in\omega}\subseteq\mathsf{A}\mathbb{P}$ such that $A_{n+1}$ refines $A_n$, $A_n$ corefines $A_{n+1}$ and $A_n\cap A_{n+1}$ contains only atoms of $\mathbb{P}$, for all $n\in\omega$, then $\mathbb{P}$ is a weakly graded $\omega$-poset with levels $\mathbb{P}_n=A_n$, for all $n\in\omega$.
\end{cor}

\begin{proof}
    As above, we obtain \eqref{mpqn} from the fact that each $A_{n+1}$ refines $A_n$, showing that $\mathbb{P}$ is a Noetherian poset in which each element has finite rank.  To show that $\mathbb{P}_n=A_n$ and hence that $\mathbb{P}$ is an $\omega$-poset, it thus suffices to show that $A_{n+1}\setminus A_n$ refines $A_n\setminus A_{n-1}$, for all $n\in\omega$.  But if $A_{n+1}$ refines $A_n$ then, in particular, for every $p\in A_{n+1}\setminus A_n$, we have some $q\in A_n$ with $p<q$.  If $A_n\cap A_{n-1}$ contains only atoms of $\mathbb{P}$ then this implies that $q\in A_n\setminus A_{n-1}$ so we are done.
\end{proof}

\subsection{Level-Injectivity}

Here we look at order properties related to minimal caps.  First let us denote the order relation between levels $m$ and $n$ of an $\omega$-poset $\mathbb{P}$ by
\[{\leq^m_n}={\leq}\cap(\mathbb{P}_n\times\mathbb{P}_m).\]

\begin{prp}
    For any $\omega$-poset $\mathbb{P}$, the following are equivalent.
    \begin{enumerate}
        \item\label{MinimalCaps} Each level $\mathbb{P}_n$ is a minimal cap.
        \item\label{InjectiveLevels} $\leq^m_n$ is injective whenever $m\leq n$.
        \item\label{CofinalInjectivity} $\{n:{\leq^m_n}\text{ is injective}\}$ is cofinal in $\omega$, for each $m\in\omega$.
    \end{enumerate}
\end{prp}

\begin{proof}\
    \begin{itemize}
        \item[\ref{MinimalCaps}$\Rightarrow$\ref{InjectiveLevels}] If $\leq^m_n$ fails to be injective for some $m\leq n$ then we have some $p\in\mathbb{P}_m$ such that $q^\leq\cap\mathbb{P}_m\neq\{p\}$, for all $q\in\mathbb{P}_n$.  But then $\mathbb{P}_n$ refines $\mathbb{P}_m\setminus\{p\}$ and hence $\mathbb{P}_m\setminus\{p\}$ is a cap, showing that $\mathbb{P}_m$ is not a minimal cap.

        \item[\ref{InjectiveLevels}$\Rightarrow$\ref{CofinalInjectivity}] Immediate.

        \item[\ref{CofinalInjectivity}$\Rightarrow$\ref{MinimalCaps}] If $\mathbb{P}_m$ is not a minimal cap then it has a proper subcap $C\subseteq\mathbb{P}_m$, which is necessarily refined by $\mathbb{P}_n$, for some $n\geq m$, by \autoref{CapsRefineLevels}.  But then $\mathbb{P}_k\leq C$ and hence $\leq^m_k$ is not injective, for all $k\geq n$, showing that $\{n:{\leq^m_n}\text{ is injective}\}$ is not cofinal in $\omega$. \qedhere
    \end{itemize}
\end{proof}

Accordingly, let us call an $\omega$-poset $\mathbb{P}$ satisfying any/all of the above conditions \emph{level-injective}.  When $\mathbb{P}$ is atomless, we could also replace $\leq^m_n$ above with ${<^m_n}={\leq}\cap(\mathbb{P}_n\times\mathbb{P}_m)$, when $m<n$, which is a consequence of the following.

\begin{prp}
    Every level-injective $\omega$-poset $\mathbb{P}$ is weakly graded.
\end{prp}

\begin{proof}
    If $\mathbb{P}$ is an $\omega$-poset that is not weakly graded then we have some $p\in\mathbb{P}_n$, where $n>\mathsf{r}(p)$.  Choosing $n$ maximal with this property, it follows that $p\notin\mathbb{P}_{n+1}$, even though all the lower bounds of $p$ in $\mathbb{P}_{n+1}$ have rank $n+1$ and are thus below some element of rank $n$, necessarily different from $p$.  Thus $\mathbb{P}_{n+1}$ refines $\mathbb{P}_n\setminus\{p\}$, showing that $\mathbb{P}_n$ is not a minimal cap and hence $\mathbb{P}$ is not level-injective.
\end{proof}

With a little extra care, we can also choose the cap-basis in \autoref{CoversAreCaps} to be a level-injective $\omega$-poset with levels among some prescribed family of covers.  

\begin{prp}\label{WeaklyGradedCapBasis}
Any countable family $\mathcal{C}$ of minimal open covers of a compact $\mathsf{T}_1$ space $X$ that is coinitial $($w.r.t. refinement$)$ among all covers of $X$ has a subfamily that forms the levels of a level-injective $\omega$-cap-basis.
\end{prp}

\begin{proof}
    Like in the proof of \autoref{CoversAreCaps}, let $(C_n)_{n\in\omega}$ enumerate $\mathcal{C}$ and define $(n_k)_{k\in\omega}$ as follows.  Let $n_0$ be arbitrary. If $n_k$ has been defined then note that, for any $x\in X$, we have $p\in C_{n_k}$ and $q\in C_k$ with $x\in p\cap q$.  If $p\cap q=\{x\}$ then set $b_x=\{x\}$.  Otherwise we may take away a point of $p\cap q$ (as $X$ is $\mathsf{T}_1$) to obtain open $b_x$ with $x\in b_x\subsetneqq p\cap q$.  As $C_{n_k}$ is minimal, this implies that no subset of $b_x$ lies in $C_{n_k}$.  Now $(b_x)_{x\in X}$ is an open cover of $X$ which must then have a refinement $C_{n_{k+1}}$, for some $n_{k+1}$.  Thus $C_{n_{k+1}}$ refines both $C_{n_k}$ and $C_k$, with the additional property that $C_{n_{k+1}}\cap C_{n_k}$ consists only of singletons.  As in the proof of \autoref{CoversAreCaps}, this implies that $\mathbb{P}=\bigcup_{k\in\omega}C_{n_k}$ is a cap-basis for $X$ and that each $C_{n_k}$ also corefines $C_{n_{k+1}}$.  Moreover, each $C_{n_k}$ is a minimal cover and hence a minimal cap in $\mathbb{P}$.  Thus $\mathbb{P}$ is a level-injective $\omega$-poset with levels $\mathbb{P}_k=C_{n_k}$, by \autoref{AtomicPosetLevels}.
\end{proof}

If we want the levels of an $\omega$-poset to determine not just the caps but even the bands then we need a slight strengthening of level-injectivity.  To describe this, let us introduce some more terminology and notation.

Take a poset $(\mathbb{P},\leq)$. The intervals defined by any $p,q\in\mathbb{P}$ will be denoted by
\begin{align*}
    (p,q)&=p^<\cap q^>=\{r\in\mathbb{P}:p<r<q\},\\
    [p,q]&=p^\leq\cap q^\geq=\{r\in\mathbb{P}:p\leq r\leq q\}.
\end{align*}
We call $p$ a \emph{predecessor} of $q$ (and $q$ a \emph{successor} of $p$) if $p$ is a maximal element strictly below $q$. The resulting predecessor relation will be denoted by $\lessdot$, i.e.
\[p\lessdot q\qquad\Leftrightarrow\qquad p<q\text{ and }(p,q)=\emptyset\qquad\Leftrightarrow\qquad p\neq q\text{ and }[p,q]=\{p,q\}.\]

\begin{dfn}
We call $\mathbb{P}$ \emph{predetermined} if, for all $p\in\mathbb{P}$,
\[\tag{Predetermined}p^>\neq\emptyset\qquad\Rightarrow\qquad\exists q<p\ (q^<\subseteq p^\leq).\]
\end{dfn}

Equivalently, $q<p$ and $q^<\subseteq p^\leq$ could be written just as $q^<=p^\leq$.  Also note this implies $(q,p)=\emptyset$ and hence $q\lessdot p$, i.e. $q$ is necessarily a predecessor of $p$.  In other words, $\mathbb{P}$ is predetermined precisely when every non-atomic element of $\mathbb{P}$ has a `predecessor which determines its upper bounds'.

Predetermined $\omega$-posets can also be characterised as follows.

\begin{prp}
    If $\mathbb{P}$ is an $\omega$-poset then the following are equivalent.
    \begin{enumerate}
        \item\label{PreItem} $\mathbb{P}$ is predetermined.
        \item\label{BandItem} Every non-minimal $p\in\mathbb{P}$ is a band of $q^<$, for some $q\in\mathbb{P}$.
        \item\label{LevelItem} For every $p\in\mathbb{P}$ and $n\geq\mathsf{r}(p)$, we have $q\in\mathbb{P}_n$ with $q^\leq=[q,p]\cup p^<$.
        \item\label{CapItem} Every finite cap is a band, i.e. $\mathsf{B}\mathbb{P}=\mathsf{C}\mathbb{P}\cap\mathsf{F}\mathbb{P}$.
    \end{enumerate}
\end{prp}

\begin{proof}\
    \begin{itemize}
        \item[\ref{PreItem}$\Rightarrow$\ref{BandItem}]  If $\mathbb{P}$ is predetermined then, for any non-minimal $p\in\mathbb{P}$, we have $q\lessdot p$ with $q^<=p^\leq$.  In particular, $p$ is a band of $q^<$.

        \item[\ref{BandItem}$\Rightarrow$\ref{PreItem}]  Say every non-minimal $p\in\mathbb{P}$ is a band of $q^<$, for some $q\in\mathbb{P}$.  If $\mathbb{P}$ is also Noetherian then $(q,p)$ has a maximal element $q'$, necessarily with $q'^<=p^\leq$.  This shows that $\mathbb{P}$ is predetermined.

        \item[\ref{PreItem}$\Rightarrow$\ref{LevelItem}]  If $\mathbb{P}$ is a predetermined $\omega$-poset then, for any $p\in\mathbb{P}$ and $n\geq\mathsf{r}(p)$, we can recursively define $q_k\in\mathbb{P}_k$ with $q_k^\leq=[q_k,p]\cup p^<$ as follows, for $k\geq n$.  First set $q_n=p$.  Now assume $q_k$ has been defined.  If $q_k$ is an atom then we may simply set $q_{k+1}=q_k$.  Otherwise, we have $q_{k+1}$ with $q_{k+1}^<=q_k^\leq$ and hence $\mathsf{r}(q_{k+1})=\mathsf{r}(q_k)+1$.  Thus $q_{k+1}\in\mathbb{P}_{k+1}$, as $q_k\in\mathbb{P}_k$, and
        \[q_{k+1}^\leq=\{q_{k+1}\}\cup q_k^\leq=\{q_{k+1}\}\cup[q_k,p]\cup p^<=[q_{k+1},p]\cup p^<.\]

        \item[\ref{LevelItem}$\Rightarrow$\ref{CapItem}]  Assume \ref{LevelItem} holds and take some finite cap $C\subseteq\mathbb{P}$.  By \autoref{CapsRefineLevels}, $\mathbb{P}_n\leq C$, for some $n\in\omega$, and hence $\mathbb{P}\setminus\mathbb{P}^n\leq\mathbb{P}_n\leq C$ too.  On the other hand, if $p\in\mathbb{P}^n\setminus\mathbb{P}_n$ then \ref{LevelItem} yields $q\in\mathbb{P}_n$ with $q^\leq=[q,p]\cup p^<$.  As $\mathbb{P}_n\leq C$, we have some $c\in C\cap q^\leq\subseteq p^\leq\cup p^\geq$ and hence $p\in C^\leq\cup C^\geq$.  This shows that $C$ is a band.

        \item[\ref{CapItem}$\Rightarrow$\ref{PreItem}]  Assume $\mathbb{P}$ is an $\omega$-poset which is not predetermined, so we have some non-atomic $p\in\mathbb{P}$ with $q^<\nsubseteq p^\leq$, for all $q<p$.  Then we can take minimal $n\in\omega$ such that $\mathbb{P}_n\cap p^>\neq\emptyset$.  For every $q\in\mathbb{P}_n\cap p^>$, pick $q'\in q^<\setminus p^\leq$ and note that $q'\notin p^\geq$ too, by the minimality of $n$.  Thus $C=(\mathbb{P}_n\setminus p^>)\cup\{q':q\in\mathbb{P}_n\setminus p^\leq\}$ is a finite cap, as $\mathbb{P}_n\leq C$, but not a band, as $p\notin C^\leq\cup C^\geq$.  \qedhere
    \end{itemize}
\end{proof}

\begin{cor}\label{LevelMinimalCap}
    Every predetermined $\omega$-poset is level-injective.
\end{cor}

\begin{proof}
    Every level of an $\omega$-poset $\mathbb{P}$ is a minimal band, being both a band and an antichain.  If $\mathbb{P}$ is also predetermined then any smaller cap would also be a band, by \ref{CapItem} above, and hence each level is even minimal among all caps.
\end{proof}

\begin{cor}\label{BandCapBases}
    If $X$ is a $\mathsf{T}_1$ compactum and $\mathbb{P}\subseteq\mathsf{P}X$ then
    \[\mathbb{P}\text{ is an $\omega$-band-basis}\qquad\Leftrightarrow\qquad\mathbb{P}\text{ is a predetermined $\omega$-cap-basis}.\]
\end{cor}

\begin{proof}
    If $\mathbb{P}$ is an $\omega$-cap-basis of $X$ then, by \ref{CapItem} above, $\mathbb{P}$ is predetermined if and only if every finite cover is a band, i.e. if and only if $\mathbb{P}$ is actually a band-basis.
\end{proof}

Using this, we can improve on \autoref{CoversAreCaps} by showing that every $\mathsf{T}_1$ compactum even has an $\omega$-band-basis (unlike the improvement in \autoref{WeaklyGradedCapBasis}, however, we can not specify the potential levels of the $\omega$-band-basis in advance).

First we need the following preliminary result.

\begin{lemma} \label{PredeterminedLemma}
    For any basis $B$ and finite open family $C$ of a $\mathsf{T}_1$ compactum $X$, there is a minimal cover $D\subseteq B$ of $X$ and $(x_d)_{d\in D}\subseteq X$ such that, for all $d\in D$,
    \begin{equation}\label{dbigcap}
        d=\bigcap\{e\in C\cup D:x_d\in e\}\qquad\text{and}\qquad d\neq\{x_d\}\ \Rightarrow\ d\notin C.
    \end{equation}
\end{lemma}

\begin{proof}
    For all $F\subseteq C$, let us define
    \begin{align*}
        L_F&=\{x \in X: x^\in\subseteq F\}=X \setminus \bigcup(C \setminus F).\\
        X_F&=\{x \in X: x^\in = F\}=L_F\setminus\bigcup_{G\subsetneqq F}L_G.
    \end{align*}
    Note each $L_F$ is closed and the $(X_F)_{F\subseteq C}$ are disjoint subsets covering $X$.  We will recursively define further closed subsets $K_F \subseteq X_F$ with minimal covers $D_F\subseteq B$ of $K_F$ such that $\bigcup D_F \subseteq \bigcap F$, $L_F\subseteq\bigcup\bigcup_{G\subseteq F}D_G$ and
    \begin{equation}\label{GsubsetneqqF}
        G\subsetneqq F\qquad\Rightarrow\qquad K_F\cap\bigcup D_G=\emptyset.
    \end{equation}
    (Incidentally, it is quite possible for $K_F$ to be empty for many $F\subseteq C$, but this just means that $D_F$ will also be empty.)  If $G\nsubseteq F$ then, taking any $g\in G\setminus F$, we see that $K_F\cap\bigcup D_G\subseteq L_F\cap g=\emptyset$ too so \eqref{GsubsetneqqF} can automatically be strengthened to
    \begin{equation*}%\label{GnsubseteqF}
        F\neq G\qquad\Rightarrow\qquad K_F\cap\bigcup D_G=\emptyset.
    \end{equation*}
    
    Once we have constructed these sets we see that, whenever $d\in D_F$, minimality means we have $x_d\in d\cap K_F\setminus\bigcup(D_F\setminus\{d\})$.  As $K_F\subseteq X\setminus\bigcup_{G\neq F}D_G$, it follows that $x_d\in d\setminus\bigcup(D\setminus\{d\})$, where $D=\bigcup_{G\subseteq C}D_G$.  As $X=L_C\subseteq\bigcup\bigcup_{G\subseteq C}D_G=\bigcup D$, this shows that $D$ is a minimal cover of $X$ and that $x_d\in e\in D$ implies $e=d$.  Moreover, $x_d\in e\in C$ implies that $e\in F$, as $x_d\in K_F\subseteq X_F\subseteq L_F$, and hence $d\subseteq\bigcup D_F\subseteq\bigcap F\subseteq e$.  This proves the first part of \eqref{dbigcap}.
        
    To perform the recursive construction, first let $D_\emptyset\subseteq B$ be any minimal cover of $K_\emptyset=L_\emptyset=X_\emptyset=X\setminus\bigcup C$.  Once $K_G$ and $D_G$ have been defined, for $G\subsetneqq F$, we set $K_F=L_F \setminus\bigcup\bigcup_{G\subsetneqq F}D_G\subseteq L_F\setminus\bigcup\bigcup_{G\subsetneqq F}L_G=X_F\subseteq\bigcap F$.  By compactness, we then have a minimal cover $D_F\subseteq B$ of $K_F$ with
    \[\bigcup D_F\ \subseteq\ \bigcap F\ \subseteq\ X\setminus\bigcup_{G\subsetneqq F}X_G\ \subseteq\ X\setminus\bigcup_{G\subsetneqq F}K_G.\]
    As $X$ is $\mathsf{T}_1$, we can further ensure that $d\subsetneqq\bigcap F$ and hence $d\notin C$, for each $d\in D_F$, unless $K_F=\bigcap F=\{x\}$, for some $x\in X$, in which case the only option is $D_F=\{\{x\}\}$.  This ensures that the second part of \eqref{dbigcap} also holds.  Now just note
    \[L_F\ \subseteq\ K_F\cup\bigcup\bigcup_{G\subsetneqq F}D_G\ \subseteq\ \bigcup D_F\cup\bigcup\bigcup_{G\subsetneqq F}D_G\ =\ \bigcup\bigcup_{G\subseteq F}D_G\]
    so the recursive construction may continue.
\end{proof}

\begin{thm}\label{PredeterminedSubBasis}
Any countable basis of a $\mathsf{T}_1$ compactum contains an $\omega$-band-basis.
\end{thm}

\begin{proof}
    As in the proof of \autoref{CoversAreCaps}, let $(C_n)_{n\in\omega}$ enumerate all finite minimal covers of a $\mathsf{T}_1$ compactum $X$ coming from any given countable basis $B$.  Recursively define finite minimal covers $(B_n)_{n\in\omega}$ as follows.  Let $D_k = C_k \cup \bigcup_{j < k} B_j$.
    By \autoref{PredeterminedLemma} we have a minimal cover $B_k \subseteq B$, such that $B_k \cap D_k$ contains only singletons, as well as $(x_b)_{b\in B_k}\subseteq X$ such that $b=\bigcap\{e\in B_k\cup D_k:x_b\in e\}$, for all $b\in B_k$.
    In other words, $x_b\in b\subseteq e$, for any $e\in B_k\cup D_k$ with $x_b\in e$, so $B_k$ refines $C_k$ and $B_j$, for all $j<k$.  As in the proof of \autoref{CoversAreCaps}, this implies $\mathbb{P}=\bigcup_{k\in\omega}B_k$ is a cap-basis and each $B_k$ also corefines $B_{k+1}$.  By construction, $B_{k+1}\cap B_k$ contains only singletons, which are atoms in $\mathbb{P}$.  Thus $\mathbb{P}$ is an $\omega$-poset with levels $\mathbb{P}_k=B_k$, by \autoref{AtomicPosetLevels}.  Also, for every $b\in B_k$, we have $c\in B_{k+1}$ with $x_b\in c$.  Then $c < a$ implies $x_b\in a \in B_j$, for some $j\leq k$, and hence $b\leq a$.  This shows that $c^< \subseteq b^\leq$ and hence $c^<=b^\leq$, as long as $b$ is not an atom.  Thus $\mathbb{P}$ is also predetermined and hence an $\omega$-band-basis, by \autoref{BandCapBases}.
\end{proof}

\subsection{Graded Posets}

We call a Noetherian poset $\mathbb{P}$ \emph{graded} if the rank function maps intervals to intervals, i.e. for all $p,q\in\mathbb{P}$,
\[p<q\qquad\Rightarrow\qquad\mathsf{r}[(p,q)]=(\mathsf{r}(q),\mathsf{r}(p)).\]
In particular, this means the rank function turns predecessors into successors, i.e.
\[p\lessdot q\qquad\Rightarrow\qquad\mathsf{r}(p)=\mathsf{r}(q)+1.\]
In fact, if every element of $\mathbb{P}$ has finite rank then $\mathbb{P}$ is graded precisely when this happens.  This also makes it clear that every graded $\omega$-poset is indeed weakly graded.

\begin{rmk}
Hasse diagrams of atomless graded $\omega$-posets can thus be viewed as Bratteli diagrams (see \cite[Definition 3.1]{Putnam2018}) where the levels $(\mathbb{P}_n)$ form the vertex sets and the edges come from the predecessor relation $\lessdot$.  Indeed, any Bratteli diagram with at most one edge between distinct vertices arises as the Hasse diagram of some atomless graded $\omega$-poset.  Whether one chooses to work with diagrams or posets is thus a matter of taste, although the diagram picture will be particularly instructive in future work when we construct graded $\omega$-posets associated to interesting compacta (e.g. see \autoref{ArcExample}).
\end{rmk}

Graded $\omega$-posets are completely determined by the order relation between consecutive levels.  As such, they are the strongest interpretation of what it means for a poset to be built from a sequence of finite levels.  Naturally, we would like to construct bases of this special form.  First we begin with some simple observations.

\begin{prp}\label{GradedSummary}
Let $\PP$ be a graded $\omega$-poset.
\begin{enumerate}
    \item $\PP$ is level-injective if and only if $\PP$ is predetermined.
    \item The levels are pairwise disjoint if and only if $\PP$ is atomless.
    \item  If $\PP$ is a basis of a $\mathsf T_1$ space, then every level $\PP_n$ consolidates $\PP_{n + 1}$.
\end{enumerate}
\end{prp}
\begin{proof}
    The `if' part of (1) follows from \autoref{LevelMinimalCap}.  Conversely, suppose that $\PP$ is not predetermined  so we have non-atomic $p \in \PP_n$ such that $q^< \setminus p^\leq\neq\emptyset$, for every $q \in \PP_{n + 1} \cap p^>$.  Since $\PP$ is graded, we then have $r\in\PP_n\cap q^<\setminus p^\leq$.  It follows that $\PP_n \setminus \{p\}$ is refined by $\PP_{n + 1}$, and so $\PP_n$ is not a minimal cap.
    
    As $\mathbb{P}$ is graded and hence weakly graded, (2) is immediate.
    
    To prove (3), take $b \in \PP_n$ and let $B = \{c \in \PP_{n + 1}: c \subseteq b\}$.  If we have $x \in b \setminus \bigcup B$, then we have $u \in \PP$ such that $x \in u \subseteq b$ as $\PP$ is a basis.  We may further assume that $u \not\subseteq c$ for every $c \in B$ as the space is $\mathsf T_1$.  Hence, $u \in \PP_m$ for some $m > n$.  But since $\PP$ is graded, we get $u \subseteq v \subseteq b$ for some $v \in \PP_{n + 1}$.  Hence, $x \in v \in B$, which is a contradiction.
\end{proof}

To ensure the cap-bases in \autoref{CoversAreCaps} are graded, we need the following.

\begin{lemma}\label{GradedLemma}
 Let $(C_n)$ be a sequence of minimal covers of a set $X$ with each $C_n$ consolidating $C_{n+1}$ and $C_{n+1}\cap C_n$ only containing singletons $\{\{x\}:x\in X\}$.  Further let $\mathbb P=\bigcup_{n\in\omega}C_n$, considered as a poset with ${\leq}={\subseteq}$. Then
 \begin{enumerate}
     \item\label{GradedLemma1}  $\mathbb{P}$ is a predetermined graded poset with $n^\mathrm{th}$ level $\mathbb{P}_n=C_n$, and
     \item\label{GradedLemma2} if $\mathbb{P}$ is a basis for a compact topology then $\mathbb{P}$ is also an $\omega$-cap-basis.
\end{enumerate}
\end{lemma}

\begin{proof}\
\begin{enumerate}
\item  First note $C_n^\leq=\bigcup_{k\leq n}C_k$, for all $n\in\omega$. Indeed, if we had $c<d\in C_k$, for some $k>n$ then, as $C_k\leq C_n$, we would have some $c'\in C_n$ with $c<d\leq c'$, contradicting the minimality of $C_n$. In particular, as $C_0=C_0^\leq$ is a minimal cover of $X$, it consists entirely of maximal elements of $\mathbb{P}$, i.e. elements of rank $0$.
    
We claim that, for all $n\in\omega$,
\[(C_{n+1}\setminus C_n)^\lessdot\quad\subseteq\quad C_n\setminus\{\{x\}:x\in X\}\quad\subseteq\quad\mathsf{r}^{-1}\{n\}.\]
For the first inclusion, take $c\in(C_{n+1}\setminus C_n)^\lessdot$, which means we have $d\in C_{n+1}\setminus C_n$ with $d\lessdot c$. In particular, $c\in C_{n+1}^\lessdot$ so we must have $m\leq n$ with $c\in C_m$. By minimality, we can choose some $x\in d\setminus\bigcup(C_{n+1}\setminus\{d\})$. As each cover consolidates the next, we have $c_m\geq\ldots\geq c_{n+1}$ with $c_m=c$ and $x\in c_k\in C_k$, for all $k$ between $m$ and $n+1$. By our choice of $x$, we must have $c_{n+1}=d$ and hence $c_n>d$ because $d\in C_{n+1}\setminus C_n$. In particular, $c_n$ is not a singleton so other inequalities must be strict too, i.e. $c=c_m>\ldots>c_{n+1}=d$. The only way we could have $d\lessdot c$ then is if $m=n$. This proves the first inclusion. The second now follows by induction -- the $n=0$ case was observed above, while all successors of elements of $C_{n+1}\setminus\{\{x\}:x\in X\}\subseteq C_{n+1}\setminus C_n$ must lie in $C_n\setminus\{\{x\}:x\in X\}$ and hence have rank $n$, so all elements of $C_{n+1}\setminus\{\{x\}:x\in X\}$ have rank $n+1$.

In particular, each $p\in\mathbb{P}$ has finite rank and all its successors $p^\lessdot$ have the same rank, proving that $\mathbb{P}$ is graded. Also note that singletons persist as soon as they appear, i.e. if $\{x\}\in C_n$ then $\{x\}\in C_{n+1}$, again because each cover consolidates the next. Thus each $C_n$ consists precisely of the elements of rank $n$ together with singletons (and hence minimal elements of $\mathbb{P}$) of smaller rank, i.e. $C_n=\mathbb{P}_n$. Finally, for any $p\in\mathbb{P}$ we can again take $x\in p\setminus\bigcup(C_{\mathsf{r}(p)}\setminus p)$. If $p$ is not minimal, we can then take $q\in C_{\mathsf{r}(p)+1}$ with $x\in q<p$ and show that $q^<=p^\leq$, which means that $\mathbb{P}$ is predetermined.

\item Now assume $\mathbb{P}$ is also a basis for a compact topology.  In particular, each minimal cover $C_n$ must be finite and hence $\mathbb{P}$ is an $\omega$-poset.  We claim that, moreover, every cover $C\subseteq\mathbb{P}$ must be refined by some level $C_n$.  Indeed, by compactness, we can replace $C$ with a finite subset if necessary.  As each level is a consolidation of the next, we can further replace each non-atomic element of $C$ having smallest rank with elements in a level below.  Continuing in this manner, we eventually obtain a new cover $D$ refining the original cover $C$ whose elements are all contained in a single level $C_n$.  As $C_n$ is a minimal cover, $D$ must then be the entirety of $C_n$, proving the claim.  As levels are caps, this shows that $\mathbb{P}$ is a cap-basis.\qedhere
\end{enumerate}

\end{proof}

Note for $\mathbb{P}$ to be graded here, not just Noetherian, it is crucial that each cover is not only refined by the next cover but also consolidates it, as the following shows.

\begin{xpl}
Let $X=[0,1]$ and define
\begin{align*}
C_1&=\{[0,\tfrac{3}{4}),(\tfrac{1}{4},1]\},\\
C_2&=\{[0,\tfrac{2}{3}),(\tfrac{1}{3},1]\},\\
C_3&=\{[0,\tfrac{1}{2}),(\tfrac{1}{4},\tfrac{2}{3}),(\tfrac{1}{3},\tfrac{3}{4}),(\tfrac{1}{2},1]\}.
\end{align*}
The Hasse diagram of the resulting poset $(C_1\cup C_2\cup C_3,\subseteq)$ looks like this:
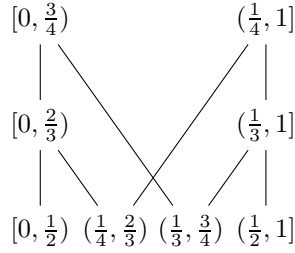
\begin{center}
\begin{tikzpicture}[y={(0, 2em)}]
		\node (00) at (0,0) [align=center]{$[0,\tfrac{3}{4})$};
		\node (30) at (3,0) [align=center]{$(\tfrac{1}{4},1]$};
		\node (0-2) at (0,-2) [align=center]{$[0,\tfrac{2}{3})$};
		\node (3-2) at (3,-2) [align=center]{$(\tfrac{1}{3},1]$};
		\node (0-4) at (0,-4) [align=center]{$[0,\tfrac{1}{2})$};
		\node (1-4) at (1,-4) [align=center]{$(\tfrac{1}{4},\tfrac{2}{3})$};
		\node (2-4) at (2,-4) [align=center]{$(\tfrac{1}{3},\tfrac{3}{4})$};
		\node (3-4) at (3,-4) [align=center]{$(\tfrac{1}{2},1]$};
		\draw (00)--(0-2)--(0-4);
		\draw (30)--(3-2)--(3-4);
		\draw (0-2)--(1-4);
		\draw (30)--(1-4);
		\draw (3-2)--(2-4);
		\draw (00)--(2-4);
\end{tikzpicture}
\end{center}
Note that $C_3$ refines $C_2$ which in turn refines $C_1$. However, $C_3\ni(\tfrac{1}{4},\tfrac{2}{3})\subseteq(\tfrac{1}{4},1]\in C_1$ even though there is no element of $C_2$ in between, i.e. $C_1\cup C_2\cup C_3$ is not graded.
\end{xpl}

Using \autoref{GradedLemma} we can construct graded $\omega$-band-bases.

\begin{thm}\label{PredeterminedGradedBasis}
Every second countable $\mathsf{T}_1$ compactum has a graded $\omega$-band-basis.
\end{thm}

\begin{proof}
We modify the proof of \autoref{CoversAreCaps} so we can use \autoref{GradedLemma}. To start with, again take any countable basis $B$ for $\mathsf{T}_1$ compactum $X$ and let $(B_n)_{n\in\omega}$ enumerate all finite covers of $X$ from $B$. Recursively define another sequence of finite open covers $(C_n)$ as follows. Let $C_0=\{X\}$. If $C_n$ has been defined then, for each $x\in X$, let
\[d_x=\bigcap\{a\in B_n\cup C_n:x\in a\}.\]
As $B_n$ and $C_n$ are finite, so is $D=\{d_x:x\in X\}$. For each $d\in D$, choose some $x_d$ such that $d=d_{x_d}$ and denote the set of all the other chosen points by
\[f_d=\bigcup_{e\in D\setminus\{d\}}\{x_e\}.\]
We then have a minimal open cover refining both $B_n$ and $C_n$ given by
\[E=\{d\setminus f_d:d\in D\}.\]
Also note that if $y\in c\in C_n$ then $y\neq x_d$ for any $d\neq d_y$ (because $y=x_d$ implies $d_y=d_{x_d}=d$) so $y\in d_y\setminus f_{d_y}\subseteq c$. This shows that $c=\bigcup(E\cap c^\geq)$, for all $c\in C_n$, i.e. $C_n$ consolidates $E$. At this stage it is possible that there could be some non-singleton $c\in C_n\cap E$. However, this can only happen when $c$ is contained in some $b\in B_n$ and disjoint from all other subsets in $(B_n\setminus\{b\})\cup C_n$ and hence $E$ -- otherwise we would have some $d\in D$ with $d\subsetneqq c$ and so certainly $d\setminus f_d\subsetneqq c$, while all other elements of $E$ would avoid $x_d\in d\subseteq c$. For any non-singleton $c\in C_n\cap E$, we can thus pick arbitrary distinct $y_c,z_c\in c$ and replace $c$ with $c\setminus\{y_c\}$ and $c\setminus\{z_c\}$ without destroying the minimality of $E$. In other words, to ensure consecutive covers can only contain singletons, we define $C_{n+1}$ by
\[C_{n+1}=E\setminus C_n\cup\bigcup_{c\in E\cap C_n}\{c\setminus\{y_c\},c\setminus\{z_c\}\}.\]

This completes the recursion and the poset $\mathbb{P}=\bigcup_{n\in\omega}C_n$ is then a predetermined graded $\omega$-poset, by \autoref{GradedLemma}.  As $X$ is compact, every cover of $X$ from $\mathbb{P}$ is refined by $B_n$ and hence $C_{n+1}$, for some $n\in\omega$.  As in the proof of \autoref{CoversAreCaps}, $\mathbb{P}$ is then an $\omega$-cap-basis and hence an $\omega$-band-basis, by \autoref{BandCapBases}.
\end{proof}

Unlike \autoref{PredeterminedSubBasis}, we can not choose the graded $\omega$-band-bases above to lie within some basis given in advance.  Indeed, the following result shows that most Hausdorff compacta have bases which do not contain any graded $\omega$-basis.

As usual, we view any ordinal $\alpha$ as a topological space with respect to the interval topology, i.e. generated by subbasic sets $\beta^<$ and $\beta^>$, for all $\beta<\alpha$. 

\begin{prp}
    For any Hausdorff compactum $X$, the following are equivalent.
    \begin{enumerate}
        \item $X$ is homeomorphic to $\alpha$, for some $\alpha<\omega^2$.
        \item Every basis for $X$ contains a graded $\omega$-$($cap-$)$basis.
    \end{enumerate}
\end{prp}

\begin{proof}
    If $X=\alpha<\omega^2$ then any basis for $X$ contains a basis $\mathbb{P}$ such that each $p\in\mathbb{P}$ is either a singleton or contains a unique non-zero limit ordinal $\omega(n+1)$ such that $p\subseteq(\omega n,\omega(n+1)+1)$.
    Taking a further subset if necessary, we can ensure that the neighbourhoods of any fixed non-zero limit ordinal $\omega n$ are linearly ordered and hence $T_n=\{p\in\mathbb{P}:p\subseteq(\omega n,\omega(n+1)+1)\}$ consists of atoms together with at most one decreasing sequence.  In particular, each $T_n$ and is graded and hence $\mathbb{P}=\{\{0\}\}\cup\bigcup_{\omega n+1\in\alpha}T_n$ is a graded $\omega$-cap-basis.

    Conversely, say $X$ is not homeomorphic to any $\alpha<\omega^2$.  We can further assume that $X$ is second countable (otherwise $X$ certainly could not have any $\omega$-basis and we would be done).  We claim that the non-isolated points of $X$ have some limit point $y\in X$.  Indeed, $X=Y\cup S$, for (unique) perfect $Y$ and countable scattered $S$.  If $Y\neq\emptyset$ then just take any $y\in Y$.  If $Y=\emptyset$ then $X=S$ must be homeomorphic to some ordinal $\alpha>\omega^2$ and we can just take $y$ to be (the point identified with) $\omega^2$, which is the limit of $(\omega(n+1))_{n\in\omega}$.  This proves the claim and it follows that $y$ has a neighbourhood basis consisting of non-closed open sets -- if $O$ is a clopen neighbourhood of $y$ just take any non-isolated $z\in O\setminus\{y\}$ and note that $O\setminus\{z\}$ is still open but no longer closed.  These neighbourhoods of $y$ together with all open sets avoiding $y$ thus form a basis $B$ for $X$.  As $X$ is Hausdorff and hence regular, we can argue as in the proof of \autoref{CoversAreCaps} to obtain another basis $\mathbb{P}\subseteq B$ such that strict containment implies closed containment (just choose each $b\ni x$ there so that $\mathrm{cl}(b)\subseteq\bigcap\{c\in\bigcup_{j\leq k}C_{n_j}:x\in c\}$), i.e.
    \[p\subsetneqq q\qquad\Rightarrow\qquad\mathrm{cl}(p)\subseteq q.\]
    As each $p\in\mathbb{P}$ containing $y$ is not closed, $p$ can never be the union of a finite subset of $\mathbb{P}\setminus\{p\}$ (as $p$ would then be the union of their closures too and hence itself closed).  In particular, $\mathbb{P}$ can not contain a graded $\omega$-basis, as each level would then have to consolidate the next, by \autoref{GradedSummary}.
\end{proof}

The following summarises \autoref{WeaklyGradedCapBasis}, \autoref{PredeterminedSubBasis}, and \autoref{PredeterminedGradedBasis}.
\begin{thm}\label{BasesSummary}
Every second countable $\mathsf T_1$ compactum $X$ has an $\omega$-cap-basis $\PP$.
Moreover, we can arrange any of the following $($but not any two simultaneously$)$.
\begin{enumerate}
    \item $\PP$ is level-injective and the levels $\PP_n$ are members of a given co-initial family of minimal open covers.
    \item $\PP$ is predetermined and its elements are members of a given countable basis.
    \item $\PP$ is predetermined and graded.
\end{enumerate}
\end{thm}

\subsection{Additional Properties}\label{ss:AdditionalProps}

Before moving on, let us examine some other simple order properties possessed by all cap-bases of $\mathsf{T}_1$ spaces. Specifically, let us call a poset $\mathbb{P}$ \emph{branching} if no principal down-set $p^>$ has a singleton band, i.e.
\[\tag{Branching}p<q\qquad\Rightarrow\qquad\exists r<q\ (p\nleq r\text{ and }r\nleq p).\]
In particular, this implies no $p\in\mathbb{P}$ has a unique predecessor, so the Hasse diagram of $\mathbb{P}$ does indeed branch as much as possible. This even characterises branching posets among $\omega$-posets or, more generally, posets which only have finite intervals.

\begin{prp}\label{BranchingBasis}
Any basis of non-empty open sets of a $\mathsf{T}_1$ space is branching.
\end{prp}

\begin{proof}
Take a basis $\mathbb{P}$ of non-empty sets of a $\mathsf{T}_1$ space $X$. For any $p,q\in\mathbb{P}$ with $p<q$, we can take $x\in p$ and $y\in q\setminus p$. We then have some $r\in\mathbb{P}$ with $y\in r\subseteq q\setminus\{x\}$. Note $p\nleq r$, as $x\in p\setminus r$, and $r\nleq p$, as $y\in r\setminus p$. This shows that $\mathbb{P}$ is branching.
\end{proof}

In particular, every poset arising in \autoref{PredeterminedSubBasis} is branching. It is natural to wonder if this is the only extra restriction, i.e. does every branching predetermined $\omega$-poset arise as a cap-basis of some (necessarily compact) $\mathsf{T}_1$ space? In fact, this will even hold under a certain weaker assumption which we now describe.

First let us define the \emph{cap-order} relation $\precsim$ on $\mathsf{P}\mathbb{P}$ by
\[\tag{Cap-Order}\label{CapOrder}Q\precsim R\qquad\Leftrightarrow\qquad\forall F\subseteq\mathbb{P}\ (F\cup Q\in\mathsf{C}\mathbb{P}\ \Rightarrow\ F\cup R\in\mathsf{C}\mathbb{P}).\]
Note it suffices to consider finite $F$ here, as every cap has a finite subcap. Further note that $\precsim$ is a preorder containing refinement as a subrelation. In particular, on singletons it contains the original order $\leq$. Let us call a poset $\mathbb{P}$ \emph{cap-determined} if it actually agrees with $\leq$ on singletons, i.e. for all $p,q\in\mathbb{P}$,
\[\tag{Cap-Determined}p\precsim q\qquad\Rightarrow\qquad p\leq q.\]
More explicitly this means that, whenever $p\nleq q$, we have some $F\subseteq\mathbb{P}$ (which we can take to be finite) such that $F\cup\{p\}$ is a cap but $F\cup\{q\}$ is not.

\begin{prp}\label{CapBasesAreCapDetermined}
Every cap-basis of a $\mathsf{T}_1$ space is cap-determined.
\end{prp}

\begin{proof}
Take a cap-basis $\mathbb{P}$ of a $\mathsf{T}_1$ space $X$. Whenever $p\nleq q$, we have some $x\in p\setminus q$. As $X$ is $\mathsf{T}_1$ and $\mathbb{P}$ is a basis, we can cover $X\setminus p$ with a subcollection $F\subseteq\mathbb{P}$ whose elements all avoid $x$. Thus $F\cup\{p\}$ is a cover of $X$ and hence a cap of $\mathbb{P}$.  On the other hand, no member of $F\cup\{q\}$ contains $x$ so it can not be a cover of $X$ and is thus not a cap of $\mathbb{P}$, by \eqref{CapsAreCovers}. This shows that $\mathbb{P}$ is cap-determined.
\end{proof}

The relationship between these various notions can be summarised as follows.

\begin{prp}\label{prp:branchingcapdet}
If $\mathbb{P}$ is an $\omega$-poset then
\[\mathbb{P}\text{ is branching and predetermined}\ \ \Rightarrow\ \ \mathbb{P}\text{ is cap-determined}\ \ \Rightarrow\ \ \mathbb{P}\text{ is branching}.\]
\end{prp}

\begin{proof}
For the first implication, assume $\mathbb{P}$ is predetermined and take any $p\in\mathbb{P}$. We claim that we can recursively construct $(p_n)_{n\geq\mathsf{r}(p)}$ such that $p_n\in\mathbb{P}_n$ and $p$ is a band of $p_n^\leq$, for all $n\geq\mathsf{r}(p)$. First set $p_{\mathsf{r}(p)}=p$. Now assume $p_n$ has been already been constructed. If $p_n$ is already minimal in $\mathbb{P}$ then it must lie in all levels beyond $n$ too and we may simply set $p_{n+1}=p_n$. Otherwise, we can take $p_{n+1}\lessdot p_n$ such that $p_{n+1}^<=p_n^\leq$, as $\mathbb{P}$ is predetermined, noting that this implies $\mathsf{r}(p_{n+1})=\mathsf{r}(p_n)+1$ (otherwise we would have $q\gtrdot p_{n+1}$ with $\mathsf{r}(q)=\mathsf{r}(p_{n+1})-1>\mathsf{r}(p_n)$ so $q\ngeq p_n$, a contradiction). As $p$ is a band for $p_n^\leq=p_{n+1}^<$, it is also a band for $p_{n+1}^\leq=p_{n+1}^<\cup\{p_{n+1}\}$. This completes the recursion.

Now say that $p\nleq q$. First consider the case where $q\nleq p$ as well. Let $F=\mathbb{P}_{\mathsf{r}(p)}\setminus\{p\}$ so certainly $F\cup\{p\}$ is a cap. However, $F\cup\{q\}$ is not refined by $\mathbb{P}_n$, for any $n>\mathsf{r}(p)$, because $\mathbb{P}_n$ contains the $p_n$ constructed above, which can not be below any element of $F\cup\{q\}$, as none of these are comparable with $p$. Thus $F\cup\{q\}$ is not a cap, by \autoref{CapsRefineLevels}. On the other hand, if $q<p$ then, as long as $\mathbb{P}$ is branching, we can take $r<p$ which is incomparable with $q$. The argument just given then yields $F$ such that $F\cup\{r\}$ and and hence $F\cup\{p\}$ is a cap while $F\cup\{q\}$ is not. This shows that $\mathbb{P}$ is cap-determined.

For the second implication, assume $\mathbb{P}$ is cap-determined. So if $p<q$ then we have $F\subseteq\mathbb{P}$ such that $F\cup\{q\}$ is a cap but $F\cup\{p\}$ is not. Take any $n>\mathsf{r}(p)$ such that $\mathbb{P}_n$ refines $F\cup\{q\}$. As $F\cup\{p\}$ is not a cap, we have $r\in\mathbb{P}_n\setminus(F\cup\{p\})^\geq$. In particular, $r\nleq p$ but also $p\nleq r$, as $\mathsf{r}(p)<n\leq\mathsf{r}(r)$. Moreover, $r\nleq f$, for all $f\in F$, and hence $r\leq q$, as $\mathbb{P}_n$ refines $F\cup\{q\}$. This shows that $\mathbb{P}$ is branching.
\end{proof}

Even when $\mathbb{P}$ is not cap-determined, $B\precsim C$ is meant to signify that $B$ is covered by $C$ in a certain sense, which we will make more precise in \eqref{AprecsimB} below.  For the moment, let us just note a few further properties of $\precsim$.  Firstly, as one would expect, the caps of $\mathbb{P}$ are precisely the maximal elements with respect to $\precsim$, i.e.
\begin{equation}\label{MaximalCaps}
C\in\mathsf{C}\mathbb{P}\qquad\Leftrightarrow\qquad\mathbb{P}\precsim C.
\end{equation}
Indeed, if $C$ is a cap then $B\precsim C$, for any $B\subseteq\mathbb{P}$, as every superset of a cap is a cap (in particular, we can take $B=\mathbb{P}$).  On the other hand, if $C=C\cup\emptyset$ is a cap and $C\precsim A$ then $A=A\cup\emptyset$ is also a cap (in particular, we can take $C=\mathbb{P}$).

We also immediately see that the empty set $\emptyset$ is minimal with respect to $\precsim$, although in general there can be elements of $\mathbb{P}$ that are minimal too.  However,
\begin{equation}\label{CapDeterminedPrime}
\mathbb{P}\text{ is cap-determined}\qquad\Rightarrow\qquad\forall p\in\mathbb{P}\ (p\not\precsim\emptyset).
\end{equation}
Indeed, if $p\precsim\emptyset$ then $p\precsim q$, for all $q\in\mathbb{P}$, so if $\mathbb{P}$ is cap-determined then $p$ is a minimum of $\mathbb{P}$, i.e. $\mathbb{P}=p^\leq$.  But then $\{p\}$ itself is already a band and hence a cap, even though the empty set $\emptyset$ is never a cap, contradicting $p\precsim\emptyset$.

Lastly, we show that $\precsim$ is determined by its restriction to singletons on the left.

\begin{prp}
For any poset $\mathbb{P}$ and $B,C\subseteq\mathbb{P}$,
\begin{equation}\label{bprecC}
B\precsim C\qquad\Leftrightarrow\qquad\forall b\in B\ (b\precsim C).
\end{equation}
\end{prp}

\begin{proof}
First let us note that $\precsim$ respects pairwise unions, i.e. for all $A,B,C\subseteq\mathbb{P}$,
\begin{equation}\label{ABprecC}
A,B\precsim C\qquad\Rightarrow\qquad A\cup B\precsim C.
\end{equation}
To see this, take any $F\subseteq\mathbb{P}$ such that $A\cup B\cup F\in\mathsf{C}\mathbb{P}$.  If $A\precsim C$ then this implies that $B\cup C\cup F\in\mathsf{C}\mathbb{P}$.  If $B\precsim C$ too then this further implies that $C\cup F=C\cup C\cup F\in\mathsf{C}\mathbb{P}$.  This shows that $A\cup B\precsim C$.

Now if $B\precsim C$ then certainly $b\precsim C$, for all $b\in B$.  Conversely, if $B\not\precsim C$ then we have some $D\subseteq\mathbb{P}$ such that $B\cup D\in\mathsf{C}\mathbb{P}$ but $C\cup D\notin\mathsf{C}\mathbb{P}$.  We then have some finite $F\subseteq B$ such that $F\cup D$ is still a cap and hence $F\not\precsim C$.  If we had $f\precsim C$, for all $f\in F$, then \eqref{ABprecC} would imply $F\precsim C$, a contradiction.  Thus $f\not\precsim C$, for some $f\in F\subseteq B$, as required.
\end{proof}

We summarize implications between considered properties of $\omega$-posets in \autoref{fig:poset_properties}.
The notion of a prime poset is defined in \autoref{DefPrimePoset} in the next section.

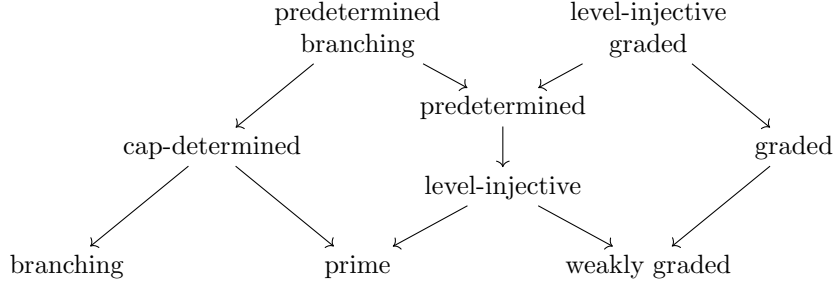
\begin{figure}[!ht]
\begin{tikzpicture}[
		x = {(5.5em, 0)},
		y = {(0, 3em)},
		text height = 1.5ex,
		text depth = 0.25ex,
		multiline/.style = {
			align = center,
			text height =,
			text width =,
		}
	]
	\node (predetermined) at (0, 0) {predetermined};
	\node (levelinjective) at (0, -1) {level-injective};
	\node (weakly graded) at (1, -2) {weakly graded};
	\node (prime) at (-1, -2) {prime};
	\node (capdetermined) at (-2, -0.5) {cap-determined};
	\node (graded) at (2, -0.5) {graded};
	\node (branching) at (-3, -2) {branching};
	\node (pb) at (-1, 1) [multiline] {predetermined \\ branching};
	\node (lg) at (1, 1) [multiline] {level-injective \\ graded};
	
	\graph{
		(pb) -> {(capdetermined), (predetermined)},
		(lg) -> {(predetermined), (graded)},
		(predetermined) -> (levelinjective) -> {(prime), (weakly graded)},
		(graded) -> (weakly graded),
		(capdetermined) -> {(branching), (prime)},
	};
\end{tikzpicture}
\caption{Implications between properties of $\omega$-posets.}
\label{fig:poset_properties}
\end{figure}

\section{The Spectrum}\label{TheSpectrum}

In this section, we construct a $\mathsf{T}_1$ compactum from any poset and relate its topological properties to the order properties of the original poset.

\begin{center}
\textbf{Throughout this section fix some poset $(\mathbb{P},\leq)$.}
\end{center}

\subsection{Selectors}\label{ss.Selectors}

The points of our desired compactum will be certain subsets of $\mathbb{P}$ which contain at least one element from every cap (i.e. `transversals' of the caps).

\begin{dfn}
We call $S\subseteq\mathbb{P}$ a \emph{selector} if it overlaps all caps, i.e.
\[\tag{Selector}C\in\mathsf{C}\mathbb{P}\qquad\Rightarrow\qquad S\cap C\neq\emptyset.\]
\end{dfn}

Equivalently, $S\subseteq\mathbb{P}$ is a selector precisely when its complement $\mathbb{P}\setminus S$ is not a cap (as being a cap and containing a cap are the same thing).

We will be particularly interested in minimal selectors.

\begin{prp}\label{SelectorsContainMinimalSelectors}
Every selector contains a minimal selector.
\end{prp}

\begin{proof}
Note that every cap $C$ contains a finite subcap -- bands are finite by definition so if $B$ is a band refining $C$ then we can simply choose a finite subset of $C$ that is still refined by $B$. For $S$ to be a selector, it thus suffices for $S$ to select elements from just the finite caps. The intersection of a chain of selectors is therefore again a selector so Kuratowski--Zorn implies every selector contains a minimal selector.
\end{proof}

The first thing to note about minimal selectors is the following.

\begin{prp}\label{MinimalSelectors}
Every minimal selector is an up-set.
\end{prp}

\begin{proof}
Take a minimal selector $S\subseteq\mathbb{P}$. Minimality means that, for every $s\in S$, we have some $C\in\mathsf{C}\mathbb{P}$ such that $S\cap C=\{s\}$ (otherwise $S\setminus\{s\}$ would be a strictly smaller selector). For any $p\geq s$, note that $(C\setminus\{s\})\cup\{p\}$ is refined by $C$ and is thus also a cap. As $S$ must also overlap this new cap, the only possibility is that $S$ also contains $p$. This shows that $S^\leq=S$, i.e. $S$ is an up-set.
\end{proof}

Moreover, to verify that an up-set is a selector, it suffices to consider a subfamily of caps $\mathcal{B}\subseteq\mathsf{C}\mathbb{P}$ that is coinitial with respect to refinement, e.g. the bands $\mathsf{B}\mathbb{P}$ or even just the levels $(\mathbb{P}_n)$ if $\mathbb{P}$ is an $\omega$-poset, thanks to \autoref{CapsRefineLevels}.

\begin{prp}\label{UpsetSelectors}
Take an up-set $U\subseteq\mathbb{P}$.  For any coinitial $\mathcal{B}\subseteq\mathsf{C}\mathbb{P}$,
\[U\text{ is a selector}\qquad\Leftrightarrow\qquad U\text{ overlaps every }B\in\mathcal{B}.\]
If $\mathbb{P}$ is an $\omega$-poset, $U$ is a selector precisely when $U$ is infinite or contains an atom.
\end{prp}

\begin{proof}
As $\mathcal{B}\subseteq\mathsf{C}\mathbb{P}$, $\Rightarrow$ is immediate.  Conversely, say $U\cap B$, for all $B\in\mathcal{B}$.  For any $C\in\mathsf{C}\mathbb{P}$, coinitiality yields $B\in\mathcal{B}$ refining $C$.  This means any $b\in B\cap U$ has an upper bound $c\in C$, which is thus also in $U$, as $U$ is an up-set.  Thus $U$ is a selector.

Next note that if $a\in\mathbb{P}$ is an atom then $a^\leq$ is a selector.  Indeed, for any band $B\in\mathsf{B}\mathbb{P}$, the minimality of $a$ implies $a\in B^\geq$ and hence $a^\leq\cap B\neq\emptyset$.  As $a^\leq$ is up-set and bands are coinitial in $\mathsf{C}\mathbb{P}$, we are done.

It follows that if $U$ contains an atom then $U$ is a selector.  Now assume $\mathbb{P}$ is an $\omega$-poset.  If $U$ is infinite then $U$ contains elements of arbitrary rank.  In particular, $U$ overlaps all levels of $\mathbb{P}$, which are coinitial by \autoref{CapsRefineLevels}, showing that $U$ is again a selector.  Conversely, if $U$ is finite and contains no atoms then we have a level of $\mathbb{P}$ which is disjoint from $U$, showing $U$ is not a selector.
\end{proof}

\subsection{Spectra}\label{ss.Spectrum}

As alluded to above, minimal selectors will form the points of the desired compactum $\mathsf{S}\mathbb{P}$ that we are about to define.  While the definition of $\mathsf{S}\mathbb{P}$ applies to arbitrary $\mathbb{P}$, it is best behaved when $\mathbb{P}$ is an $\omega$-poset, as we will soon see.  For example, minimal selectors are then special kinds of filters, as noted in \autoref{MinimalSelectorsAreFilters} below, just like in many more classical topological dualities.  Under suitable regularity conditions, they can even be characterised as the maximal round filters, as shown below in \autoref{RoundFilterSelectors}.

First let us define the \emph{power space} of $\mathbb{P}$ as the power set $\mathsf{P}\mathbb{P}$ with the topology generated by the subbasis $(p_\mathsf{P}^\in)_{p\in\mathbb{P}}$, where
\[p_\mathsf{P}^\in=\{S\in\mathsf{P}\mathbb{P}:p\in S\}.\]
Equivalently, this is the topology we get from identifying every $S\subseteq\mathbb{P}$ with its characteristic function $\chi_S\in\mathbf{2}^\mathbb{P}$, where $\mathbf{2}=\{0,1\}$ is the Sierpi\'nski space (where $\{1\}$ is open but $\{0\}$ is not) and $\mathbf{2}^\mathbb{P}$ is given the usual product topology.

%We will be particularly interested in the subspace of minimal selectors.

\begin{dfn}\label{Spectrum}
The \emph{spectrum} is the subspace of $\mathsf{P}\mathbb{P}$ consisting of minimal selectors
\[\mathsf{S}\mathbb{P}=\{S\subseteq\mathbb{P}:S\text{ is a minimal selector}\}.\]
\end{dfn}

So $\mathsf{S}\mathbb{P}$ has a subbasis consisting of the sets $p_{\mathsf{S}}^\in=p_\mathsf{P}^\in\cap\mathsf{S}\mathbb{P}$, for $p\in\mathbb{P}$.  From now on we will usually drop the subscript and just write $p_{\mathsf{S}}^\in$ as $p^\in$.

By \autoref{MinimalSelectors}, minimal selectors are always up-sets so, for all $p,q\in\mathbb{P}$,
\[p\leq q\qquad\Rightarrow\qquad p^\in\subseteq q^\in.\]
We can thus view the sets $(p^\in)_{p\in\mathbb{P}}$ as a more concrete representation of the poset $\mathbb{P}$ as a subbasis of a topological space.  However, this representation may not always be faithful, at least with respect to the original ordering, i.e. it is possible to have $p^\in\subseteq q^\in$ even when $p\nleq q$.  It is even possible for $p^\in$ to be empty, for some $p\in\mathbb{P}$.

For example, consider the graded $\omega$-poset $\mathbb{P}=\omega\times\{0,1\}$ where
\begin{equation*}
(n,\delta)\leq(n',\delta')\qquad\Leftrightarrow\qquad n'\leq n\text{ and }\delta'\leq\delta.
\end{equation*}
The levels of $\mathbb{P}$ are then given by $\mathbb{P}_0=\{(0,0)\}$ and $\mathbb{P}_n=\{(n,0),(n-1,1)\}$, for all $n>0$.  The only minimal selector is then $\omega\times\{0\}$ so $(n,1)^\in=\emptyset$, for all $n\in\omega$.

The representation $p\mapsto p^\in$ will, however, be faithful with respect to $\precsim$, as defined in \eqref{CapOrder}.  In particular, it will be faithful with respect to the original order precisely when $\mathbb{P}$ is cap-determined.

\begin{prp}\label{prp:caporder2}
For any $A,B\subseteq\mathbb{P}$,
\begin{equation}\label{AprecsimB}
A\precsim B\qquad\Leftrightarrow\qquad\bigcup_{a\in A}a^\in\subseteq\bigcup_{b\in B}b^\in.
\end{equation}
\end{prp}

\begin{proof}
By \eqref{bprecC}, it suffices to consider a singleton $A=\{a\}$.

Now take a minimal selector $S\in a^\in$.  Minimality means we have a cap $C\in\mathsf{C}\mathbb{P}$ such that $C\cap S=\{a\}$.  If $a\precsim B$ then it follows that $B\cup(C\setminus\{a\})$ is a cap and hence $B\cap S=(B\cup(C\setminus\{a\}))\cap S\neq\emptyset$, as $S$ is a selector, i.e. $S\in\bigcup_{b\in B}b^\in$.  This shows that $a^\in\subseteq\bigcup_{b\in B}b^\in$.

Conversely, if $a\not\precsim B$ then we have $F\subseteq\mathbb{P}$ such that $\{a\}\cup F$ is a cap but $B\cup F$ is not.  This means $\mathbb{P}\setminus(B\cup F)$ is a selector and hence contains a minimal selector $S$, by \autoref{SelectorsContainMinimalSelectors}.  As $\{a\}\cup F$ is a cap and $F$ is disjoint from $S$, it follows that $a\in S$ so $S\in a^\in\setminus\bigcup_{b\in B}b^\in$, as $B$ is disjoint from $S$, i.e. $S$ witnesses $a^\in\nsubseteq\bigcup_{b\in B}b^\in$.
\end{proof}

For any $C\subseteq\mathbb{P}$, we denote the corresponding family of open sets in $\mathsf{S}\mathbb{P}$ by
\[C_\mathsf{S}=\{c^\in:c\in C\}.\]

\begin{cor}\label{CapDetermined}
The map $p\mapsto p^\in$ is an order isomorphism from $\mathbb{P}$ onto the canonical subbasis $\mathbb{P}_\mathsf{S}$ of the spectrum precisely when $\mathbb{P}$ is cap-determined.
\end{cor}
\begin{proof}
For any $p,q\in\mathbb{P}$ we have
\[p\leq q\qquad\Rightarrow\qquad p\precsim q\qquad\Leftrightarrow\qquad p^\in\subseteq q^\in\]
by \eqref{AprecsimB} and former observations.
The remaining implication $p \leq q \Leftarrow p \precsim q$ is equivalent by definition to $\PP$ being cap-determined.
\end{proof}

\autoref{prp:caporder2} yields the first fundamental properties of the spectrum.

\begin{prp}\label{SpectrumCompactT1}
The spectrum is a compact $\mathsf{T}_1$ space. Moreover, $C\subseteq\mathbb{P}$ is a cap precisely when the corresponding subbasic sets $C_\mathsf{S}$ cover the whole spectrum.
\end{prp}

\begin{proof}
Given any distinct $S,T\in\mathsf{S}\mathbb{P}$, minimality implies that we have $s\in S\setminus T$ and $t\in T\setminus S$. This means $S\in s^\in\not\ni T$ and $T\in t^\in\not\ni S$, showing that $\mathsf{S}\mathbb{P}$ is $\mathsf{T}_1$.

By \eqref{MaximalCaps}, $C\subseteq\mathbb{P}$ is a cap precisely when $\mathbb{P}\precsim C$, which is equivalent to saying $C_\mathsf{S}$ covers the entire spectrum, by \eqref{AprecsimB}.  As every cap contains a finite subcap, $X$ is compact, by the Alexander--Wallman subbasis lemma (see \cite{Wallman1938} or \cite{Alexander1939}).
\end{proof}

The spectrum can also recover a space from the order structure of a cap-basis.

\begin{prp}\label{SpaceRecovery}
If $\mathbb{P}$ is a cap-basis of a $\mathsf{T}_1$ space $X$ then
\[x\mapsto x^\in=\{p\in\mathbb{P}:x\in p\}\]
is a homeomorphism from $X$ onto $\mathsf{S}\mathbb{P}$.
\end{prp}

\begin{proof}
Take any $x\in X$. By assumption, any cap $C\in\mathsf{C}\mathbb{P}$ is a cover of $X$ and hence we have some $c\in C$ containing $x$, i.e. $c\in x^\in\cap C$. This shows that $x^\in$ is a selector. Now take any $p\in x^\in$. For any $y\in X\setminus p$, we have some $q\in y^\in\setminus x^\in$, as $X$ is $\mathsf{T}_1$. This means $C=\{p\}\cup(\mathbb{P}\setminus x^\in)$ is a cover of $X$ and hence a cap of $\mathbb{P}$ with $C\cap x^\in=\{p\}$. Thus $x^\in$ is a minimal selector.

On the other hand, for any selector $S\in\mathsf{S}\mathbb{P}$, we know that $\mathbb{P}\setminus S$ can not cover $X$ (otherwise it would be cap with $S\cap(\mathbb{P}\setminus S)=\emptyset$, a contradiction). So we can pick $x\in X$ not covered by $\mathbb{P}\setminus S$, which means $x^\in\subseteq S$. If $S$ is a minimal selector then this implies $x^\in= S$. This shows that $\mathsf{S}\mathbb{P}=\{x^\in:x\in X\}$. Also $x\neq y$ implies $x^\in\neq y^\in$, as $X$ is $\mathsf{T}_1$, so $x\mapsto x^\in$ is a bijection from $X$ onto $\mathsf{S}\mathbb{P}$.

Finally, note that $x\mapsto x^\in$ maps each $p\in\mathbb{P}$ onto $p^\in$, as
\[x\in p\qquad\Leftrightarrow\qquad p\in x^\in\qquad\Leftrightarrow\qquad x^\in\in p^\in.\]
As $\mathbb{P}$ is a (sub)basis of $X$ and $(p^\in)_{p\in\mathbb{P}}$ is a subbasis of the spectrum $\mathsf{S}\mathbb{P}$, this shows that the map $x\mapsto x^\in$ is a homeomorphism from $X$ onto $\mathsf{S}\mathbb{P}$.
\end{proof}

Spectra thus yield a large class of spaces.

\begin{cor}\label{SpacesFromGradedPosets}
Every second countable compact $\mathsf{T}_1$ space arises as the spectrum of some predetermined branching graded $\omega$-poset.
\end{cor}

\begin{proof}
By \autoref{BandCapBases}, any second countable compact $\mathsf{T}_1$ space $X$ has a graded $\omega$-band-basis $\mathbb{P}$, which is predetermined, by \autoref{PredeterminedGradedBasis}, and branching, by \autoref{BranchingBasis}. 
 Moreover, its spectrum $\mathsf{S}\mathbb{P}$ is homeomorphic to $X$, by \autoref{SpaceRecovery}.
\end{proof}

\begin{rmk}
    Any graded $\omega$-poset is determined by the order relation between consecutive levels.  By \autoref{SpacesFromGradedPosets}, we should therefore be able to construct any second countable compact $\mathsf{T}_1$ space by recursively defining relations between finite sets $\mathbb{P}_0,\mathbb{P}_1,\ldots$ and then looking at the spectrum of the resulting poset $\mathbb{P}=\bigcup_{n\in\omega}\mathbb{P}_n$. The exact nature of the construction will of course depend on the space we wish to construct, as we will soon see in the examples of the next subsection.  In future work, we will examine more examples constructed within the framework of Fra\"iss\'e theory as it applies to certain subcategories of relations between graphs.
\end{rmk}

In \autoref{GradedLemma}, we saw how graded posets arise from consolidations.  Conversely, levels of graded posets correspond to consolidations in the spectrum.

Note that $\mathbb{P}_{n\mathsf{S}}$ below refers to the $_\mathsf{S}$ operation applied to the $n^\mathrm{th}$ level of $\mathbb{P}$, i.e.
\[\mathbb{P}_{n\mathsf{S}}=\{p^\in:p\in\mathbb{P}_n\}.\]

\begin{prp}
    If $\mathbb{P}$ is a graded $\omega$-poset then $\mathbb{P}_{m\mathsf{S}}$ consolidates $\mathbb{P}_{n\mathsf{S}}$ when $m\leq n$.
\end{prp}

\begin{proof}
    If $\mathbb{P}$ is an $\omega$-poset and $m\leq n$ then certainly $\mathbb{P}_n\leq\mathbb{P}_m$ and hence $\mathbb{P}_{n\mathsf{S}}\leq\mathbb{P}_{m\mathsf{S}}$.  Now take any $p\in\mathbb{P}_m$.  For any $S\in p^\in$, minimality yields $C\in\mathsf{C}\mathbb{P}$ with $C\cap S=\{p\}$.  Then we have $k\geq n$ with $\mathbb{P}_k\leq C$ and hence $\mathbb{P}_k\cap S\subseteq p^\geq$.  As $\mathbb{P}$ is graded, for any $q\in\mathbb{P}_k\cap S$, we have $r\in\mathbb{P}_n\cap(q,p)\subseteq q^\leq\subseteq S$ and hence $S\in r^\in\subseteq p^\in$.  Thus $p^\in=\bigcup\{r^\in:r\in\mathbb{P}_n\cap p^\geq\}$, showing that $\mathbb{P}_{m\mathsf{S}}$ consolidates $\mathbb{P}_{n\mathsf{S}}$.
\end{proof}

Before moving on, however, let us make a couple more observations about spectra arising from general $\omega$-posets.  The first thing to note is that every element $S$ of the spectrum of an $\omega$-poset is not just an up-set but even a \emph{filter}, i.e.
\begin{equation*}
\tag{Filter}p,q\in S\qquad\Leftrightarrow\qquad\exists r\in S\ (r\leq p,q)
\end{equation*}
(note $\Rightarrow$ means $S$ is down-directed while $\Leftarrow$ just means $S$ is an up-set).

\begin{prp}\label{MinimalSelectorsAreFilters}
If $\mathbb{P}$ is an $\omega$-poset then every $S\in\mathsf{S}\mathbb{P}$ is a filter.
\end{prp}

\begin{proof}
Assume $\mathbb{P}$ is an $\omega$-poset and take a minimal selector $S\in\mathsf{S}\mathbb{P}$. For any $q,r\in S$, we have caps $C,D\in\mathsf{C}\mathbb{P}$ such that $C\cap S=\{q\}$ and $D\cap S=\{r\}$. By \autoref{CapsRefineLevels}, $C$ and $D$ are refined by levels of $\mathbb{P}$. As the levels are linearly ordered by refinement, we can find a single level $L\in\mathsf{C}\mathbb{P}$ which refines both $C$ and $D$. As $S$ is a selector, we can take $s\in S\cap L$. As $L$ refines $C$ and $D$, we have $c\in C$ and $d\in D$ such that $s\leq c,d$ and hence $c,d\in s^\leq\subseteq S$. But $q$ and $r$ are the only elements of $S$ in $C$ and $D$ respectively so $q=c\geq s$ and $r=d\geq s$, which shows that $S$ is down-directed. By \autoref{MinimalSelectors}, $S$ is also an up-set.
\end{proof}

\begin{cor}\label{BasisOrderIsomorphism}
If $\mathbb{P}$ is an $\omega$-poset then $\mathbb{P}_\mathsf{S}$ is a basis for $\mathsf{S}\mathbb{P}$.
\end{cor}

\begin{proof}
Whenever $S\in p^\in\cap q^\in$, we have $r\in S$ with $r\leq p,q$, by \autoref{MinimalSelectorsAreFilters}.  But this means $S\in r^\in\subseteq p^\in\cap q^\in$, showing that $\mathbb{P}_\mathsf{S}$ is a basis.
\end{proof}

This yields a kind of converse to \autoref{PredeterminedGradedBasis}.

\begin{cor}\label{BranchingPredeterminedBasis}
Any cap-determined $\omega$-poset $\mathbb{P}$ arises as a cap-basis of a $\mathsf{T}_1$ space.
\end{cor}

\begin{proof}
Immediate from \autoref{CapDetermined}, \autoref{SpectrumCompactT1} and \autoref{BasisOrderIsomorphism}.
\end{proof}

\subsection{Examples}\label{Examples}

Our spectrum generalises the well-known construction of a metrisable Stone space from the branches of an $\omega$-tree (sometimes called its \emph{branch space}, as in \cite[\S III]{Fuller1983} for example).  Of course, the advantage of our spectrum, as applied to more general graded $\omega$-posets, is that we can also construct connected spaces, the simplest example being the arc.

\begin{xpl}\label{ArcExample}
Let $X$ be the arc, which we can take to be the unit interval $[0,1]$ in its usual topology. Define open covers $(C_n)$ of $X$ by
\begin{equation*}
C_n = \{\operatorname{int}([(k - 1)/2^{n+1}, (k + 1)/2^{n + 1}]): 1 \leq k \leq 2^{n + 1} - 1\}.
\end{equation*}
So each $C_n$ consists of $2^{n+1}-1$ evenly spaced intervals, each of length $2^{-n}$. Then $\mathbb{P}=\bigcup_{n\in\omega}C_n$ is a predetermined graded $\omega$-poset, by \autoref{GradedLemma}, which can also be seen directly from its Hasse diagram, as drawn below.
\begin{center}
\begin{tikzpicture}
\node (00) at (0,0) [align=center]{$[0,1]$};
\node (-2-2) at (-2,-2) [align=center]{$[0,\tfrac{1}{2})$};
\node (0-2) at (0,-2) [align=center]{$(\tfrac{1}{4},\tfrac{3}{4})$};
\node (2-2) at (2,-2) [align=center]{$(\tfrac{1}{2},1]$};
\node (-3-4) at (-3,-4) [align=center]{$[0,\tfrac{1}{4})$};
\node (-2-4) at (-2,-4) [align=center]{$(\tfrac{1}{8},\tfrac{3}{8})$};
\node (-1-4) at (-1,-4) [align=center]{$(\tfrac{1}{4},\tfrac{1}{2})$};
\node (0-4) at (0,-4) [align=center]{$(\tfrac{3}{8},\tfrac{5}{8})$};
\node (1-4) at (1,-4) [align=center]{$(\tfrac{1}{2},\tfrac{3}{4})$};
\node (2-4) at (2,-4) [align=center]{$(\tfrac{5}{8},\tfrac{7}{8})$};
\node (3-4) at (3,-4) [align=center]{$(\tfrac{3}{4},1]$};
\draw (00)--(-2-2);
\draw (00)--(0-2);
\draw (00)--(2-2);
\draw (-2-2)--(-3-4);
\draw (-2-2)--(-2-4);
\draw (-2-2)--(-1-4);
\draw (0-2)--(-1-4);
\draw (0-2)--(0-4);
\draw (0-2)--(1-4);
\draw (2-2)--(1-4);
\draw (2-2)--(2-4);
\draw (2-2)--(3-4);
\node (0-4.7) at (0,-4.7) [align=center]{\vdots};
\end{tikzpicture}
\end{center}
Note $\mathbb{P}$ is a cap-basis, by \autoref{OmegaCapBasesInMetricCompacta} (or \autoref{GradedLemma} \ref{GradedLemma2}).  By \autoref{SpaceRecovery}, the spectrum of $\mathbb{P}$ then recovers the original space $X$, i.e. the arc. A more combinatorial construction of the arc could thus proceed as follows -- first define relations between finite linearly ordered sets $\mathbb{P}_n$ such that each element of $\mathbb{P}_n$ is related to 3 consecutive elements of $\mathbb{P}_{n+1}$ and consecutive pairs in $\mathbb{P}_n$ are related to exactly 1 common element in $\mathbb{P}_{n+1}$. Then let $\mathbb{P}=\bigcup_{n\in\omega}\mathbb{P}_n$ with the order defined from the relations between consecutive $\mathbb{P}_n$'s. Finally, define the arc as the spectrum of $\mathbb{P}$.
\end{xpl}

The following example shows that a basis forming a cap-determined poset is not necessarily a cap-basis, i.e. the converse of \autoref{CapBasesAreCapDetermined} is not true (although the poset will yield a cap-basis of a different space, namely its spectrum, which will be a quotient of the original space -- see \autoref{CapBasification} below).
\begin{xpl}\label{CircleArc}
Let $Y$ be the unit circle in the complex plane with the usual topology, and let $\theta\maps \RR \to Y$ be the covering map $x \mapsto e^{2\pi i x}$, so the restriction $\theta\maps [0, 1] \to Y$ is the quotient map identifying the end-points.
We define open covers $(D_n)$ of $Y$ by 
\begin{equation*}
D_n = \{\theta\im{((k-1)/2^{n + 1},(k + 1)/2^{n + 1})}: 0 \leq k < 2^{n + 1}\}
\end{equation*}
So each $D_n$ for $n \geq 1$ consists of $2^{n + 1}$ evenly spaced arcs of length $2\pi / 2^n$.
In particular, $D_1$ consists of the images of the intervals $(-1/4, 1/4)$, $(0, 1/2)$, $(1/4, 3/4)$, and $(1/2, 1)$.
We put $D_0 = \{X\}$ and $\QQ = \bigcup_{n \in \omega} D_n$.
As in the previous example, $\QQ$ is a predetermined graded $\omega$-poset and a cap-basis of $Y$ so the spectrum of $\QQ$ recovers the space $Y$, i.e. the circle.

Let $C'_n = C_n \cup \{[0, 1/2^{n+1}) \cup (1 - 1/2^{n + 1}, 1]\}$ where $C_n$ is the cover of the arc $X = [0, 1]$ from the previous example for $n \geq 1$ and $C'_0 = C_0 = \{X\}$, and let $\PP' = \bigcup_{n \in \omega} C'_n$.
Observe that $p \mapsto \operatorname{int}(\theta\im{p})$ is an isomorphism of posets $\PP' \to \QQ$.
It follows that $\PP'$ is a cap-determined poset and an open basis of $X$ (as it contains the original cap-basis $\PP$), but is not a cap-basis of $X$ (as its spectrum is the circle and not the arc).
\end{xpl}

Our primary interest is in Hausdorff spaces, but our spectrum can indeed produce more general $\mathsf{T}_1$ spaces.  Some of these are not even sober($\Leftrightarrow$ each irreducible closed set has a unique dense point), like the cofinite topology on a countably infinite set.

\begin{xpl}
Let $X = \omega$ with the cofinite topology (i.e. non-empty open sets are exactly the cofinite ones), and let $\PP = \{p_{n, i}: i \leq n,\, n \in \omega\}$ where $p_{n, i} = (\omega \setminus n + 1) \cup \{i\}$, so $p_{0, 0} = \omega$, $p_{1, 0} = \omega \setminus \{1\}$, $p_{1, 1} = \omega \setminus \{0\}$, $p_{2, 0} = \omega \setminus \{1, 2\}$, $p_{2, 1} = \omega \setminus \{0, 2\}$, $p_{2, 2} = \omega \setminus \{0, 1\}$, and so on.
Clearly, $\{p_{n, i}: n > i\}$ is an open basis at $i \in X$, and so $\PP$ is an open basis of $X$.

Every $p_{n, i}$, $i \leq n \in \omega$, has exactly two immediate predecessors: $p_{n + 1, i}$ and $p_{n + 1, n + 1}$, and so every $p_{n, i}$ with $i < n \neq 0$ has a unique immediate successor $p_{n - 1, i}$, while $p_{n, n}$ with $n \neq 0$ has all elements $p_{n - 1, i}$ for $i \leq n - 1$ as immediate successors.
It follows that $\PP$ is a predetermined branching atomless graded $\omega$-poset, as shown in \autoref{fig:cofinite}, with disjoint levels $\PP_n = \{p_{n, i}: i \leq n\}$.

The levels $\PP_n$ are minimal covers of $X$ since $p_{n, i}$ is the unique set containing $i$.
Also, every $\PP_n$ is a consolidation of $\PP_{n + 1}$ since $p_{n, i} = p_{n + 1, i} \cup p_{n + 1, n + 1}$.
Altogether, $\PP$ is a cap-basis of $X$ by \autoref{GradedLemma}.
\end{xpl}

\begin{figure}[ht!]
\begin{tikzpicture}[x={(0, -3em)}, y={(3em, 0)}]
	\def \nmax {4}
	
	\begin{scope}[every node/.style={circle, fill, inner sep=2pt}]
		\foreach \n in {0, ..., \nmax}
			\foreach \k in {0, ..., \n}
				\node (\n\k) at (\n, \k) {};
	\end{scope}
	
	\foreach \n in {0, ..., \numexpr\nmax - 1\relax}
		\pgfmathparse{\n+1}
		\pgfmathtruncatemacro{\np}{\pgfmathresult}
		\foreach \k in {0, ..., \n}
			\draw (\n\k) -- (\np\k);
	
	\foreach \n in {1, ..., \nmax}
		\pgfmathparse{\n-1}
		\pgfmathtruncatemacro{\nm}{\pgfmathresult}
		\foreach \k in {0, ..., \nm}
			\draw (\nm\k) -- (\n\n);
	
	\node at (\nmax + 0.5, \nmax/2) {...};
\end{tikzpicture}

\caption{The poset $\PP$ for the cofinite topology on $\omega.$}
\label{fig:cofinite}
\end{figure}
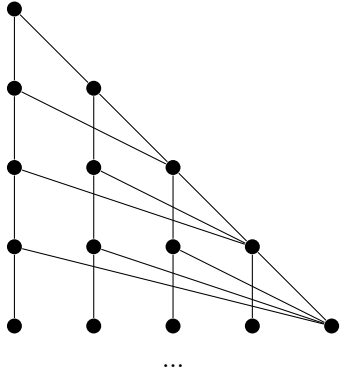

However, this does not have an uncountable extension.

\begin{xpl}
An uncountable set $X$ with the cofinite topology has no cap-basis.
\end{xpl}

\begin{proof}
    Suppose $\QQ$ is a subbasis of $X$ consisting of nonempty and so cofinite sets.
    By the $\Delta$-system lemma, there are pairwise disjoint finite sets $R$ and $F_\alpha$, $\alpha \in \omega_1$, such that $\mathcal{D} = \{X \setminus (R \cup F_\alpha): \alpha \in \omega_1\} \subseteq \PP$.
    Let $B \subseteq \PP$ be any band.
    Since $B$ is finite, there is $b \in B$ comparable to uncountably many elements of $\mathcal{D}$.
    Since $b$ is cofinite, it cannot be below uncountably many elements of $\mathcal{D}$.
    Hence, $b$ is above uncountably many elements of $\mathcal{D}$, and so $b \supseteq X \setminus R$.
    It follows that the finite upwards closed family $F = \{b \subseteq X: b \supseteq X \setminus R\}$ is a selector.
    
    If $\QQ$ was a cap-basis, a minimal selector contained in $F$ would correspond to a point of $X$ with a finite local (sub)basis, by \autoref{SpaceRecovery}, which is impossible.
\end{proof}

The following gives an example of spectrum that is not a first-countable space.
\begin{xpl}
Let $\kappa$ be an infinite cardinal, let $X = \kappa \cup \{\infty\}$ be the one-point compactification of $\kappa$ with the discrete topology, and let $\PP = \{F, \kappa \setminus F: F \subseteq \kappa \text{ finite}\} \setminus \{\emptyset\}$ be the finite-cofinite algebra on $\kappa$ minus the bottom element.
For every $\alpha \in \kappa$, let $S_\alpha = \{\alpha\}^\subseteq$ be the principal filter generated by $\{\alpha\}$, and let $S_\infty$ be the family of all cofinite elements of $\PP$.
We show that the map $f\maps X \to \Spec{\PP}$ defined by $x \mapsto S_x$ is a homeomorphism.

\begin{proof}
Since $S_\alpha$ is an up-set and $\{\alpha\}$ is an atom in $\PP$, $S_\alpha$ is a selector as every band contains an element above $\{\alpha\}$.
Moreover, $S_\alpha$ is a minimal selector since every subselector $S \subseteq S_\alpha$ has to overlap the band $\{\{\alpha\}, \kappa \setminus \{\alpha\}\}$ and so contains $\{\alpha\}$. By \autoref{SelectorsContainMinimalSelectors} there is a minimal selector $S' \subseteq S$, and it is equal to $S_\alpha$ as it is an up-set (\autoref{MinimalSelectors}) containing $\{\alpha\}$.

$S_\infty$ is a selector since every band is finite and so has to contain a cofinite element.
For every finite $F \subseteq \kappa$, the family $\{\{\alpha\}: \alpha \in F\} \cup \{\kappa \setminus F\}$ is a band, and so every selector either contains an atom $\{\alpha\}$, or contains all cofinite elements.
Hence, $S_\infty$ is a minimal selector, and $\Spec{\PP} = \{S_\alpha: \alpha \in \kappa\} \cup \{S_\infty\}$.

We have already shown that $f$ is a bijection.
Now for every finite $F \subseteq \kappa$ we have $F^\in = \{S_\alpha: \alpha \in F\}$ and $(\kappa \setminus F)^\in = \{S_\alpha: \alpha \notin F\} \cup \{S_\infty\}$, and so the elements of $\PP$ correspond to basic open sets of $X$.
\end{proof}
\end{xpl}

\subsection{Subcompacta}\label{Subcompacta}

Next we examine closed subsets of the spectrum.

\begin{prp}\label{SubsetsDefineClosures}
Any $Q\subseteq\mathbb{P}$ determines a closed subset of $\mathsf{S}\mathbb{P}$ given by
\[Q^\supseteq=\{S\in\mathsf{S}\mathbb{P}:S\subseteq Q\}.\]
\end{prp}

\begin{proof}
If $S\in\mathrm{cl}(Q^\supseteq)$ then every subbasic neighbourhood $p^\in$ containing $S$ must also contain some $T\in Q^\supseteq$. In other words, for every $p\in S$, we have $T\in\mathsf{S}\mathbb{P}$ with $p\in T\subseteq Q$. Thus $S\subseteq Q$, showing that $S\in Q^\supseteq$ and hence $\mathrm{cl}(Q^\supseteq)=Q^\supseteq$.
\end{proof}

In fact, every closed subset of the spectrum of an $\omega$-poset arises in this way.

\begin{prp}\label{prop:closures}
If $\mathbb{P}$ is an $\omega$-poset then the closure of any $X\subseteq\mathsf{S}\mathbb{P}$ is given by
\begin{equation}\label{Closure}
\mathrm{cl}(X)=(\bigcup X)^\supseteq.
\end{equation}
\end{prp}

\begin{proof}
By the previous result, $(\bigcup X)^\supseteq$ is a closed subset of $\mathsf{S}\mathbb{P}$ which certainly contains $X$. Conversely, if $S\in(\bigcup X)^\supseteq$ then, for any $p\in S$, we have some $T\in X$ with $p\in T$, i.e. every subbasic neighbourhood of $S$ contains an element of $X$. However, if $\mathbb{P}$ is an $\omega$-poset then $(p^\in)_{p\in\mathbb{P}}$ is actually a basis for $\mathsf{S}\mathbb{P}$, by \autoref{BasisOrderIsomorphism}. Thus this shows that $S\in\mathrm{cl}(X)$, which in turn shows that $\mathrm{cl}(X)=(\bigcup X)^\supseteq$.
\end{proof}

Let us call $Q\subseteq\mathbb{P}$ \emph{prime} if $q\not\precsim\mathbb{P}\setminus Q$, for all $q\in Q$, where $\precsim$ is the relation from \eqref{CapOrder}.
Put another way, this means that $Q$ must overlap every subset which is cap-above any element of $Q$, i.e. for all $C\subseteq\mathbb{P}$,
\[\tag{Prime}\label{Prime}Q\ni q\precsim C\qquad\Rightarrow\qquad Q\cap C\neq\emptyset.\]
Indeed, if $q\not\precsim\mathbb{P}\setminus Q$ and $q\precsim C$ then $C\nsubseteq\mathbb{P}\setminus Q$, i.e. $Q\cap C\neq\emptyset$, showing \eqref{Prime} holds. Conversely, if $Q\ni q\precsim\mathbb{P}\setminus Q$ then $\mathbb{P}\setminus Q$ itself witnesses the failure of \eqref{Prime}.

It follows that any non-empty prime $Q\subseteq\mathbb{P}$ is automatically a selector -- if $q\in Q$ and $C\in\mathsf{C}\mathbb{P}$ then certainly $q\precsim C$ and hence $Q\cap C\neq\emptyset$.  Actually, more is true.

\begin{prp}\label{PrimeImpliesUnionOfMinimal}
Prime subsets are precisely the unions of minimal selectors.
\end{prp}

\begin{proof}
Take a minimal selector $S\subseteq\mathbb{P}$. For every $s\in S$, minimality yields $F\subseteq S\setminus\mathbb{P}$ such that $F\cup\{s\}$ is a cap. But $S\setminus\mathbb{P}(=(S\setminus\mathbb{P})\cup F)$ is not a cap, simply because $S$ is a selector, so $F$ witnesses $s\not\precsim S\setminus\mathbb{P}$. This shows that every minimal selector is prime and hence the same is true of any union of minimal selectors.

Conversely, take any prime $Q\subseteq\mathbb{P}$.  For every $q\in Q$, this means $q\not\precsim\mathbb{P}\setminus Q$ so we have $S\in q^\in\setminus\bigcup_{p\in\mathbb{P}\setminus Q}p^\in$, by \eqref{AprecsimB}, and hence $q\in S\subseteq Q$.  This shows that $Q$ is a union of minimal selectors.
\end{proof}

In fact, as long as $\mathbb{P}$ is an $\omega$-poset, the minimal selectors forming a prime subset $Q$ determine the spectrum of $Q$ when considered as an $\omega$-poset in its own right.

\begin{prp}\label{QSpec}
If $\mathbb{P}$ is an $\omega$-poset and $Q\subseteq\mathbb{P}$ is prime then $\mathsf{S}Q=Q^\supseteq$.
\end{prp}

\begin{proof}
First we claim that the atoms of any prime $Q\subseteq\mathbb{P}$ must already be atoms in $\mathbb{P}$.  Indeed, any $q\in Q$ is contained in some $S\in Q^\supseteq$.  If $q$ is not an atom in $\mathbb{P}$ then we have some level $\mathbb{P}_n$ disjoint from $q^\leq$.  Making $n$ larger if necessary, we may further assume that $S\cap\mathbb{P}_n\subseteq q^\geq$, as $S$ is a minimal selector, and hence $\emptyset\neq S\cap\mathbb{P}_n\subseteq q^>$, showing that $q$ is not an atom in $S\subseteq Q$.  This proves the claim and hence $\mathsf{C}Q=\{C\cap Q:C\in\mathsf{C}\mathbb{P}\}$, by \autoref{UpsetCaps}.  But if $S\subseteq Q$ and $C\in\mathsf{C}\mathbb{P}$ then $S\cap C=S\cap C\cap Q\neq\emptyset$, so it follows that $S$ is a selector in $\mathbb{P}$ if and only if $S$ is a selector in $Q$.  The same then applies to minimal selectors, i.e. $\mathsf{S}Q=Q^\supseteq$.
\end{proof}

In this way, prime subsets of $\mathbb{P}$ correspond exactly to closed subsets of $\mathsf{S}\mathbb{P}$.

\begin{cor}\label{StrongSelectorVsClosedSubsets}
If $\mathbb{P}$ is an $\omega$-poset then we have mutually inverse order isomorphisms between prime $Q\subseteq\mathbb{P}$ and closed subsets $X$ of the spectrum $\mathsf{S}\mathbb{P}$ given by
\begin{equation}\label{StrongSelectorVsClosedSubsetsEq}
    Q\mapsto\mathsf{S}Q\qquad\text{and}\qquad X\mapsto\bigcup X.
\end{equation}
\end{cor}

\begin{proof}
By \autoref{SubsetsDefineClosures}, \autoref{PrimeImpliesUnionOfMinimal} and \autoref{QSpec}, $X\mapsto\bigcup X$ and $Q\mapsto Q^\supseteq=\mathsf{S}Q$ take prime selectors to closed subsets and vice versa. By \autoref{prop:closures}, $X=(\bigcup X)^\supseteq$ whenever $X$ is closed.  By \autoref{PrimeImpliesUnionOfMinimal} again, $Q=\bigcup(Q^\supseteq)$ whenever $Q$ is prime.  Thus these maps are inverse to each other.
\end{proof}

\begin{rmk}
    The frame of open subsets of $\mathsf{S}\mathbb{P}$ can thus be obtained directly from $\mathbb{P}$.  Specifically, complements of prime subsets ordered by inclusion form a frame $\mathbb{F}$ which is order isomorphic to the open subsets of $\mathsf{S}\mathbb{P}$, by \autoref{StrongSelectorVsClosedSubsets}.  Thus $\mathbb{F}$ can be viewed as a kind of completion of $\mathbb{P}$, once we identify each $p\in\mathbb{P}$ with $p^\succsim\in\mathbb{F}$.
\end{rmk}

\begin{dfn} \label{DefPrimePoset}
Let us call a poset $\mathbb{P}$ \emph{prime} if it is prime in itself, i.e. if $p^\in\neq\emptyset$ or, equivalently, $p\not\precsim\emptyset$, for all $p\in\mathbb{P}$.
\end{dfn}

While there do exist non-prime $\omega$-posets (e.g. $\mathbb{P}=\omega\times\{0,1\}$ mentioned just before \autoref{prp:caporder2}), every $\omega$-poset $\mathbb{P}$ contains a prime $\omega$-subposet $\bigcup\mathsf{S}\mathbb{P}$ with exactly the same spectrum, by \autoref{QSpec}.  Also, cap-determined $\omega$-posets are necessarily prime, by \eqref{CapDeterminedPrime}, as are level-injective $\omega$-posets.

\begin{prp} \label{PrimeMinimalCover}
    Every level-injective $\omega$-poset is prime.
\end{prp}

\begin{proof}
    If $\mathbb{P}$ is level-injective then every level of $\mathbb{P}$ is a minimal cap.  For any $p\in\mathbb{P}$, this means $\mathbb{P}_{\mathsf{r}(n)}\setminus\{p\}$ is not a cap and hence $p\not\precsim\emptyset$, showing that $\mathbb{P}$ is prime.
\end{proof}

\begin{thm}\label{OmegaCapBasesFromOmegaPosets}
If $\mathbb{P}$ is a prime $\omega$-poset then $\mathbb{P}_\mathsf{S}$ is an $\omega$-cap-basis for $\mathsf{S}\mathbb{P}$.
\end{thm}

\begin{proof}
    We already showed that $\mathbb{P}_\mathsf{S}$ is a basis in \autoref{BasisOrderIsomorphism}.

    Now take a cover of $\mathsf{S}\mathbb{P}$ from $\mathbb{P}_\mathsf{S}$, i.e. of the form $C_\mathsf{S}$, for some $C\subseteq\mathbb{P}$.  Then $C$ is a cap of $\mathbb{P}$, by \autoref{SpectrumCompactT1}, and is thus refined by some band $B\subseteq\mathbb{P}$.  This implies $B_\mathsf{S}$ is a band of $\mathbb{P}_\mathsf{S}$ which refines $C_\mathsf{S}$, showing that $C_\mathsf{S}$ is a cap of $\mathbb{P}_\mathsf{S}$.  Conversely, caps of $\mathbb{P}_\mathsf{S}$ are covers, by \autoref{CapsCovers}, seeing as $p^\in\neq\emptyset$, for all $p\in\mathbb{P}$, as $\mathbb{P}$ is prime.  This shows that $\mathbb{P}_\mathsf{S}$ is a cap-basis.

    Next say that we have $p\in\mathbb{P}$ and infinite $Q\subseteq\mathbb{P}$ such that $p^\in\subsetneqq q^\in$, for all $q\in Q$.  Then $p\notin Q^\leq\subseteq S\ni p$, for any $S\in p^\in$, even though $Q^\leq$ is an infinite up-set and hence a selector, contradicting the minimality of $S$.  Thus $\mathbb{P}_\mathsf{S}$ is Noetherian and every element of $\mathbb{P}_\mathsf{S}$ has finite rank.

    Now say $\mathbb{P}_\mathsf{S}$ had an infinite level $\mathbb{P}_{\mathsf{S}n}$, which must cover $\mathsf{S}\mathbb{P}$, by \autoref{LevelsCovers}.  Take minimal $L\subseteq\mathbb{P}$ with $L_\mathsf{S}=\mathbb{P}_{\mathsf{S}n}$, which must be an antichain in $\mathbb{P}$, as $L_\mathsf{S}$ is an antichain in $\mathbb{P}_\mathsf{S}$.  By \autoref{CapAntichains}, $L$ can not be a cap, i.e. $\mathbb{P}\setminus L$ is a selector and hence contains a minimal selector $S\notin\bigcup L_S$, contradicting the fact $L_S$ covers $\mathsf{S}\mathbb{P}$.  Thus $\mathbb{P}_\mathsf{S}$ has finite levels and is thus an $\omega$-poset.
\end{proof}

When $\mathbb{P}$ is prime and $X=p^\in$ in \eqref{StrongSelectorVsClosedSubsetsEq}, the union $\bigcup X$ can be described in simple terms using the \emph{common lower bound} relation ${\wedge}={\geq}\circ{\leq}$, i.e. $p\wedge q$ means there exists some $r\in\mathbb{P}$ below both $p$ and $q$ or, more symbolically,
\[p\wedge q\qquad\Leftrightarrow\qquad\exists r\in\mathbb{P}\ (p,q\geq r).\]
Note the following is saying $p\wedge q$ holds precisely when $p^\in\cap q^\in\neq\emptyset$.

\begin{prp}\label{PropPrime}
If $\mathbb{P}$ is a prime $\omega$-poset then, for all $p\in\mathbb{P}$,
\begin{equation}\label{Upni}
\bigcup p^\in=p^\wedge.
\end{equation}
\end{prp}

\begin{proof}
If $\mathbb{P}$ is an $\omega$-poset then every $S\in p^\in$ is a filter, by \autoref{MinimalSelectorsAreFilters}, and hence $S\subseteq p^\wedge$, showing that $\bigcup p^\in\subseteq p^\wedge$. Conversely, if $\mathbb{P}$ is prime and $p\wedge q$ then, taking any $r\in p^\geq\cap q^\geq$, by assumption we have some $S\in r^\in\subseteq p^\in\cap q^\in$ and hence $q\in\bigcup p^\in$, showing that $p^\wedge\subseteq\bigcup p^\in$.
\end{proof}

Incidentally, this result also holds for any cap-basis $\mathbb{P}$ of a space $X$ (which again applies to all cap-determined $\omega$-posets, by \autoref{BranchingPredeterminedBasis}).  Indeed, in this case
\[p\wedge q\qquad\Leftrightarrow\qquad p\cap q\neq\emptyset,\]
for if $p,q\supseteq r\in\mathbb{P}$ then $p\cap q\neq\emptyset$, as $r\neq\emptyset$ (see the comments after \autoref{CapBasis}).  Conversely, if $x\in p\cap q$ then, as $\mathbb{P}$ is a basis, we have $r\in\mathbb{P}$ with $x\in r\subseteq p\cap q$. 
 \autoref{SpaceRecovery} then yields $\bigcup p^\in=\bigcup_{x\in p}x^\in=p^\wedge$.

Here is another simple observation about $\wedge$ that will soon be useful.

\begin{prp}\label{WedgeSplit}
    For any $p,q\in\mathbb{P}$ and $C\in\mathsf{C}\mathbb{P}$,
    \[p\wedge q\qquad\Rightarrow\qquad\exists c\in C\ (p\wedge c\wedge q).\]
\end{prp}

\begin{proof}
    If $C\in\mathsf{C}\mathbb{P}$ then we have $B\in\mathsf{B}\mathbb{P}$ with $B\leq C$.  If $p\wedge q$ then we have $r\leq p,q$.  As $B$ is a band, we then have $b\in B\cap(r^\leq\cap r^\geq)$.  If $b\leq r$ then $b\leq p,q,b$, while if $r\leq b$ then $r\leq p,q,b$.  In either case $p\wedge b\wedge q$ and hence $p\wedge c\wedge q$, for any $c\in C\cap b^\leq$.
\end{proof}

\subsection{Stars}\label{Stars}

For Hausdorff spectra, stars play a particularly important role.  Specifically, as in \cite[\S2.3]{PicadoPultr2012}, we denote the \emph{star} of $p\in\mathbb{P}$ in $C\in\mathsf{C}\mathbb{P}$ by
\[Cp=C\cap p^\wedge.\]

The first thing to observe is the following.

\begin{prp}\label{NonemptyStars}
Stars are never empty.
\end{prp}

\begin{proof}
For any $p\in\mathbb{P}$ and $C\in\mathsf{C}\mathbb{P}$, certainly $p\wedge p$ so $Cp\neq\emptyset$, by \autoref{WedgeSplit}.
\end{proof}

For any $C\in\mathsf{C}\mathbb{P}$, let us define a relation $\vartriangleleft_C$ on $\mathbb{P}$ by
\[p\vartriangleleft_Cq\qquad\Leftrightarrow\qquad Cp\leq q.\]
Note $\vartriangleleft_C$ is also compatible with the ordering, i.e. for all $p,p',q,q'\in\mathbb{P}$,
\[\tag{Compatibility}\label{Compatibility}p\leq p'\vartriangleleft_Cq'\leq q\qquad\Rightarrow\qquad p\vartriangleleft_Cq.\]
Indeed, if $p\leq p'\vartriangleleft_Cq'\leq q$ then $Cp\subseteq Cp'\leq q'\leq q$, i.e. $p\vartriangleleft_Cq$.  Also
\[\tag{Transitivity}\label{Transitivity}p\vartriangleleft_Cq\vartriangleleft_Cr\qquad\Rightarrow\qquad p\vartriangleleft_Cr,\]
as the left side means $Cp\leq q$ and hence $Cp\subseteq Cq\leq r$, thus giving the right side.  Also note that refining the cap results in a weaker relation, i.e. for all $B,C\in\mathsf{C}\mathbb{P}$,
\begin{equation}\label{RefiningStars}
B\leq C\qquad\Rightarrow\qquad\vartriangleleft_C\ \subseteq\ \vartriangleleft_B.
\end{equation}
Indeed, if $B\leq C$ and $p\vartriangleleft_Cq$ then $pB\leq pC\leq q$, i.e. $p\vartriangleleft_Bq$.

The \emph{star-below} relation is the minimal relation $\vartriangleleft$ on $\mathbb{P}$ containing all of these
\[\tag{Star-Below}\vartriangleleft\ =\ \bigcup_{C\in\mathsf{C}\mathbb{P}}\vartriangleleft_C,\]
i.e. $p\vartriangleleft q$ means $p\vartriangleleft_Cq$, for some $C\in\mathsf{C}\mathbb{P}$ and hence some $B\in\mathsf{B}\mathbb{P}$, by \eqref{RefiningStars}.  Whenever $p\vartriangleleft_Cq$, for some $C\in\mathsf{C}\mathbb{P}$, note that we can always replace $C$ with $D={C\setminus q^\geq}\cup\{q\}\in\mathsf{C}\mathbb{P}$.  In other words, $\vartriangleleft$ could also be defined more explicitly by
\begin{equation*}%\label{EqualityStar}
    p\vartriangleleft q\qquad\Leftrightarrow\qquad\exists C\in\mathsf{C}\mathbb{P}\ (Cp=\{q\}).
\end{equation*}

We again immediately see that $\vartriangleleft$ is compatible with the ordering.  As long as $\mathbb{P}$ is an $\omega$-poset then it is also transitive -- in this case any $B,C\in\mathsf{C}\mathbb{P}$ has a common refinement $D\in\mathsf{C}\mathbb{P}$ and so $p\vartriangleleft_Bq\vartriangleleft_Cr$ implies $p\vartriangleleft_Dq\vartriangleleft_Dr$, by \eqref{RefiningStars}, and hence $p\vartriangleleft_Dr$, by \eqref{Transitivity}.  We also see that $\vartriangleleft\ \subseteq\wedge$ and even
\begin{equation}\label{WedgeStar}
    {\wedge}\circ{\vartriangleleft}\ \subseteq\ {\wedge}.
\end{equation}
Indeed, if $p\wedge q\vartriangleleft r$ then we have $s\leq p$ with $s\leq q\vartriangleleft r$ and hence $s\vartriangleleft_Cr$, for some $C\in\mathsf{C}\mathbb{P}$.  Then \autoref{NonemptyStars} yields $t\in Cs$ so $p\geq s\wedge t\leq r$ and hence $p\wedge r$.

The significance of $\vartriangleleft$ is that it represents `closed containment' in the spectrum.

\begin{prp}\label{ClosedContainmentProp}
If $\mathbb{P}$ is an $\omega$-poset then, for all $p\in\mathbb{P}$,
\begin{equation}\label{ClosedContainment}
p\vartriangleleft q\qquad\Rightarrow\qquad\mathrm{cl}(p^\in)\subseteq q^\in.
\end{equation}
The converse also holds when $\mathbb{P}$ is also prime.
\end{prp}

\begin{proof}
If $p\vartriangleleft q$ then we have $C\in\mathsf{C}\mathbb{P}$ with $Cp\leq q$.  Take $S\in\mathrm{cl}(p^\in)=(\bigcup p^\in)^\supseteq$, by \eqref{Closure}.  By \autoref{MinimalSelectorsAreFilters}, all minimal selectors are filters so $\bigcup p^\in\subseteq p^\wedge$ and hence $S\subseteq p^\wedge$.  It follows that $\emptyset\neq S\cap C\subseteq Cp\leq q$ and hence $q\in S^\leq=S$, i.e. $S\in q^\in$. This shows that $\mathrm{cl}(p^\in)\subseteq q^\in$.

Now assume $\mathbb{P}$ is prime.  If $p\not\vartriangleleft q$ then $Cp\nleq q$, for every $C\in\mathsf{C}\mathbb{P}$, i.e. $p^\wedge\setminus q^\geq$ is a selector.  By \autoref{MinimalSelectors}, we have a minimal selector $S\subseteq p^\wedge\setminus q^\geq$ and hence $S\in p^{\wedge\supseteq}=(\bigcup p^\in)^\supseteq=\mathrm{cl}(p^\in)$, by \eqref{Closure} and \eqref{Upni} (this is where we need $\mathbb{P}$ to be prime).  Thus $S\in\mathrm{cl}(p^\in)\setminus q^\in$ witnesses $\mathrm{cl}(p^\in)\nsubseteq q^\in$.
\end{proof}

In particular, $p\vartriangleleft q$ implies $p^\in\subseteq q^\in$ and hence $p\precsim q$, by \eqref{AprecsimB}, so
\[\mathbb{P}\text{ is a cap-determined $\omega$-poset}\qquad\Rightarrow\qquad\vartriangleleft\ \subseteq\ \leq.\]
However, there are non-cap-determined $\omega$-posets with $\vartriangleleft\ \nsubseteq\ \leq$.  For example, if we take $\mathbb{P}=-\omega$ then, for any $p,q\in\mathbb{P}$, we see that $C=\{\min(p,q)\}$ is a band with $Cp=C\leq q$ and hence $p\vartriangleleft q$, i.e. ${\vartriangleleft}=\mathbb{P}\times\mathbb{P}\nsubseteq{\leq}$.

There is one other situation worth noting, though, when $p\vartriangleleft q$ implies $p\leq q$.

\begin{prp}
If $\mathbb{P}$ is an $\omega$-poset and $p\in\mathbb{P}$ is an atom then $p^\vartriangleleft=p^\leq$.
\end{prp}

\begin{proof}
Take a level $L$ containing $p$ and note $pL=\{p\}$.  Indeed, if $l\in pL$ then we have $q\in p^\geq\cap l^\geq$ so $q=p$, as $p$ is an atom, and hence $l=p$, as distinct elements of $L$ are incomparable.  Thus $p\leq q$ implies $p\vartriangleleft_Lq$.  Conversely, if $p\vartriangleleft q$ then we have a band $B$ with $p\vartriangleleft_Bq$.  Thus we have $b\in B$ comparable to $p$ and hence $p\leq b$, as $p$ is an atom.  Thus $b\in Bp\leq q$ and hence $p\leq b\leq q$.
\end{proof}

But sometimes we can replace $\vartriangleleft$ with $\vartriangleleft\cap\leq$.  First let us call $R\subseteq\mathbb{P}$ \emph{round} if
\[\tag{Round}R\subseteq R^\vartriangleleft,\]
i.e. $R$ is round if each $r\in R$ is star-above some $q\in R$.  Let us also call $S\subseteq\mathbb{P}$ \emph{star-prime} if it overlaps every star of every element of $S$, i.e.
\[\tag{Star-Prime}p\in S\quad\text{and}\quad C\in\mathsf{C}\mathbb{P}\qquad\Rightarrow\qquad S\cap Cp\neq\emptyset.\]
For example, $\mathbb{P}$ itself is always star-prime, by \autoref{NonemptyStars}.

\begin{prp}\label{PRoundStrong}
If $\mathbb{P}$ is an $\omega$-poset and $S\subseteq\mathbb{P}$ is both round and star-prime then, for every $r\in S$, we have $s\in S$ such that both $s\vartriangleleft r$ and $s\leq r$ hold. %$S\subseteq S^\trianglelefteq$.
\end{prp}

\begin{proof}
Take any $r\in S$.  If $S$ is round then we have $p,q\in S$ and $C,D\in\mathsf{C}\mathbb{P}$ with $p\vartriangleleft_Cq\vartriangleleft_Dr$.  If $\mathbb{P}$ is an $\omega$-poset then we have $B\in\mathsf{C}\mathbb{P}$ refining both $C$ and $D$.  If $S$ is star-prime then we have $s\in Bp\cap S$.  We then have $c\in C$ with $c\geq s\mathrel{\wedge}p$ and hence $c\in Cp\leq q$.  So $s\leq c\leq q\vartriangleleft r$ and hence $s\vartriangleleft r$, by \eqref{Compatibility}.  On the other hand, we also have $d\in D$ with $d\geq s\leq q$ so $d\in Dq\leq r$ and hence $s\leq d\leq r$.
\end{proof}

If $\mathbb{P}$ is round, we can also improve on \eqref{WedgeStar} as follows.

\begin{prp}
    If $\mathbb{P}$ is round then
    \begin{equation*}%\label{WedgeStarEq}
    {\wedge}={\wedge}\circ{\vartriangleleft}.
    \end{equation*}
\end{prp}

\begin{proof}
    We already know ${\wedge}\supseteq{\wedge}\circ{\vartriangleleft}$, by \eqref{WedgeStar}.  Conversely, say $p\wedge q$, so we have $r\in p^\geq\cap q^\geq$.  If $\mathbb{P}$ is round then we have $s\vartriangleleft r$ so $p,q\vartriangleright s$, by \eqref{Compatibility}, and hence $p\wedge s\vartriangleleft q$, by \eqref{WedgeStar} again, showing that ${\wedge}\subseteq{\wedge}\circ{\vartriangleleft}$.
\end{proof}

To say more, we will also need the caps to be `round' in an appropriate sense.

\subsection{Regularity}

The key condition for Hausdorff spectra is regularity.

\begin{dfn}\label{dfn:regularposet}
We call $\mathbb{P}$ \emph{regular} if every cap is $\vartriangleleft$-refined by another cap, i.e.
\[\tag{Regular}\mathsf{C}\mathbb{P}\subseteq\mathsf{C}\mathbb{P}^\vartriangleleft.\]
\end{dfn}

Equivalently, here we could strengthen $\vartriangleleft$-refinement to \emph{star-refinement} where
\[\tag{Star-Refinement}C\text{ star-refines }D\qquad\Leftrightarrow\qquad C\vartriangleleft_CD.\]

\begin{prp}\label{RegularStarRefinement}
An $\omega$-poset $\mathbb{P}$ is regular precisely when every band or cap is star-refined by another band or cap.
\end{prp}

\begin{proof}
One direction is immediate from the fact that star-refinement is stronger than $\vartriangleleft$-refinement.  Conversely, say $\mathbb{P}$ is regular and take $D\in\mathsf{C}\mathbb{P}$.  By regularity and \eqref{Compatibility}, we have $B\in\mathsf{B}\mathbb{P}$ with $B\vartriangleleft D$, i.e. for each $b\in B$, we have $C_b\in\mathsf{C}\mathbb{P}$ and $d_b\in D$ with $C_bb\leq d_b$.  As $B$ is finite and $\mathbb{P}$ is an $\omega$-poset, we have $A\in\mathsf{B}\mathbb{P}$ with $A\leq B$ and $A\leq C_b$, for all $b\in B$.  For every $p\in A$, we then have $b\in B$ with $p\leq b$ and hence $Ap\leq C_bp\subseteq C_bb\leq d_b$, showing that $A\vartriangleleft_AD$.
\end{proof}

In regular $\omega$-posets, the spectrum consists of round filters.  In fact, it suffices to consider $L\subseteq\mathbb{P}$ that are merely \emph{linked} in that $p\wedge q$, for all $p,q\in L$.

\begin{prp}\label{RoundLinked}
Every round linked selector is minimal.  If $\mathbb{P}$ is an $\omega$-poset,
\[\text{Every minimal selector is round}\qquad\Leftrightarrow\qquad\mathbb{P}\text{ is regular}.\]
\end{prp}

\begin{proof}
If $S$ is round then, for any $s\in S$, we have $t\in S$ and $C\in\mathsf{C}\mathbb{P}$ with $t\vartriangleleft_Cs$.  If $S$ is also linked then $S\cap C\subseteq Ct$ so this implies $S\cap C\leq s$ and hence $(C\setminus S)\cup\{s\}$ is cap, as it is refined by the cap $C$. As $S\cap((C\setminus S)\cup\{s\})=\{s\}$, if $S$ is also a selector then this shows that it must be minimal.

Now if $\mathbb{P}$ is not regular then we have $C\in\mathsf{C}\mathbb{P}\setminus\mathsf{C}\mathbb{P}^\vartriangleleft$.  This means that $ C^\vartriangleright$ does not contain any cap, i.e. $\mathbb{P}\setminus C^\vartriangleright$ is a selector and hence contains some minimal selector $S$, by \autoref{MinimalSelectors}.  In particular, we have some $c\in C\cap S$ and hence $s\not\!\vartriangleleft c$, for all $s\in S$, showing that $S$ is not round.

Conversely, say $\mathbb{P}$ is a regular $\omega$-poset and take a minimal selector $S$.  For any $s\in S$, minimality yields $C\in\mathsf{C}\mathbb{P}$ such that $S\cap C=\{s\}$.  As $\mathbb{P}$ is regular, we have $D\in\mathsf{C}\mathbb{P}$ with $D\vartriangleleft_DC$.  As $S$ is a selector, we have $d\in D\cap S$.  Taking $c\in C$ with $d\vartriangleleft_Dc$ and hence $d\leq c$, it follows $c\in S$ so $s=c\vartriangleright d$, showing that $S$ is round.
\end{proof}

Regularity thus means that the spectrum is Hausdorff/regular/metrisable.

\begin{cor}\label{RegularImpliesHausdorff}
If $\mathbb{P}$ is an $\omega$-poset then
\[\mathbb{P}\text{ is regular}\qquad\Rightarrow\qquad\mathsf{S}\mathbb{P}\text{ is Hausdorff}.\]
The converse also holds as long as $\mathbb{P}$ is prime.
\end{cor}

\begin{proof}
If $\mathbb{P}$ is regular $\omega$-poset then, whenever $S\in p^\in$, \autoref{RoundLinked} yields $q\in S$ with $q\vartriangleleft p$ so $S\in q^\in$ and $\mathrm{cl}(q^\in)\subseteq p^\in$, by \eqref{ClosedContainment}. This shows that $\mathsf{S}\mathbb{P}$ is a regular space and, in particular, Hausdorff.

Conversely, if $\mathbb{P}$ is a prime $\omega$-poset that is not regular then, by \autoref{RoundLinked}, we have some non-round $S\in\mathsf{S}\mathbb{P}$, i.e. we have $c\in S\setminus S^\vartriangleleft$ so $\mathrm{cl}(s^\in)\nsubseteq c^\in$, for all $s\in S$, by \autoref{ClosedContainmentProp}.  This means that $S$ has no closed neighbourhood contained in $c^\in$, showing $\mathsf{S}\mathbb{P}$ is not a regular space.  This, in turn, means that $\mathsf{S}\mathbb{P}$ is not even Hausdorff, as we already know that $\mathsf{S}\mathbb{P}$ is compact, by \autoref{SpectrumCompactT1}.
\end{proof}

We can now also characterise minimal selectors in regular $\omega$-posets as follows.  In particular, in this case the spectrum consists precisely of maximal round filters, just like those considered in compingent lattices in \cite{Shirota1952} and \cite{deVries1962}.

\begin{prp}\label{RoundFilterSelectors}
If $\mathbb{P}$ is a regular $\omega$-poset then
\begin{align*}
    \mathsf{S}\mathbb{P}&=\{S\subseteq\mathbb{P}:S\text{ is a round linked selector }\}\\
    &=\{S\subseteq\mathbb{P}:S\text{ is a round filter selector }\}\\
    &=\{S\subseteq\mathbb{P}:S\text{ is a maximal round filter}\}.
\end{align*}
\end{prp}

\begin{proof}
By \autoref{RoundLinked}, every round linked selector is minimal and every minimal selector is round.  By \autoref{BasisOrderIsomorphism}, every minimal selector is also a filter and, in particular, linked.  This proves the first two equalities.

For the last, first note that any round filter $R$ containing a selector $S$ must again be a selector and hence a minimal selector, by what we just proved, which implies $R=S$.  This shows that round filter selectors are always maximal among round filters.  Conversely, say $M$ is a maximal round filter.  If $M$ were not a selector then it would be finite and not contain any atoms of $\mathbb{P}$, by \autoref{UpsetSelectors}.  As $M$ is a filter, finiteness implies it has a minimum $m$, but then $M=m^\geq$ would not be maximal, as $m$ is not an atom, a contradiction.  Thus $M$ is a selector.
\end{proof}

When the space $X$ in \autoref{WeaklyGradedCapBasis} and \autoref{PredeterminedSubBasis} is Hausdorff, minor modifications of the proofs allow us to construct the cap-bases so that strict containment implies closed containment, i.e.
\[p\subsetneqq q\qquad\Rightarrow\qquad\mathrm{cl}(p)\subseteq q.\]
In terms of the resulting poset, this means ${<}\subseteq{\vartriangleleft}$ (and we can likewise modify the proof of \autoref{PredeterminedSubposet} below when $\mathbb{P}$ is regular to ensure ${<}\subseteq{\vartriangleleft}$ on the subposet $\mathbb{Q}$).  When ${<}\subseteq{\vartriangleleft}$, our spectrum consists precisely of the \emph{ultrafilters}, i.e. the maximal filters in $\mathbb{P}$.  This ultrafilter spectrum is just like that considered for Boolean algebras in the classic Stone duality (originally formulated in terms of maximal ideals -- see \cite{Stone1936}) and has also been considered for general posets more recently in \cite{MummertStephan2010}.

\begin{cor}\label{cor:ultrafilter}
    If $\mathbb{P}$ is a regular $\omega$-poset with ${<}\subseteq{\vartriangleleft}$ then
    \[\mathsf{S}\mathbb{P}=\{U\subseteq\mathbb{P}:U\text{ is an ultrafilter}\}.\]
\end{cor}

\begin{proof}
    Assume $\mathbb{P}$ is an $\omega$-poset with ${<}\subseteq{\vartriangleleft}$.  Take an ultrafilter $U\subseteq\mathbb{P}$.  If $U$ has no minimum then it is round because ${<}\subseteq{\vartriangleleft}$.  If $U$ has a minimum $m$ then this must be an atom, by maximality, in which case $m\vartriangleleft m$ so $U$ is again round.  So all ultrafilters are round and hence these are precisely the maximal round filters.  The result now follows immediately from \autoref{RoundFilterSelectors}.
\end{proof}

However, for graded posets, this only happens when $\vartriangleleft$ is reflexive.  In this case the spectrum has to be totally disconnected and so this never happens for the continua ($\Leftrightarrow$ connected compacta) we are primarily interested in.

\begin{prp}
    If $\mathbb{P}$ is a graded $\omega$-poset with ${<}\subseteq{\vartriangleleft}$ then $\vartriangleleft$ is reflexive.
\end{prp}

\begin{proof}
    Assume $\mathbb{P}$ is an $\omega$-poset with ${<}\subseteq{\vartriangleleft}$ and take any $p\in\mathbb{P}$.  If $p$ is an atom then, in particular, $p\vartriangleleft p$.  If $p$ is not an atom then $F=p^\geq\cap\mathbb{P}_{\mathsf{r}(p)+1}$ is a finite set with $p^>=\bigcup_{f\in F}f^\geq$.  As ${<}\subseteq{\vartriangleleft}$, we then have $C\in\mathsf{C}\mathbb{P}$ with $f\vartriangleleft_Cp$, for all $f\in F$.  Take any $q\in Cp$, so $q\in C$ and we have $r\leq p,q$.  If $r=p$ then $q=p$ because $p=r<q$ would imply $q\geq f$ and, in particular, $q\in Cf\leq p$, for any $f\in F$, a contradiction.  On the other hand, if $r<p$ then $r\leq f$, for some $f\in F$, which implies $q\in Cf\leq p$.  In either case, $q\leq p$, showing that $Cp\leq p$, i.e. $p\vartriangleleft_Cp$.  As $p$ was arbitrary, this proves that $\vartriangleleft$ is reflexive.
\end{proof}

Regularity also yields the following characterisations of prime subsets.

\begin{prp}
Consider the following statements about some $S\subseteq\mathbb{P}$.
\begin{enumerate}
\item\label{PrimeSelector} $S$ is prime.
\item\label{RoundPrimeUpset} $S$ is star-prime and round.
\item\label{MinimalInP} $S$ is a round up-set whose atoms are all already atoms in $\mathbb{P}$.
\end{enumerate}
If $\mathbb{P}$ is an $\omega$-poset then \ref{RoundPrimeUpset}$\Rightarrow$\ref{MinimalInP}$\Rightarrow$\ref{PrimeSelector}.  If $\mathbb{P}$ is also regular then \ref{PrimeSelector}$\Rightarrow$\ref{RoundPrimeUpset} as well.
\end{prp}

\begin{proof}\
\begin{itemize}
\item[\ref{RoundPrimeUpset}$\Rightarrow$\ref{MinimalInP}] If $S$ is round then, for any $p\in S$, we have $q\in S$ and $C\in\mathsf{C}\mathbb{P}$ with $q\vartriangleleft_Cp$.  For any $t\geq p$, this means $D=(C\setminus Cq)\cup\{t\}$ is refined by $C$ and is thus also a cap with $Dq=\{t\}$.  If $S$ is also star-prime then $t\in S$, showing that $S$ is an up-set.  Moreover, if $p$ is not an atom in $\mathbb{P}$ then we can choose $B\in\mathsf{C}\mathbb{P}$ refining $C$ with $p\notin B$ (e.g. take $t<p$ and $B=\mathbb{P}_n$ for some $n\geq\mathsf{r}(t)$ with $\mathbb{P}_n\leq C$).  As $S$ is star-prime, we then have $r\in S\cap Bq\leq Cq\leq p$.  Thus $S\ni r<p$, showing that $p$ is not an atom in $S$ either.

\item[\ref{MinimalInP}$\Rightarrow$\ref{PrimeSelector}] Take any $p\in S$.  If we have some atom $a$ of $\mathbb{P}$ with $a\vartriangleleft p$ and hence $a\leq p$ then $a^\leq$ is a minimal selector containing $p$.  Otherwise, assuming $S$ is round and has no extra atoms, we can recursively define a sequence of distinct $p_n\in S$ with $p=p_0$ and $p_n\vartriangleright p_{n+1}$, for all $n\in\omega$.  As long as $S$ is also an up-set, the upwards closure $U=\bigcup_{n\in\omega}p_n^\leq$ is then a round linked selector.  In particular, $U$ is a minimal selector containing $p$, by \autoref{RoundLinked}.  So $S$ is a union of minimal selectors and thus prime, by \autoref{PrimeImpliesUnionOfMinimal}.

\item[\ref{PrimeSelector}$\Rightarrow$\ref{RoundPrimeUpset}] Now if $\mathbb{P}$ is regular then every $S\in\mathsf{S}\mathbb{P}$ is round and linked, by \autoref{RoundLinked}.  Thus, for every $s\in S$ and $C\in\mathsf{C}\mathbb{P}$, $\emptyset\neq S\cap C\subseteq S\cap Cs$, i.e. $S\cap Cs\neq\emptyset$.  So every minimal selector is round and star-prime and hence the same applies to any union of minimal selectors.  By \autoref{PrimeImpliesUnionOfMinimal}, these are precisely the prime subsets. \qedhere
\end{itemize}
\end{proof}

As $\mathbb{P}$ itself is always star-prime, the above result implies that, in particular, any round $\omega$-poset is prime and, conversely, any regular prime $\omega$-poset is round.  Also, if $\vartriangleleft\ \subseteq\ <$ (e.g. if $\mathbb{P}$ is a cap-basis of proper open subsets of a continuum) then no round subset can contain any atoms, making the last condition in \ref{MinimalInP} superfluous, i.e. in this case every round up-set is prime (and conversely if $\mathbb{P}$ is also regular).

Lastly, we note linked selectors can be made round by taking the star-up-closure.

\begin{cor}\label{LinkedSelector}
    If $\mathbb{P}$ is regular and $S\subseteq\mathbb{P}$ is a linked selector then $S^\vartriangleleft\in\mathsf{S}\mathbb{P}$.
\end{cor}

\begin{proof}
    First note $S^\vartriangleleft$ is linked, by \eqref{WedgeStar}.  To see that $S^\vartriangleleft$ is a selector, take any $C\in\mathsf{C}\mathbb{P}$.  As $\mathbb{P}$ is regular, we have $B\in\mathsf{C}\mathbb{P}$ with $B\vartriangleleft C$.  As $S$ is a selector, we have $b\in B\cap S$.  Then we have $c\in C\cap b^\vartriangleleft\subseteq C\cap S^\vartriangleleft$, as required.  To see that $S^\vartriangleleft$ is round, take any $t\in S^\vartriangleleft$, so we have $s\in S$ and $C\in\mathsf{C}\mathbb{P}$ with $s\vartriangleleft_Ct$.  As $\mathbb{P}$ is regular, we have $B\in\mathsf{C}\mathbb{P}$ with $B\vartriangleleft C$.  As $S^\vartriangleleft$ is a selector, we have $b\in B\cap S^\vartriangleleft$.  Then we have $c\in C\cap b^\vartriangleleft\subseteq C\cap S^{\vartriangleleft\vartriangleleft}\subseteq Cs$, again by \eqref{WedgeStar}, so $b\vartriangleleft c\leq t$, showing that $S^\vartriangleleft$ is indeed round.  Thus $S^\vartriangleleft$ is a minimal selector, by \autoref{RoundLinked}.
\end{proof}

This gives us the following variant of \autoref{SpaceRecovery}, showing that \autoref{CircleArc} is one instance of a more general phenomenon where the spectrum of a regular $\omega$-basis is a quotient of the original compactum.

\begin{cor}\label{CapBasification}
If $\mathbb{P}\subseteq\mathsf{P}X\setminus\{\emptyset\}$ is a regular $\omega$-basis of a $\mathsf{T}_1$ space $X$ then
\[\eta(x)=x^{\in\vartriangleleft}=\{p\in\mathbb{P}:\exists q\in\mathbb{P}\ (x\in q\vartriangleleft p)\}\]
defines a continuous map $\eta:X\rightarrow\mathsf{S}\mathbb{P}$.  If $X$ is compact then $\eta$ is also a closed surjective map.  In this case, $\eta$ is also injective precisely when $\mathbb{P}$ is a cap-basis.
\end{cor}

\begin{proof}
    Take any $x\in X$ and first note $x^\in$ is linked, as $\mathbb{P}$ is a basis.  Also any $C\in\mathsf{C}\mathbb{P}$ covers $X$, by \autoref{CapsCovers}, 
    and hence overlaps $x^\in$, showing that $x^\in$ is also a selector.  By \autoref{LinkedSelector}, $x^{\in\vartriangleleft}\in\mathsf{S}\mathbb{P}$, showing that $\eta$ maps $X$ to $\mathsf{S}\mathbb{P}$.  Continuity is then immediate from the fact $\eta^{-1}[p^\in]=\bigcup p^\vartriangleright$, for all $p\in\mathbb{P}$.

    Now assume $X$ is compact.  First we claim that, for all $p,q\in\mathbb{P}$,
    \[p\vartriangleleft q\qquad\Rightarrow\qquad\mathrm{cl}(p)\subseteq q.\]
    To see this, just note again that any $C\in\mathsf{C}\mathbb{P}$ covers $X$, by \autoref{CapsCovers}, and hence $p\vartriangleleft_Cq$ implies $\mathrm{cl}(p)\subseteq\bigcup Cp\subseteq q$, as $\mathbb{P}$ is a basis.  By \autoref{RoundLinked}, any $S\in\mathsf{S}\mathbb{P}$ is round and so this means $\bigcap S=\bigcap_{s\in S}\mathrm{cl}(s)\neq\emptyset$, as $X$ is compact.  Taking any $x\in\bigcap S$, it follows that $S\subseteq x^\in$ and hence $S\subseteq S^\vartriangleleft\subseteq x^{\in\vartriangleleft}$.  Thus $S=x^{\in\vartriangleleft}$, as $x^{\in\vartriangleleft}$ is a minimal selector, showing that $\eta$ is surjective.

    Similarly, we can show that $\eta$ is a closed map.  To see this, take any closed $Y\subseteq X$ and any $S\in\mathrm{cl}(\eta[Y])$.  By compactness, $\emptyset=Y\cap\bigcap S(=Y\cap\bigcap_{s\in S}\mathrm{cl}(s))$ would imply that $Y\cap\bigcap F=\emptyset$, for some finite $F\subseteq S$.  As $S\in\mathrm{cl}(\eta[Y])\cap\bigcap_{f\in F}f^\in$, we would then have $y\in Y$ with $\eta(y)\in\bigcap_{f\in F}f^\in$.  But this means $F\subseteq y^{\in\vartriangleleft}\subseteq y^\in$ and hence $y\in Y\cap\bigcap F=\emptyset$, a contradiction.  Thus we must have some $y\in Y\cap\bigcap S$ so $S\subseteq S^\vartriangleleft\subseteq y^{\in\vartriangleleft}$ and hence $S=y^{\in\vartriangleleft}\in\eta[Y]$, showing that $\eta[Y]$ is closed.
    
    If $\mathbb{P}$ is a cap-basis then $x^\in$ is already a minimal selector so $x^{\in\vartriangleleft}=x^\in$, for any $x\in X$, and hence $\eta$ is injective, by \autoref{SpaceRecovery}.  Conversely, if $\mathbb{P}$ is not a cap-basis then $X$ has a cover $C\subseteq\mathbb{P}$ which is not a cap.  Thus $\mathbb{P}\setminus C$ is a selector and hence contains a minimal selector $S$, again with $\bigcap S\neq\emptyset$, by compactness.  If we had $\bigcap S=\{x\}$, for some $x\in X$, then $x$ would lie in some $c\in C$.  But then $\bigcap S\setminus c=\emptyset$ so compactness would yield finite $F\subseteq S$ with $\bigcap F\setminus c=\emptyset$.  As $\mathbb{P}$ is a basis, we would then have $s\in S$ with $x\in s\subseteq\bigcap F$ and hence $s\setminus c=\emptyset$, meaning $s\subseteq c$ and hence $c\in s^\leq\subseteq S$, contradicting $S\subseteq\mathbb{P}\setminus C$.  Thus $\bigcap S$ contains at least two distinct $x,y\in X$, necessarily with $S\subseteq x^\in\cap y^\in$ and hence $S=\eta(x)=\eta(y)$, showing that $\eta$ is not injective.
\end{proof}

\subsection{Subcontinua}\label{Subcontinua}

Next we examine connected subsets of the spectrum.

While the open subset $p^\in$ coming from a single $p\in\mathbb{P}$ may not be connected, subsets of $\mathbb{P}$ can still form analogous `clusters'.  First let us extend $\wedge$ to subsets $A,B\subseteq\mathbb{P}$ by defining
    \[A\wedge B\qquad\Leftrightarrow\qquad\exists a\in A\ \exists b\in B\ (a\mathrel{\wedge}b).\]
We call $C\subseteq\mathbb{P}$ a \emph{cluster} if
\[\tag{Cluster}A\neq\emptyset\neq B\quad\text{and}\quad A\cup B=C\qquad\Rightarrow\qquad A\wedge B.\]
In other words, $C$ is a cluster precisely when it is connected as a subset of the graph with edge relation $\wedge$.  Put another way, $C$ fails to be a cluster precisely when $C$ has a non-trivial discrete partition $\{A,B\}$, meaning $A\neq\emptyset\neq B$, $A\cap B=\emptyset$, $A\cup B=C$ and $a^\geq\cap b^\geq=\emptyset$, for all $a\in A$ and $b\in B$.

The first thing to observe is that clusters are `upwards closed'.  For convenience, here and below we let ${\sqsubset_D}={\sqsubset}|_D={\sqsubset}\cap(\mathbb{P}\times D)$, for any ${\sqsubset}\subseteq\mathbb{P}\times\mathbb{P}$ and $D\subseteq\mathbb{P}$.

\begin{prp}
    If $C\subseteq\mathbb{P}$ is a cluster and $C\leq D$ then $C^{\leq_D}$ is also a cluster.
\end{prp}

\begin{proof}
    Say $C^{\leq_D}=A\cup B$, where $A\neq\emptyset\neq B$ and hence $A^{\geq_C}\neq\emptyset\neq B^{\geq_C}$.  If $C\leq D$ then $C=A^{\geq_C}\cup B^{\geq_C}$.  If $C$ is also a cluster then we must have $c\in A^{\geq_C}$ and $d\in B^{\geq_C}$ with $c\wedge d$.  This means we have $a\in A$ and $b\in B$ with $a\geq c\wedge d\leq b$ and hence $a\wedge b$, showing that $C^{\leq_D}$ is also a cluster.
\end{proof}

Connected subsets of the spectrum yield clusters in all caps.

\begin{prp}\label{ConnectedClusters}
    If $\mathbb{P}$ is an $\omega$-poset and $X\subseteq\mathsf{S}\mathbb{P}$ is connected then $C\cap\bigcup X$ is a cluster, for every cap $C\in\mathsf{C}\mathbb{P}$.
\end{prp}

\begin{proof}
    If $C\cap\bigcup X$ were not a cluster, for some $C\in\mathsf{C}\mathbb{P}$, then it would have a discrete partition $\{A,B\}$.  As every minimal selector in an $\omega$-poset is a filter, this means $Y=\bigcup_{a\in A}a^\in$ and $Z=\bigcup_{b\in B}b^\in$ are disjoint non-empty (as $p^\in\neq\emptyset$, for all $p\in\bigcup X$) open subsets covering $X$, contradicting its connectedness.
\end{proof}

Recall from \autoref{SubsetsDefineClosures} that any $Q\subseteq\mathbb{P}$ defines a closed subset of the spectrum $Q^\supseteq=\{S\in\mathsf{S}\mathbb{P}:S\subseteq Q\}$.  As a converse to the above, we can show that if $Q\cap C$ is a cluster, even just coinitially often, then $Q^\supseteq$ is connected.

\begin{prp}\label{CoinitialClusters}
    If $\mathbb{P}$ is a regular prime $\omega$-poset and $Q\subseteq\mathbb{P}$ is an up-set selector,
    \[\{C\in\mathsf{C}\mathbb{P}:Q\cap C\text{ is a cluster}\}\text{ is coinitial in }\mathsf{C}\mathbb{P}\quad\Rightarrow\quad Q^\supseteq\text{ is connected}.\]
\end{prp}

\begin{proof}
    If $\mathbb{P}$ is a regular $\omega$-poset then $\mathsf{S}\mathbb{P}$ is Hausdorff, by \autoref{RegularImpliesHausdorff}.  Thus if $Q^\supseteq$ were not connected then we would have $A,B\subseteq\mathbb{P}$ such that the corresponding open sets $O=\bigcup_{a\in A}a^\in$ and $N=\bigcup_{b\in B}b^\in$ form a disjoint minimal cover of $Q^\supseteq$.  Assuming $\mathbb{P}$ is also prime (and hence $a\wedge b$ implies $a^\in\cap b^\in\neq\emptyset$), this means $a^\geq\cap b^\geq=\emptyset$, for all $a\in A$ and $b\in B$.
    
    If $Q$ is an up-set selector and $\{C\in\mathsf{C}\mathbb{P}:Q\cap C\text{ is a cluster}\}$ is coinitial in $\mathsf{C}\mathbb{P}$, we claim that $Q\setminus(A\cup B)$ is still a selector.  Indeed, this means that any $D\in\mathsf{C}\mathbb{P}$ is refined by some $C\in\mathsf{C}\mathbb{P}$ such that $Q\cap C$ is cluster.  Then $D'=(Q\cap C)^{\leq_D}\subseteq Q\cap D$ is also a cluster so $Q\cap D\subseteq A\cup B$ would imply that $D'$ is contained in either $A$ or $B$.  Assume $D'\subseteq A$.  Take $S\in N\cap Q^\supseteq$, so we have some $b\in B\cap S$.  As $S$ is a selector, $S\cap C\neq\emptyset$ and hence we also have $a\in(S\cap C)^{\leq_D}\subseteq D'\subseteq A$.  But then $a\wedge b$, as $a,b\in S$, a contradiction.  Likewise, we get a contradiction if $D'\subseteq B$, so the only possibility is that, in fact, $Q\cap D\nsubseteq A\cup B$.  As $D$ was an arbitrary cap, this shows that $Q\setminus(A\cup B)$ is still a selector and hence contains some minimal selector $T$.  But then $T\in Q^\supseteq\setminus(O\cup N)$, again a contradiction.  Thus $Q^\supseteq$ is connected.
\end{proof}

In particular, \autoref{ConnectedClusters} and \autoref{CoinitialClusters} tell us that if $\mathbb{P}$ is a regular prime $\omega$-poset, $X\subseteq\mathsf{S}\mathbb{P}$ is closed and $\mathcal{C}=\{C\in\mathsf{C}\mathbb{P}:\bigcup X\cap C\text{ is a cluster}\}$ then
\begin{equation*}%\label{EqConnectedCoinitial}
    X\text{ is connected}\qquad\Leftrightarrow\qquad\mathcal{C}=\mathsf{C}\mathbb{P}\qquad\Leftrightarrow\qquad\mathcal{C}\text{ is coinitial in }\mathsf{C}\mathbb{P}.
\end{equation*}

Hereditarily indecomposable spaces have been a topic of much interest in continuum theory since the discovery of the pseudoarc (see \cite{Bing1948} and \cite{Moise1948}).  Here we will show how to characterise them in terms of certain `tangled' refinements.  These are more in the original spirit of Bing's crooked refinements, in contrast to the crooked covers introduced by Krasinkiewicz and Minc to characterise hereditary indecomposability (as discussed in \cite{KrasinkiewiczMinc1976}, \cite{OversteegenTymchatyn1986} and \cite{BartosKubis2022}).

First let us recall some standard terminology for a compactum $X$.  We call any closed connected $Y\subseteq X$ a \emph{subcontinuum}.  We call $X$ \emph{indecomposable} if it is not the union of two proper subcontinua.  We call $X$ \emph{hereditarily indecomposable} if every subcontinuum is indecomposable (note here we do not require $X$ itself to be connected, although that is the case of primary interest).  This is equivalent to saying that any two subcontinua of $X$ that overlap are comparable, i.e. one is contained in the other. 

This motivates the definition of a `tangled' refinement.  Specifically, we call a refinement $A\subseteq\mathbb{P}$ of $B\subseteq\mathbb{P}$ \emph{tangled} if, for all clusters $C,D\subseteq A$,
\[C\wedge D\qquad\Rightarrow\qquad C\subseteq D^{\leq_B\geq}\quad\text{or}\quad D\subseteq C^{\leq_B\geq}.\]
More explicitly, $C\subseteq D^{\leq_B\geq}$ means that every $c\in C$ shares an upper bound in $B$ with some $d\in D$, while $D\subseteq C^{\leq_B\geq}$ means that every $d\in D$ shares an upper bound in $B$ with some $c\in C$.  We denote tangled refinements by $\looparrowright$, i.e.
\[A\looparrowright B\qquad\Leftrightarrow\qquad A\text{ is a tangled refinement of }B.\]
In particular, $A\looparrowright B$ implies $A\leq B$.  Next we show that tangled refinements are auxiliary to general refinements in the sense that if $A$ is a tangled refinement of $B$ then any refinement of $A$ is a tangled refinement of any family refined by $B$.

\begin{prp}
    For any $A,A',B,B'\subseteq\mathbb{P}$,
    \[A'\leq A\looparrowright B\leq B'\qquad\Rightarrow\qquad A'\looparrowright B'.\]
\end{prp}

\begin{proof}
    Take clusters $C,D\subseteq A'$ so $C^{\leq_A}$ and $D^{\leq_A}$ are then also clusters.  If $C\wedge D$ then $C^{\leq_A}\wedge D^{\leq_A}$ and hence, by the definition of $\looparrowright$, either $C^{\leq_A}\subseteq D^{\leq_A\leq_B\geq}\subseteq D^{\leq_B\geq}$ or $D^{\leq_A}\subseteq C^{\leq_A\leq_B\geq}\subseteq C^{\leq_B\geq}$.  If $C^{\leq_A}\subseteq D^{\leq_B\geq}$ then $A'\leq A$ and $B\leq B'$ implies
    \[C\subseteq C^{\leq_A\geq}\subseteq D^{\leq_B\geq\geq}=D^{\leq_B\geq}\subseteq D^{\leq_B\leq_{B'}\geq\geq}\subseteq D^{\leq_{B'}\geq}.\]
    Likewise, if $D^{\leq_A}\subseteq C^{\leq_B\geq}$ then $D\subseteq C^{\leq_{B'}\geq}$, showing that $A'\looparrowright B'$.
\end{proof}

Let us call $P\in\mathsf{F}\mathbb{P}$ a \emph{path} if $P$ is a path graph with respect to the relation $\wedge$, which means we have an enumeration $\{p_1,\ldots,p_n\}$ of $P$ such that
\[p_j\wedge p_k\qquad\Leftrightarrow\qquad|j-k|\leq1.\]
For paths, tangled refinements can be characterised in a similar manner to the crooked refinements from \cite{Bing1948} used to construct the pseudoarc, as we now show.

Note that any cluster in a path $P$ is also a path and each pair $q,r\in P$ is contained in a unique minimal cluster/subpath, which we will denote by $[q,r]$.  We further define $[q,r)=[q,r]\setminus\{r\}$, $(q,r]=[q,r]\setminus\{q\}$ and $(q,r)=[q,r]\setminus\{q,r\}$.

\begin{prp}
    If $P,Q\subseteq\mathbb{P}$ are paths with $P\leq Q$ then
    \begin{align}
        \label{CrookedPath}P\looparrowright Q\ &\Leftrightarrow\ \forall a,d\in P\ \exists b\in[a,d]\ \exists c\in[b,d]\ \exists q,r\in Q\ (a,c\leq q\ \&\ b,d\leq r).\\
        \label{TangledPath}&\Leftrightarrow\ \forall a,d\in P\ \exists b\in[a,d]\ \exists c\in[b,d]\ ([a,d]\subseteq[a,b]^{\leq_Q\geq}\cap[c,d]^{\leq_Q\geq}).
    \end{align}
\end{prp}

\begin{proof}
    Assume the right side of \eqref{CrookedPath} holds.  To show that \eqref{TangledPath} holds, take any $a,d\in P$.  Then we have $g,h\in Q$ with $[a,d]^\leq=[g,h]$ and we may pick $a',d'\in[a,d]$ with $a'\leq g$ and $d'\leq h$.  If $p^{\leq_Q}=\{g\}$, for some $p\in[a,d]$ then we may further ensure that $a'^{\leq_Q}=\{g\}$ and, likewise, if $p^{\leq_Q}=\{h\}$, for some $p\in[a,d]$ then we may further ensure that $d'^{\leq_Q}=\{h\}$.  By \eqref{TangledPath}, we have $b\in[a',d']$ and $c\in[b,d']$ with $a',c\leq q$ and $b,d'\leq r$, for some $q,r\in Q$.  If $a'^{\leq_Q}=\{g\}$ then $q=g$ and hence $[c,d']^{\leq_Q\geq}=[g,h]^\geq\supseteq[a,d]$.  On the other hand, if $a'^{\leq_Q}\neq\{g\}$ then $p^{\leq_Q}\neq\{g\}$ and hence $p^\leq\cap(g,h]\neq\emptyset$, for all $p\in[a,d]$, which again yields $[c,d']^{\leq_Q\geq}\supseteq(g,h]^\geq\supseteq[a,d]$.  Likewise, we see that $[a',b]^{\leq_Q\geq}\supseteq[a,d]$.  Expanding $[a',b]$ and $[c,d']$ to include $a$ and $d$ then shows that \eqref{TangledPath} holds.

    Now assume \eqref{TangledPath} holds and take any clusters/subpaths $A,D\subseteq P$ with $A\wedge D$.  This means $A\cup D=[a,d]$, for some $a,d\in P$, so we have $b\in[a,d]$, $c\in[b,d]$ satisfying \eqref{TangledPath}.  It follows that $A$ or $D$ must contain $[a,b]$ or $[c,d]$ and hence that $D\subseteq A^{\leq_Q\geq}$ or $A\subseteq D^{\leq_Q\geq}$, e.g. if $[a,b]\subseteq A$ then $D\subseteq[a,d]\subseteq[a,b]^{\leq_Q\geq}\subseteq A^{\leq_Q\geq}$.  This shows that $P\looparrowright Q$.

    Finally, assume that $P\looparrowright Q$ and take any $a,d\in P$.  Then we have $b\in[a,d]$ such that $b$ shares an upper bound in $Q$ with $d$ but no element of $[a,b)$ does.  As $P\looparrowright Q$, this implies $[a,b)\subseteq[b,d]^{\leq_Q\geq}$ and, in particular, we have some $c\in[b,d]$ sharing an upper bound in $Q$ with $a$ (because $a\in[a,b)$, as long as $a\neq b$, while if $a=b$ then we can just take $c=a$ too).  This shows that the right side of \eqref{CrookedPath} holds.
\end{proof}

In the next result it will be convenient to consider a slightly weakening of $\looparrowright$.  Specifically, let us call a refinement $A\subseteq\mathbb{P}$ of $B\subseteq\mathbb{P}$ \emph{weakly tangled}, denoted $A\looparrowright_\mathsf{w}B$, if, for all clusters $C,D\subseteq A$,
\[C\cap D\neq\emptyset\qquad\Rightarrow\qquad C\subseteq D^{\leq_B\geq}\quad\text{or}\quad D\subseteq C^{\leq_B\geq}.\]
We call $\mathbb{P}$ \emph{(weakly) tangled} if every cap has a (weakly) tangled refinement, i.e.
\[\tag{(Weakly) Tangled}C\in\mathsf{C}\mathbb{P}\qquad\Rightarrow\qquad\exists D\in\mathsf{C}\mathbb{P}\ (D\looparrowright_{(\mathsf{w})}C).\]
We immediately see that $A\vartriangleleft B\looparrowright_\mathsf{w}C$ implies $A\looparrowright C$ so
\[\mathbb{P}\text{ is regular and weakly tangled}\qquad\Rightarrow\qquad\mathbb{P}\text{ is tangled}.\]
The following result thus tells us that, among prime regular $\omega$-posets, those with hereditarily indecomposable spectra are precisely the tangled posets.

\begin{thm}
    If $\mathbb{P}$ is an $\omega$-poset then
    \[\mathbb{P}\text{ is weakly tangled}\qquad\Rightarrow\qquad\mathsf{S}\mathbb{P}\text{ is hereditarily indecomposable}.\]
    The converse holds if $\mathbb{P}$ is also regular and prime.
\end{thm}

\begin{proof}
    Assume $\mathsf{S}\mathbb{P}$ is not hereditarily indecomposable, so we have overlapping incomparable subcontinua $Y,Z\subseteq\mathsf{S}\mathbb{P}$.  So we can take $S\in Y\setminus Z$ and $T\in Z\setminus Y$ and obtain a minimal open cover of $\mathsf{S}\mathbb{P}$ consisting of the sets $\mathsf{S}\mathbb{P}\setminus Y$, $\mathsf{S}\mathbb{P}\setminus Z$ and $\mathsf{S}\mathbb{P}\setminus\{S,T\}$.  This is refined by some basic cover $(c^\in)_{c\in C}$, necessarily with $C\in\mathsf{C}\mathbb{P}$, by \autoref{SpectrumCompactT1}.  Now take $D\in\mathsf{C}\mathbb{P}$ with $D\leq C$.  By \autoref{ConnectedClusters}, we have clusters $A=D\cap\bigcup Y$ and $B=D\cap\bigcup Z$, necessarily with $A\cap B\neq\emptyset$, as $Y\cap Z\neq\emptyset$.  We also have $a\in D\cap S\subseteq A$ and $b\in D\cap T\subseteq B$.  Taking any $c\in C$ with $a\leq c$, we see that $S\in a^\in\subseteq c^\in$ and so $c^\in\nsubseteq\mathsf{S}\mathbb{P}\setminus Y$ and $c^\in\nsubseteq\mathsf{S}\mathbb{P}\setminus\{S,T\}$, the only remaining option then being $c^\in\subseteq\mathsf{S}\mathbb{P}\setminus Z$.
    But whenever $B\ni b'\leq c'\in C$, we see that $b'\in\bigcup Z$, so we have $U\in Z$ with $b'\in U$ and hence $U\in b'^\in\subseteq c'^\in$.  Thus implies $c'^\in\nsubseteq\mathsf{S}\mathbb{P}\setminus Z$ and hence $c'\neq c$.  Likewise, we see that $b$ has no common upper bound in $C$ with any element of $A$.  This shows that $D$ is not a weakly tangled refinement of $C$.  As $D$ was arbitrary, this shows that $\mathbb{P}$ is not a weakly tangled poset, thus proving $\Rightarrow$.

    Conversely, assume $\mathbb{P}$ is regular and prime but not weakly tangled, so we have $C\in\mathsf{B}\mathbb{P}$ such that $D\not\looparrowright_\mathsf{w} C$, for all $D\in\mathsf{C}\mathbb{P}$.  Take a coinital decreasing sequence $(C_n)\subseteq\mathsf{B}\mathbb{P}$ with $C_0\leq C$.  As $C_n\not\looparrowright_\mathsf{w}C$, we have overlapping clusters $A_n,B_n\subseteq C_n$ such that $A_n\nsubseteq B_n^{\leq_C\geq}$ and $B_n\nsubseteq A_n^{\leq_C\geq}$.  Taking a subsequence if necessary, we can obtain clusters $A,B\subseteq C$ with $A_n^{\leq_C}=A$ and $B_n^{\leq_C}=B$, for all $n\in\omega$.  Taking further subsequences if necessary, we may assume we have clusters $D_m,E_m\subseteq C_m$ with $A_n^{\leq_{C_m}}=D_m$ and $B_n^{\leq_{C_m}}=E_m$ whenever $m<n$.  Note $D_n\subseteq B^\geq$ would imply
    \[A_{n+1}\subseteq A_{n+1}^{\leq_{C_n}\geq}=D_n^\geq\subseteq B^{\geq\geq}\subseteq B^\geq=B_{n+1}^{\leq_C\geq},\]
    a contradiction.  Thus $D_n\nsubseteq B^\geq$ and, likewise, $E_n\nsubseteq A^\geq$, for all $n\in\omega$.

    Now note that $Q=\bigcap_{m\in\omega}\bigcup_{n>m}A_n^\leq$ is an up-set such that $Q\cap C=A_n^{\leq_C}=A$ and $Q\cap C_n=A_{n+1}^{\leq_{C_n}}=D_n$, for all $n\in\omega$.  It follows that $Q^\supseteq$ is connected, by \autoref{CoinitialClusters}.  Likewise, we have an up-set $R=\bigcap_{m\in\omega}\bigcup_{n>m}B_n^\leq$ such that $R\cap C=B$ and $R^\supseteq$ is connected.  Also $Q'=\bigcup_{n\in\omega}(D_n\setminus B^\geq)^\leq\subseteq Q\setminus B$ is a selector and hence contains a minimal selector in $Q^\supseteq\setminus R^\supseteq$, seeing as $R\cap C=B$.  Likewise, $R'=\bigcup_{n\in\omega}(E_n\setminus A^\geq)^\leq\subseteq R\setminus A$ contains a minimal selector in $R^\supseteq\setminus Q^\supseteq$.  Lastly note that $\emptyset\neq D_n\cap E_n\subseteq Q\cap R$, for all $n\in\omega$, so $Q\cap R$ also contains a minimal selector in $Q^\supseteq\cap R^\supseteq$.  Thus $Q^\supseteq$ and $R^\supseteq$ are incomparable overlapping subcontinua and hence $\mathsf{S}\mathbb{P}$ is not hereditarily indecomposable.
\end{proof}

\section{Functoriality}\label{Functoriality}

Here we examine order theoretic analogs of continuous maps, using these to obtain a more combinatorial equivalent of the usual category of metrisable compacta.

\begin{center}
\textbf{Throughout this section, fix some posets $\mathbb{P}$, $\mathbb{Q}$, $\mathbb{R}$ and $\mathbb{S}$.}
\end{center}
For extra clarity, we will sometimes use subscripts to indicate which poset we are referring to, e.g. $\leq_\mathbb{P}$ and $\leq_\mathbb{Q}$ refer to order relations on $\mathbb{P}$ and $\mathbb{Q}$ respectively.

\subsection{Continuous Maps}

\begin{dfn}
We call $\sqsupset\ \subseteq\mathbb{Q}\times\mathbb{P}$ a \emph{refiner} if
\[\tag{Refiner}\mathsf{C}\mathbb{Q}\subseteq\mathsf{C}\mathbb{P}^\sqsubset,\]
i.e. if each cap of $\mathbb{Q}$ is refined by some cap of $\mathbb{P}$.
\end{dfn}

For example, in this terminology a poset is regular precisely when the star-above relation $\vartriangleright\ \subseteq\mathbb{P}\times\mathbb{P}$ is a refiner.

We can use refiners to encode continuous maps as follows.

\begin{prp} \label{MapToRefiner}
If $\mathbb{P}$ is an $\omega$-poset and $\phi:\mathsf{S}\mathbb{P}\rightarrow\mathsf{S}\mathbb{Q}$ is continuous then
\begin{equation}\label{sqphi}
q\sqsupset_\phi p\qquad\Leftrightarrow\qquad\phi^{-1}[q^\in]\supseteq p^\in
\end{equation}
defines a refiner $\sqsupset_\phi\ \subseteq\mathbb{Q}\times\mathbb{P}$
such that $S^{\sqsubset_\phi} = \phi(S)$ for every $S \in \Spec{\PP}$.
\end{prp}

\begin{proof}
Any $C\in\mathsf{C}\mathbb{Q}$ defines a cover $C_\mathsf{S}$ of $\mathsf{S}\mathbb{Q}$, which in turn yields a cover $(\phi^{-1}[c^\in])_{c\in C}$ of $\mathsf{S}\mathbb{P}$.  If $\mathbb{P}$ is an $\omega$-poset then $\mathbb{P}_\mathsf{S}$ is a basis for $\mathsf{S}\mathbb{P}$, by \autoref{BasisOrderIsomorphism}, so we have $B\subseteq\mathbb{P}$ such that $B_\mathsf{S}$ refines $(\phi^{-1}[c^\in])_{c\in C}$, with respect to inclusion, and hence $B$ refines $C$, with respect to ${\sqsubset_\phi}={\sqsupset_\phi^{-1}}$, i.e. $B\sqsubset_\phi C$.  By \autoref{SpectrumCompactT1}, $B$ is a cap of $\mathbb{P}$, so this shows that $\sqsupset_\phi$ is indeed a refiner.

If $q \in S^{\sqsubset_\phi}$ then there is some $p \in S$ with $\phi^{-1}[q^\in] \supseteq p^\in \ni S$ so $\phi(S) \in q^\in$ and hence $q \in \phi(S)$.
On the other hand, if $q \in \phi(S)$, i.e. $\phi(S) \in q^\in$, then by continuity there is some $p \in S$ such that $\phi^{-1}[q^\in] \supseteq p^\in$.
Hence $q \sqsupset_\phi p \in S$ so $q \in S^{\sqsubset_\phi}$.
\end{proof}

A relation $\sqsupset\ \subseteq\mathbb{P}\times\mathbb{Q}$ is \emph{$\wedge$-preserving} if, for all $p,p'\in\mathbb{P}$ and $q,q'\in\mathbb{Q}$,
\begin{equation}
    \tag{$\wedge$-Preservation}q\sqsupset p\wedge p'\sqsubset q'\qquad\Rightarrow\qquad q\wedge q'.
\end{equation}
As long as $\mathbb{P}$ is prime and $\mathbb{Q}$ is an $\omega$-poset, the refiner $\sqsupset_\phi$ defined in \eqref{sqphi} will also be $\wedge$-preserving.  Indeed, if $\mathbb{P}$ is prime then, for any $p,p'\in\mathbb{P}$ and $s\leq p,p'$, we have $S\in\mathsf{S}\mathbb{P}$ containing $s$.  If $q\sqsupset_\phi p$ and $q'\sqsupset_\phi p'$ then $q,q'\in\phi(S)$ and hence $q\wedge q'$, assuming $\mathbb{Q}$ is an $\omega$-poset, by \autoref{MinimalSelectorsAreFilters}.

Conversely, as long as we restrict to regular posets (and hence Hausdorff spectra), we can define continuous maps from $\wedge$-preserving refiners.

\begin{prp} \label{RefinerToMap}
If $\mathbb{P}$ is an $\omega$-poset and $\mathbb{Q}$ is a regular poset then any $\wedge$-preserving refiner $\sqsupset\ \subseteq\mathbb{Q}\times\mathbb{P}$ defines a continuous map $\phi_\sqsupset:\mathsf{S}\mathbb{P}\rightarrow\mathsf{S}\mathbb{Q}$ by
\begin{equation}\label{phisq}
\phi_\sqsupset(S)=S^{\sqsubset\vartriangleleft}.
\end{equation}
If $\mathbb{Q}$ and $\mathbb{R}$ are also regular $\omega$-posets and $\sqni\ \subseteq\mathbb{R}\times\mathbb{Q}$ is another $\wedge$-preserving refiner,
\begin{equation}\label{CompositionPreservation}
    \phi_{\sqni}\circ\phi_\sqsupset=\phi_{\sqni\circ\sqsupset}.
\end{equation}
\end{prp}

\begin{proof}
For any $S\in\mathsf{S}\mathbb{P}$, we see that $S^\sqsubset$ is a linked selector, as $\sqsupset$ is a $\wedge$-preserving refiner.  Thus $S^{\sqsubset\vartriangleleft}\in\mathsf{S}\mathbb{Q}$, by \autoref{LinkedSelector}, showing 
 that $\phi_\sqsupset$ maps $\mathsf{S}\mathbb{P}$ to $\mathsf{S}\mathbb{Q}$.  For continuity just note that $\phi_\sqsupset^{-1}[q^\in]=\bigcup_{p\sqsubset\circ\vartriangleleft q}p^\in$ is open, for any $q\in\mathbb{Q}$.

Next note that the larger subset $S^\sqsubset\cup S^{\sqsubset\vartriangleleft}$ is still linked, because $\sqsubset$ is $\wedge$-preserving and $\wedge\ \circ\vartriangleleft\ \subseteq\wedge$, by \eqref{WedgeStar}.  If $\mathbb{Q}$ and $\mathbb{R}$ are also regular $\omega$-posets and $\sqni\ \subseteq\mathbb{R}\times\mathbb{Q}$ is another $\wedge$-preserving refiner then it follows that $S^{\sqsubset\sqin}$ and $S^{\sqsubset\vartriangleleft\sqin}$ are again selectors with linked union.  Now, for any $C\in\mathsf{C}\mathbb{R}$, we have $A,B\in\mathsf{C}\mathbb{R}$ with $A\vartriangleleft_AB\vartriangleleft C$.  We then have $a\in A\cap S^{\sqsubset\sqin}$ and $a'\in A\cap S^{\sqsubset\vartriangleleft\sqin}$, necessarily with $a\wedge a'$.  We then also have $b\in B$ and $c\in C$ with $a,a'\leq b\vartriangleleft c$ and hence $c\in S^{\sqsubset\sqin\vartriangleleft}\cap S^{\sqsubset\vartriangleleft\sqin\vartriangleleft}\cap C$.  This shows that $S^{\sqsubset\sqin\vartriangleleft}\cap S^{\sqsubset\vartriangleleft\sqin\vartriangleleft}$ is a selector and hence, by the minimality of $\phi_{\sqni\circ\sqsupset}(S)=S^{\sqsubset\sqin\vartriangleleft}$ and $\phi_{\sqni}\circ\phi_\sqsupset(S)=S^{\sqsubset\vartriangleleft\sqin\vartriangleleft}$, it follows that $\phi_{\sqni}\circ\phi_\sqsupset(S)=\phi_{\sqni\circ\sqsupset}(S)$.
\end{proof}

Let $\mathbf{K}$ denote the category of metrisable compact spaces and continuous maps, and let $\mathbf{P}$ denote the category of regular prime $\omega$-posets and $\wedge$-preserving refiners (note that these are closed under composition and that $\id_\PP$ is always a $\wedge$-preserving refiner).  We already have a map $\mathsf{S}$ from objects $\mathbb{P}\in\mathbf{P}$ to $\mathsf{S}\mathbb{P}\in\mathbf{K}$ and we extend this to morphisms ${\sqsupset}\in\mathbf{P}^\mathbb{Q}_\mathbb{P}(=$ refiners in $\mathbb{Q}\times\mathbb{P}$) by setting $\mathsf{S}(\sqsupset)=\phi_\sqsupset\in\mathbf{K}_{\mathsf{S}\mathbb{P}}^{\mathsf{S}\mathbb{Q}}(=$ continuous maps from $\mathsf{S}\mathbb{P}$ to $\mathsf{S}\mathbb{Q}$ -- in general, for any objects $A$ and $B$ of a category $\mathbf{C}$, we denote the corresponding hom-set by $\mathbf{C}_A^B=\{m:m\text{ is a morphism from $A$ to $B$}\}$).

The previous results can thus be summarised as follows.

\begin{thm}\label{SFunctor}
    The map $\mathsf{S}:\mathbf{P}\rightarrow\mathbf{K}$ is an essentially surjective full functor.
\end{thm}

\begin{proof}
    For every $S \in \Spec{\PP}$ we have $\phi_{\id_\PP}(S) = S^\vartriangleleft = S$ since every minimal selector in a regular $\omega$-poset is round, by \autoref{RoundLinked}, and so $\mathsf{S}(\id_\PP) = \id_{\Spec{\PP}}$ for every $\PP\in\mathbf{P}$.
    Together with \eqref{CompositionPreservation}, this shows that $\mathsf{S}$ is a functor.

    Moreover, $\mathsf{S}$ is essentially surjective because every metrizable compactum $X$ is homeomorphic to $\Spec{\PP}$ for some cap-determined $\omega$-poset $\PP$, by \autoref{SpacesFromGradedPosets}, which is necessarily prime, by \eqref{CapDeterminedPrime}, and regular, by \autoref{RegularImpliesHausdorff}.

    The functor is full by \autoref{MapToRefiner} since, for every pair of prime regular $\omega$-posets $\PP$ and $\QQ$ and every continuous map $\phi\maps \Spec{\PP} \to \Spec{\QQ}$, we have $\phi = \mathsf{S}(\sqsupset_\phi)$ (because $\phi_{\sqsupset_\phi}(S)=S^{\sqsubset_\phi\vartriangleleft}=S^{\sqsubset_\phi}$, as $S^{\sqsubset_\phi}$ is already round, by \autoref{MapToRefiner} and \autoref{RoundLinked}).
    The refiner $\sqsupset_\phi$ is $\wedge$-preserving since our $\omega$-posets are prime.
\end{proof}

\begin{rmk}
    We could turn the above result into an equivalence of categories by simply identifying ${\sqsupset},{\sqni}\in\mathbf{P}^\mathbb{Q}_\mathbb{P}$ whenever $\phi_\sqsupset=\phi_{\sqni}$.  Then $\mathsf{S}$ factors as $\mathsf{E}\circ\mathsf{Q}$, where $\mathsf{E}$ is an equivalence and $\mathsf{Q}$ is the quotient functor.  However, what we would really like is a more combinatorial formulation of the quotient category.  We will achieve this in \autoref{StarComposition} via a certain category $\mathbf{S}$ with the same objects as $\mathbf{P}$ but more restrictive `strong refiners' as morphisms under a modified `star-composition'.
\end{rmk}

With \eqref{sqphi} in mind, one might expect that $q\sqsupset p$ is equivalent to $\phi_\sqsupset^{-1}[q^\in]\supseteq p^\in$.  However, both implications may fail, even for $\wedge$-preserving refiners on regular $\omega$-posets.  The best we can do at this stage is show two weaker relations are equivalent.

\begin{prp}
Whenever ${\sqsupset}\in\mathbf{P}^\mathbb{Q}_\mathbb{P}$,
\[q\sqsupset p\qquad\Rightarrow\qquad q^\wedge\supseteq p^{\wedge\sqsubset\vartriangleleft}\qquad\Leftrightarrow\qquad\mathrm{cl}(q^\in)\supseteq\phi_\sqsupset[p^\in].\]
\end{prp}

\begin{proof}
    For the first $\Rightarrow$, just note that $q\sqsupset p$ implies $p^{\wedge\sqsubset\vartriangleleft}\subseteq q^{\wedge\vartriangleleft}\subseteq q^\wedge$.

    If $q^\wedge\supseteq p^{\wedge\sqsubset\vartriangleleft}$ then, for any $S\in p^\in$,
    \[\phi_\sqsupset(S)=S^{\sqsubset\vartriangleleft}\subseteq p^{\wedge\sqsubset\vartriangleleft}\subseteq q^\wedge=\bigcup q^\in,\]
    by \eqref{Upni}, and hence $\phi_\sqsupset(S)\in\mathrm{cl}(q^\in)$, by \eqref{Closure}.  This proves the $\Rightarrow$ part.

    Conversely, say $\mathrm{cl}(q^\in)\supseteq\phi_\sqsupset[p^\in]$ and take $r\in p^{\wedge\sqsubset\vartriangleleft}$.  Then we have $s\in p^\wedge\cap r^{\vartriangleright\sqsupset}$ and $S\in p^\in\cap s^\in$, as $\mathbb{P}$ is prime (see \eqref{Upni}), necessarily with $r\in s^{\sqsubset\vartriangleleft}\subseteq S^{\sqsubset\vartriangleleft}=\phi_\sqsupset(S)$.  As $\phi_\sqsupset(S)\in\phi_\sqsupset[p^\in]\subseteq\mathrm{cl}(q^\in)$, it follows that $r\in\phi_\sqsupset(S)\subseteq\bigcup q^\in=q^\wedge$, by \eqref{Closure} and \eqref{Upni} again.  This shows that $q^\wedge\supseteq p^{\wedge\sqsubset\vartriangleleft}$, as required.
\end{proof}

By \autoref{ClosedContainmentProp}, $q\vartriangleright r\sqsupset p$ then implies $q^\in\supseteq\mathrm{cl}(r^\in)\supseteq\mathrm{cl}(\phi_\sqsupset[p^\in])$ and hence $\phi_\sqsupset^{-1}[q^\in]\supseteq\phi_\sqsupset^{-1}[\mathrm{cl}(\phi_\sqsupset[p^\in])]\supseteq\mathrm{cl}(p^\in)$, i.e.
\begin{equation}\label{StarCircSq}
    q\vartriangleright\circ\sqsupset p\qquad\Rightarrow\qquad\phi_\sqsupset^{-1}[q^\in]\supseteq\mathrm{cl}(p^\in).
\end{equation}
Later we will show how to turn this into an equivalence using star-composition.

\subsection{Homeomorphisms}\label{Homeomorphisms}

By \autoref{SFunctor}, isomorphisms in $\mathbf{P}$ yield homeomorphisms in $\mathbf{K}$.  We can also obtain homeomorphisms of spectra from much more general pairs of refiners, even between non-regular posets.

Let ${\succsim_\mathbb{P}}\subseteq\mathbb{P}\times\mathbb{P}$ be the cap-order ${\succsim}\subseteq\mathsf{P}\mathbb{P}\times\mathsf{P}\mathbb{P}$ restricted to singletons, i.e.
\[p\succsim_\mathbb{P}p'\qquad\Leftrightarrow\qquad\{p\}\succsim\{p'\}\qquad\Leftrightarrow\qquad p^\in\supseteq p'^\in\]
(see \eqref{AprecsimB}).  Likewise define ${\succsim_\mathbb{Q}}\subseteq\mathbb{Q}\times\mathbb{Q}$ and let ${\precsim_\mathbb{P}}={\succsim_\mathbb{P}^{-1}}$ and ${\precsim_\mathbb{Q}}={\succsim_\mathbb{Q}^{-1}}$.  If these have subrelations coming from compositions of a pair of refiners between them then these refiners yield mutually inverse homeomorphisms between their spectra.

\begin{prp}\label{BirefinableImpliesHomeomorphic}
If $\sqsupset\ \subseteq\mathbb{Q}\times\mathbb{P}$ and $\sqni\ \subseteq\mathbb{P}\times\mathbb{Q}$ are refiners satisfying 
\[\sqsupset\circ\sqni\ \subseteq\ \succsim_\mathbb{Q}\qquad\text{and}\qquad\sqni\circ\sqsupset\ \subseteq\ \succsim_\mathbb{P}\]
then $S\mapsto S^\sqsubset$ and $T\mapsto T^{\sqin}$ are continuous maps between $\mathsf{S}\mathbb{P}$ and $\mathsf{S}\mathbb{Q}$ satisfying
\[S=S^{\sqsubset\sqin}\qquad\text{and}\qquad T=T^{\sqin\sqsubset}.\]
\end{prp}
\begin{proof}
First note $S^\sqsubset$ is a selector in $\mathbb{P}$ whenever $S$ is a selector in $\mathbb{Q}$. Indeed, for any $D\in\mathsf{C}\mathbb{Q}$, we have $C\in\mathsf{C}\mathbb{P}$ with $C\sqsubset D$, as $\sqsupset$ is a refiner.  As $S$ is a selector, we have $c\in S\cap C$.  We then have $d\in D$ with $c\sqsubset d$ and hence $d\in S^\sqsubset\cap D$.

Likewise, any selector $T$ in $\mathbb{Q}$ gives rise to a selector $T^{\sqin}$ in $\mathbb{P}$, which in turn yields another selector $T^{\sqin\sqsubset}=T^{\sqin\circ\sqsubset}\subseteq T^{\precsim_\mathbb{Q}}$ in $\mathbb{Q}$.  If $T$ is a minimal selector then $T^{\precsim_\mathbb{Q}}\subseteq T$, by \eqref{AprecsimB}, and hence $T^{\sqin\sqsubset}=T$.  Moreover, $T^{\sqin}$ contains some minimal selector $S$, by \autoref{MinimalSelectors}.  It follows that $S^\sqsubset\subseteq T^{\sqin\sqsubset}=T$ which implies $S^\sqsubset=T$, by minimality.  This in turn implies $S=S^{\sqsubset\sqin}=T^{\sqin}$, i.e. $T^{\sqin}$ was already minimal.  This shows that $S\mapsto S^\sqsubset$ and $T\mapsto T^{\sqin}$ are mutually inverse bijections.  Lastly, note that the preimage of any subbasic open set $q^\in$ with respect to the map $S\mapsto S^\sqsubset$ is given by $\bigcup_{p\sqsubset q}p^\in$, which is again open, showing that $S\mapsto S^\sqsubset$ is continuous.  Likewise, $T\mapsto T^{\sqin}$ is also continuous, as required.
\end{proof}

We can also obtain a kind of converse to \autoref{BirefinableImpliesHomeomorphic} by noting that
\begin{equation*}%\label{sqfunctor}
{\sqsupset_\psi\circ\sqsupset_\phi}\ \subseteq\ {\sqsupset_{\psi\circ\phi}},
\end{equation*}
for any $\phi:\mathsf{S}\mathbb{P}\rightarrow\mathsf{S}\mathbb{Q}$ and $\psi:\mathsf{S}\mathbb{Q}\rightarrow\mathsf{S}\mathbb{R}$, as $r\sqsupset_\psi q\sqsupset_\phi p$ means $\psi^{-1}[r^\in]\supseteq q^\in$ so
\[(\psi\circ\phi)^{-1}[r^\in]=\phi^{-1}[\psi^{-1}[r^\in]]\supseteq\phi^{-1}[q^\in]\supseteq p^\in.\]
In particular, if $\mathbb{P}$ and $\mathbb{Q}$ are $\omega$-posets and $\phi:\mathsf{S}\mathbb{P}\rightarrow\mathsf{S}\mathbb{Q}$ is a homeomorphism then $\sqsupset_{\phi^{-1}}\circ\sqsupset_\phi\ \subseteq\ \sqsupset_{\mathrm{id}_{\mathsf{S}\mathbb{P}}}$ and $\sqsupset_\phi\circ\sqsupset_{\phi^{-1}}\ \subseteq\ \sqsupset_{\mathrm{id}_{\mathsf{S}\mathbb{Q}}}$.  But $q\sqsupset_{\mathrm{id}_{\mathsf{S}\mathbb{P}}}p$ just means $q^\in\supseteq p^\in$, which is equivalent to $q\succsim p$, by \eqref{AprecsimB}.  So this shows that
\[{\sqsupset_{\phi^{-1}}\circ\sqsupset_\phi}\,\subseteq\, {\succsim_\mathbb{P}}\qquad\text{and}\qquad{\sqsupset_\phi\circ\sqsupset_{\phi^{-1}}}\,\subseteq\, {\succsim_\mathbb{Q}}.\]

The following corollary of \autoref{CapsRefineLevels} shows that subposets of an $\omega$-poset containing infinitely many of its levels all have homeomorphic spectra.

\begin{cor}\label{HomeomorphicSubposet}
    If $\mathbb{P}$ is an $\omega$-poset, $\mathbb{Q}\subseteq\mathbb{P}$ and $\mathsf{P}\mathbb{Q}\cap\mathsf{B}\mathbb{P}$ is coinitial in $\mathsf{B}\mathbb{P}$ then $S\mapsto S\cap\mathbb{Q}$ and $T\mapsto T^\leq$ are continuous maps between $S\in\mathsf{S}\mathbb{P}$ and $T\in\mathsf{S}\mathbb{Q}$ satisfying
    \[(S\cap\mathbb{Q})^\leq=S\qquad\text{and}\qquad(T^\leq\cap\mathbb{Q})=T.\]
\end{cor}

\begin{proof}
    We claim that the caps of $\mathbb{Q}$ are precisely the caps of $\mathbb{P}$ contained in $\mathbb{Q}$, i.e.
    \[\mathsf{C}\mathbb{Q}=\mathsf{C}\mathbb{P}\cap\mathsf{P}\mathbb{Q}=\{C\in\mathsf{C}\mathbb{P}:C\subseteq\mathbb{Q}\}.\]
    Indeed, if $\mathsf{C}\mathbb{P}\ni C\subseteq\mathbb{Q}$ then, as $\mathbb{Q}$ contains a coinitial subset of $\mathsf{B}\mathbb{P}$, we have some $B\in\mathsf{B}\mathbb{P}\cap\mathsf{P}\mathbb{Q}\subseteq\mathsf{B}\mathbb{Q}$ refining $C$ and hence $C\in\mathsf{C}\mathbb{Q}$.  Conversely, take some $B\in\mathsf{B}\mathbb{Q}$.  For sufficiently large $n\in\omega$, the cone $\mathbb{P}^n$ will contain $B$ and hence the level $\mathbb{P}_n$ will be disjoint from $B^<$.  As $\mathbb{Q}$ contains a coinitial subset of $\mathsf{B}\mathbb{P}$, we have $C\in\mathsf{B}\mathbb{P}\cap\mathsf{P}\mathbb{Q}\subseteq\mathsf{B}\mathbb{Q}$ refining $\mathbb{P}_n$ which is thus also disjoint from $B^<$.  But $B$ is a band of $\mathbb{Q}$ so this implies that $C\subseteq B^\geq$, i.e. $C$ refines $B$ and hence $B$ is also a cap of $\mathbb{P}$.  This shows that all bands of $\mathbb{Q}$ are caps of $\mathbb{P}$ and hence the same applies to caps of $\mathbb{Q}$ as well, proving the claim.

    Thus the restrictions $\geq_\PP^\QQ$ and $\geq_\QQ^\PP$ of $\geq_\PP$ to $\QQ \times \PP$ and $\PP \times \QQ$ are refiners satisfying ${\geq}^\QQ_\PP \circ {\geq}^\PP_\QQ \subseteq {\geq_\QQ} \subseteq {\succsim_\QQ}$ and ${\geq}^\PP_\QQ \circ {\geq}^\QQ_\PP \subseteq {\geq_\PP} \subseteq {\succsim_\PP}$.  Noting $S^{\leq^\mathbb{Q}_\mathbb{P}}=S\cap\mathbb{Q}$ and $T^{\leq^\mathbb{P}_\mathbb{Q}}=T^\leq$, for all $S \in \Spec{\PP}$ and $T \in \Spec{\QQ}$, the result now follows from \autoref{BirefinableImpliesHomeomorphic}.
\end{proof}

We can also obtain a similar order theoretic analog of \autoref{PredeterminedSubBasis}.  First, we need the following order theoretic analog of \autoref{PredeterminedLemma}.  Let
\[\mathbb{P}_\emptyset=\{p\in\mathbb{P}:p\precsim\emptyset\}.\]
Also note $D\precsim_\mathbb{P}B$ below is saying that $D$ refines $B$, with respect to the $\precsim_\mathbb{P}$ relation on $\mathbb{P}$ (which is stronger than just saying $D\precsim B$, for the relation $\precsim$ on $\mathsf{P}\mathbb{P}$).

\begin{lemma} \label{PredeterminedLemma2}
    If $\mathbb{P}$ is an $\omega$-poset, $B$ is a cap and $C$ is a finite subset of $\mathbb{P}\setminus\mathbb{P}_\emptyset$ on which $\precsim_\mathbb{P}$ is just $\leq$ then there is a minimal cap $D\precsim_\mathbb{P}B$ and minimal caps $(E_d)_{d\in D}$ with $d\in E_d$ such that, for all $c\in C$ and $d\in D$,
    \begin{equation}\label{dbigcap2}
        c\not\precsim E_d\setminus\{d\}\ \Rightarrow\ d\leq c\qquad\text{and}\qquad c\precsim d\ \Rightarrow\ c=d\text{ and }c^\precsim\in\mathsf{S}\mathbb{P},
    \end{equation}
\end{lemma}

\begin{proof}
    For all $F\subseteq C$, we will recursively define $D_F\subseteq B^{\succsim_\mathbb{P}}\cap\bigcap_{f\in F}f^\geq$ such that $E_F=D'_F\cup(C\setminus F)$ is a cap, where $D'_F=\bigcup_{G\subseteq F}D_G$ is minimal with this property (incidentally, $D_F$ can be empty for many $F\subseteq C$).  In particular, $D=D'_C$ is a minimal cap.  Also, if $d\in D_F$ then $C\ni c\not\precsim E_F\setminus\{d\}\supseteq C\setminus F$ implies $c\in F$ and hence $d\in D_F\leq c$, proving \eqref{dbigcap2} when we take $E_d=E_F$.
        
    To perform the recursive construction, first note every $c\in C$ is contained in a minimal cap, as $C\cap\mathbb{P}_\emptyset=\emptyset$.  As $\mathbb{P}$ is an $\omega$-poset, these have a common refinement with $B$ in $\mathsf{C}\mathbb{P}$, which then refines $C\cup(B^\geq\setminus C^{\precsim_\mathbb{P}})$.  So this must also be a cap and we can then let $D_\emptyset$ be any minimal subset of $B^\geq\setminus C^{\precsim_\mathbb{P}}$ such that $C\cup D_\emptyset$ is a cap.
    
    Once $D_G$ has been defined, for $G\subsetneqq F$, note that, for each $f\in F$, we have a cap $E_{F\setminus\{f\}}=D'_{F\setminus\{f\}}\cup\{f\}\cup(C\setminus F)$.  Each $c\in C\setminus F$ is also again contained in a minimal cap.  These have a common refinement with $B$ in $\mathsf{C}\mathbb{P}$, necessarily refining
    \[E'_F=E''_F\cup((B^\geq\cap\bigcap_{f\in F}f^\geq)\setminus(C\setminus F)^{\precsim_\mathbb{P}}),\qquad\text{where}\qquad E''_F=\bigcup_{G\subsetneqq F}D_G\cup(C\setminus F).\]
    Thus $E'_F$ is a cap.  Now say $c^{\precsim_\mathbb{P}}$ is a selector, for some $c\in F$.  If $c\precsim E''_F$ then $E'_F\precsim E''_F$ so $E''_F$ is also a cap and we may set $D_F=\emptyset$.  On the other hand, if $c\not\precsim E''_F$ then, in particular, $c\not\precsim\emptyset$ so $c^{\precsim_\mathbb{P}}\in\mathsf{S}\mathbb{P}$ and $c^{\precsim_\mathbb{P}}\cap B^\geq\cap\bigcap_{f\in F}f^\geq\neq\emptyset$, as $E'_F$ is a cap.  This means $c\in B^{\succsim_\mathbb{P}}\cap\bigcap_{f\in F}f^\geq$, as $\precsim_\mathbb{P}$ is just $\leq$ on $C$, so we may set $D_F=\{c\}$.  Otherwise, $f^{\precsim_\mathbb{P}}$ is not a selector, for all $f\in F$, and hence $E'_F$ has a common refinement in $\mathsf{C}\mathbb{P}$ with each complement $\mathbb{P}\setminus f^{\precsim_\mathbb{P}}$ which, in turn, must refine $E''_F\cup(B^\geq\cap\bigcap_{f\in F}f^\geq\setminus C^{\precsim_\mathbb{P}})$.  So this last set is a cap and we may let $D_F$ be a minimal subset of $(B^\geq\cap\bigcap_{f\in F}f^\geq)\setminus C^{\precsim_\mathbb{P}}$ such that $E''_F\cup D_F$ is a cap.  As $D_F\subseteq\bigcap_{f\in F}f^\geq$ this implies that $D'_F=\bigcup_{G\subseteq F}D_G$ is minimal such that $D'_F\cup(C\setminus F)$ is a cap -- otherwise we would have $d\in D_G$, for some $G\subsetneqq F$, such that $(D'_F\setminus\{d\})\cup(C\setminus F)$ is a cap refining $D'_G\setminus\{d\}\cup(C\setminus G)$, contradicting the minimality of $D'_G$.
\end{proof}

Above, $\precsim_\mathbb{P}$ is again just $\leq$ on $C\cup D$.  Indeed, for any $c\in C$ and $d\in D$, $d\precsim_\mathbb{P}c$ implies $c\not\precsim E_d\setminus\{d\}$ (otherwise $d\precsim E_d\setminus\{d\}$, contradicting the minimality of $E_d$) and hence $d\leq c$.  On the other hand $c\precsim_\mathbb{P}d$ implies $c=d$ and, in particular, $c\leq d$.

This yields the following order theoretic analog of \autoref{PredeterminedSubBasis}.  Essentially it says that, given any $\omega$-poset $\mathbb{P}$, we can always revert to a branching predetermined $\omega$-subposet $\mathbb{Q}$ without significantly affecting caps or the spectrum (although it is worth noting that, even if $\mathbb{P}$ is graded, there is no guarantee $\mathbb{Q}$ will be graded too).

\begin{thm}\label{PredeterminedSubposet}
    Every $\omega$-poset $\mathbb{P}$ contains a predetermined branching $\omega$-poset $\mathbb{Q}$ with $\mathsf{C}\mathbb{Q}=\mathsf{C}\mathbb{P}\cap\mathsf{P}\mathbb{Q}$ such that $\mathsf{S}\mathbb{P}$ is homeomorphic to $\mathsf{S}\mathbb{Q}$ via the maps
    \[S\mapsto S\cap\mathbb{Q}\qquad\text{and}\qquad T\mapsto T^{\precsim_\mathbb{P}}.\]
\end{thm}

\begin{proof}
Recursively define minimal caps $(D_n)_{n\in\omega}$ and $(E_d^n)_{d\in D_n}^{n\in\omega}$ of $\mathbb{P}$ as follows.  First let $D_0$ be any minimal cap and set $E_0^d=D_0\setminus\{d\}$, for all $d\in D_0$.  Once $D_k$ has been defined, use the lemma above to define $D_{k+1}\precsim_\mathbb{P}B_k$ and $(E_d^k)_{d\in D_{k+1}}$ satisfying \eqref{dbigcap2}, where we take $C=C_k=\bigcup_{j\leq k}D_j$ and $B=\mathbb{P}_k$.  Note that then $D_{k+1}$ refines $D_k$ -- for any $d\in D_{k+1}$, $E_d^k\setminus\{d\}$ is not a cap and so we must have some $c\in D_k(\subseteq C_k)$ with $c\not\precsim E_d^k\setminus\{d\}$ and hence $d\leq c$, by the first part of \eqref{dbigcap2}.  As $D_k$ is a minimal cap, it must then also corefine $D_{k+1}$.  Moreover, as noted above, $\precsim_\mathbb{P}$ is just $\leq$ on $\mathbb{Q}=\bigcup_{n\in\omega}D_n$.  This and the second part of \eqref{dbigcap2} imply that $D_{k+1}\cap D_k$ consists only of atoms of $\mathbb{Q}$.  Thus $\mathbb{Q}$ is an $\omega$-poset with levels $\mathbb{Q}_n=D_n$, for all $n\in\omega$, by \autoref{AtomicPosetLevels}.

By \autoref{CapsRefineLevels}, every $C\in\mathsf{C}\mathbb{Q}$ is refined by some $\mathbb{Q}_n=D_n\in\mathsf{C}\mathbb{P}$, implying that $C\in\mathsf{C}\mathbb{P}$.  Conversely, if $C\in\mathsf{C}\mathbb{P}\cap\mathsf{P}\mathbb{Q}$ then, again by \autoref{CapsRefineLevels}, it is refined by some $\mathbb{P}_n$ and hence $D_n\precsim_\mathbb{P}C$.  As $\precsim_\mathbb{P}$ is $\leq$ on $\mathbb{Q}$, it follows that $D_n\leq C$ and hence $C\in\mathsf{C}\mathbb{Q}$.  This shows that $\mathsf{C}\mathbb{Q}=\mathsf{C}\mathbb{P}\cap\mathsf{P}\mathbb{Q}$.

It then follows that $\mathrm{id}_\mathbb{Q}\subseteq\mathbb{Q}\times\mathbb{Q}\subseteq\mathbb{Q}\times\mathbb{P}$ is a refiner.  As $\mathbb{Q}_n\precsim_\mathbb{P}\mathbb{P}_n$, for all $n\in\omega$, we have another refiner ${\sqsupset}={\succsim_\mathbb{P}}\cap\mathbb{P}\times\mathbb{Q}$.  Moreover, ${\mathrm{id}_\mathbb{Q}\circ{\sqsupset}}\subseteq{\geq_\mathbb{Q}}\subseteq{\succsim_\mathbb{Q}}$ and ${{\sqsupset}\circ\mathrm{id}_\mathbb{Q}}={\sqsupset}\subseteq{\succsim_\mathbb{P}}$ so \autoref{BirefinableImpliesHomeomorphic} yields mutually inverse homeomorphisms $S\mapsto S^{\mathrm{id}_\mathbb{Q}}=S\cap Q$ and $T\mapsto T^\sqsubset=T^{\precsim_\mathbb{P}}$ between $\mathsf{S}\mathbb{P}$ and $\mathsf{S}\mathbb{Q}$.

In particular, $(q^\in)_{q\in\mathbb{Q}}$ is a basis of $\mathsf{S}\mathbb{P}$, one which is order isomorphic to $\mathbb{Q}$, as $\precsim_\mathbb{P}$ is just $\leq$ on $\mathbb{Q}$.  Thus $\mathbb{Q}$ is branching, by \autoref{BranchingBasis}.  To see that $\mathbb{Q}$ is also predetermined, say $d\in\mathbb{Q}_n$ is not an atom in $\mathbb{Q}$.  Take any $q\in\mathbb{Q}_{n+1}$ such that $q\not\precsim E_n^d\setminus\{d\}$.  Note $q\leq c$, for some $c\in\mathbb{Q}_n$, necessarily with $c\not\precsim E_n^d\setminus\{d\}$ and hence $d\leq c$, which then implies $c=d$, as $\mathbb{Q}_n$ is an antichain.  Thus $q<d$ because $d$ is not an atom in $\mathbb{Q}$.  Likewise, if $q<c$, for some $c\in\mathbb{Q}_k$, necessarily with $k\geq n$, then $d\leq c$ so $q^<=d^\leq$, showing that $\mathbb{Q}$ is indeed predetermined.
\end{proof}

In order to prove that spectra of regular $\omega$-posets are homeomorphic we can also use a back-and-forth argument analogous to \autoref{BirefinableImpliesHomeomorphic}.

\begin{prp}\label{BackAndForthImpliesBirefinable}
If we have regular $\omega$-posets $\mathbb{P}$ and $\mathbb{Q}$ with coinitial sequences $(C_n)\subseteq\mathsf{C}\mathbb{P}$ and $(D_n)\subseteq\mathsf{C}\mathbb{Q}$ as well as co-$\wedge$-preserving surjective 
$\sqsubset_n\ \subseteq C_n\times D_n$ and $\sqin_n\ \subseteq D_{n+1}\times C_n$ with $\sqin_n\circ\sqsubset_n\ \subseteq\ \leq_\mathbb{Q}$ and $\sqsubset_{n+1}\circ\sqin_n\ \subseteq\ \leq_\mathbb{P}$, for all $n\in\omega$,
\[\sqsubset\ \ =\ \bigcup_{n\in\omega}(\sqsubset_n\circ\vartriangleleft_{D_n})\qquad\text{and}\qquad\sqin\ \ =\ \bigcup_{n\in\omega}(\sqin_n\circ\vartriangleleft_{C_n})\]
define $\wedge$-preserving refiners $\sqsupset$ and $\sqni$ such that $\phi_{\sqni}\circ\phi_\sqsupset=\mathrm{id}_{\mathsf{S}\mathbb{P}}$ and $\phi_\sqsupset\circ\phi_{\sqni}=\mathrm{id}_{\mathsf{S}\mathbb{Q}}$.
\end{prp}

\begin{proof}
As $\mathbb{Q}$ is regular, $\sqsupset$ is a refiner.  To see that $\sqsupset$ is $\wedge$-preserving, take $a\in C_m$ and $b\in C_n$ with $a\wedge b$.  If $m=n$ then $e\wedge f$ whenever $a\sqsubset_me$ and $b\sqsubset_nf$, by the assumption that $\sqsubset_m\ =\ \sqsubset_n$ is $\wedge$-preserving, and hence the same applies whenever $a\sqsubset_m\circ\vartriangleleft_{C_m}e$ and $b\sqsubset_n\circ\vartriangleleft_{C_n}f$.  Now assume that $m>n$ and take any $e,e',f,f'$ with $a\sqsubset_me\vartriangleleft_{D_m}e'$ and $b\sqsubset_nf\vartriangleleft_{D_n}f'$.  The surjectivity of the given relations then yields $c\in C_n$ and $d\in D_n$ satisfying
\[e\sqin_{m-1}\circ\sqsubset_{m-1}\circ\sqin_{m-2}\ldots\sqsubset_{n+1}\circ\sqin_nc\sqsubset_nd.\]
As $\sqsubset_{n+1}\circ\sqin_n\ \subseteq\ \leq_\mathbb{P}$, for all $n\in\omega$, it follows that $a\leq c$ and hence $b\wedge c$.  As $\sqsubset_n$ is co-$\wedge$-preserving, it follows that $f\wedge d$ and hence $d\leq f'$, as $f\vartriangleleft_{D_n}f'$.  Also $e\leq d$, as $\sqin_n\circ\sqsubset_n\ \subseteq\ \leq_\mathbb{Q}$, for all $n\in\omega$, so $e\leq f'$ and hence $e'\wedge f'$, as $e\vartriangleleft_{D_m}e'$.  A dual argument applies if $m<n$, thus showing that $\sqsupset$ is indeed $\wedge$-preserving.

Likewise, $\sqni$ is $\wedge$-preserving and hence we have continuous maps $\phi_\sqsupset:\mathsf{S}\mathbb{P}\rightarrow\mathsf{S}\mathbb{Q}$ and $\phi_{\sqni}:\mathsf{S}\mathbb{Q}\rightarrow\mathsf{S}\mathbb{P}$ as in \eqref{phisq}.  To see that $\phi_{\sqni}\circ\phi_\sqsupset=\mathrm{id}_{\mathsf{S}\mathbb{P}}$, take any $S\in\mathsf{S}\mathbb{P}$.  For any $A\in\mathsf{C}\mathbb{P}$, we have $B\in\mathsf{C}\mathbb{P}$ and $n\in\omega$ with $C_n\vartriangleleft_{C_n}B\vartriangleleft A$.  We then also have $E\in\mathsf{C}\mathbb{Q}$ and $m>n+1$ with $D_m\vartriangleleft_{D_m}E\vartriangleleft D_{n+1}$.  As $S$ is a selector, we have $s\in S\cap C_m$.  The surjectivity of all the relations involved then yields $a\in A$, $b\in B$, $c\in C_n$, $d\in D_{n+1}$, $e\in E$  and $f\in D_m$ with
\[s\sqsubset_mf\vartriangleleft_{D_m}e\vartriangleleft d\sqin_nc\vartriangleleft_{C_n}b\vartriangleleft a.\]
This means $s\sqsubset e\vartriangleleft d$ and hence $d\in\phi_\sqsupset(S)$.  Likewise, $d\sqin b\vartriangleleft a$ and hence $a\in\phi_{\sqni}(\phi_\sqsupset(S))$.  Now surjectivity again yields $q\in D_n$ and $p\in C_n$ with
\[f\sqin_{m-1}\circ\sqsubset_{m-1}\circ\sqin_{m-2}\ldots\sqsubset_{n+1}q\sqin_np.\]
As $\sqin_n\circ\sqsubset_n\ \subseteq\ \leq_\mathbb{Q}$, for all $n\in\omega$, it follows that $f\leq q$ and hence $d\wedge q$.  As $\sqin_n$ is co-$\wedge$-preserving, this implies $c\wedge p$ and hence $p\leq b\vartriangleleft a$.  Noting $s\leq p$, as $\sqsubset_{n+1}\circ\sqin_n\ \subseteq\ \leq_\mathbb{P}$, for all $n\in\omega$, it follows that $a\in S^{\leq\vartriangleleft}=S$.  We have thus shown that $S\cap\phi_{\sqni}(\phi_\sqsupset(S))\cap A\neq\emptyset$, for all $A\in\mathsf{C}\mathbb{P}$, i.e. $S\cap\phi_{\sqni}(\phi_\sqsupset(S))$ is a selector.  As $S$ and $\phi_{\sqni}(\phi_\sqsupset(S))$ are minimal selectors, this implies $S=\phi_{\sqni}(\phi_\sqsupset(S))$.  This shows that $\phi_{\sqni}\circ\phi_\sqsupset=\mathrm{id}_{\mathsf{S}\mathbb{P}}$ and a dual argument yields $\phi_\sqsupset\circ\phi_{\sqni}=\mathrm{id}_{\mathsf{S}\mathbb{Q}}$.
\end{proof}

As an application of \autoref{BackAndForthImpliesBirefinable}, we can use it to give an alternative proof of \autoref{SpacesFromGradedPosets}, at least in the Hausdorff case, one which gives us more control over the levels of the poset, like in \autoref{WeaklyGradedCapBasis}.

Let the \emph{gradification} $\mathbb{P}_\mathsf{G}$ of an $\omega$-poset $\mathbb{P}$ be the disjoint union of its levels, i.e.
\[\mathbb{P}_\mathsf{G}=\bigsqcup_{n\in\omega}\mathbb{P}_n=\bigcup_{n\in\omega}\mathbb{P}_n\times\{n\}.\]
To define the order on $\mathbb{P}_\mathsf{G}$, we first define the predecessor relation $\lessdot$ by
\[(p,n)\lessdot(q,m)\qquad\Leftrightarrow\qquad p\leq q\text{ and }m\lessdot n.\]
Let $\leq^0$ be the equality relation on $\mathbb{P}_\mathsf{G}$ and recursively define ${\leq^{n+1}}={\leq^n}\circ{\lessdot}={\lessdot}\circ{\leq^n}$, i.e. $\leq^n$ is just the composition of $\lessdot$ on $\mathbb{P}_\mathsf{G}$ with itself $n$ times.  Finally let ${\leq}=\bigcup_{n\in\omega}\leq^n$ on $\mathbb{P}_\mathsf{G}$.  In particular, the strict order $<$ on $\mathbb{P}_\mathsf{G}$ is just the transitive closure of the predecessor relation defined above.

The following result is now immediate from the construction.

\begin{prp}
    If $\mathbb{P}$ is an $\omega$-poset then $\mathbb{P}_\mathsf{G}$ is an atomless graded $\omega$-poset with
    \[\mathbb{P}_{\mathsf{G}n}=\mathbb{P}_n\times\{n\}.\]
\end{prp}

Let us call an $\omega$-poset \emph{edge-witnessing} if common lower bounds of elements in any level are always witnessed on the next, i.e. whenever $q,r\in\mathbb{P}_n$ and $q\wedge r$, we have $p\in\mathbb{P}_{n+1}$ with $p\leq q$ and $p\leq r$.  Likewise, we call an $\omega$-poset \emph{star-refining} if each level is star-refined by the next, i.e. $\mathbb{P}_{n+1}\vartriangleleft_{\mathbb{P}_{n+1}}\mathbb{P}_n$, for all $n\in\omega$.

\begin{prp}\label{GradificationSpectrum}
    The spectrum $\mathsf{S}\mathbb{P}$ of any edge-witnessing star-refining $\omega$-poset $\mathbb{P}$ is always homeomorphic to the spectrum of its gradification $\mathsf{S}\mathbb{P}_\mathsf{G}$.
\end{prp}

\begin{proof}
    For all $n\in\omega$, define ${\sqsubset_n}\subseteq\mathbb{P}_{\mathsf{G}n}\times\mathbb{P}_n$ and ${\sqin_n}\subseteq\mathbb{P}_{n+1}\times\mathbb{P}_{\mathsf{G}n}$ by
    \begin{align*}
        (p,n)\sqsubset_nq\qquad&\Leftrightarrow\qquad p=q.\\
        p\sqin_n(q,n)\qquad&\Leftrightarrow\qquad p\leq q.
    \end{align*}
    For each $n\in\omega$, we immediately see that $\sqsubset_n$ and $\sqin_n$ are surjective, $\sqsubset_n$ is co-$\wedge$-preserving, ${\sqin_n\circ\sqsubset_n}\subseteq{\leq_\mathbb{P}}$ and ${\sqsubset_{n+1}\circ\sqin_n}\subseteq{\leq_{\mathbb{P}_\mathsf{G}}}$.  As $\mathbb{P}$ is edge-witnessing, $\sqin_n$ is also co-$\wedge$-preserving.  As $\mathbb{P}$ is star-refining, so is $\mathbb{P}_\mathsf{G}$.  In particular, both $\mathbb{P}$ and $\mathbb{P}_\mathsf{G}$ are regular so $\mathsf{S}\mathbb{P}$ is homeomorphic to $\mathsf{S}\mathbb{P}_\mathsf{G}$, by \autoref{BackAndForthImpliesBirefinable}.
\end{proof}

Let us illustrate the usefulness of the above result with snake-like spaces.  First let us call an open cover $S$ a \emph{snake} if its overlap graph is a path, i.e. if there exists an enumeration $s_1,\ldots,s_n$ of $S$ such that
\[s_m\cap s_n\neq\emptyset\qquad\Leftrightarrow\qquad|m-n|\leq1.\]
We call $X$ \emph{snake-like} if every open cover is refined by a snake (this is a standard notion in continuum theory, also called \emph{chainable} as in \cite[\S12.8]{Nadler1992}).  In particular, every snake-like space is compact because snakes are finite.  Also, if $X=Y\cup Z$ for non-empty clopen $Y$ and $Z$, then any refinement of $\{Y,Z\}$ can not be a snake, i.e. snake-like spaces are necessarily connected as well.

\begin{prp}
    Every metrisable snake-like $X$ has a graded $\omega$-band-basis whose levels are all snakes.
\end{prp}

\begin{proof}
    As $X$ is connected, any minimal subcover of a snake is again a snake.  As $X$ is metrisable and snake-like, we thus have a countable collection $\mathcal{C}$ of minimal snakes which are coinitial w.r.t. refinement among all open covers.  By \autoref{WeaklyGradedCapBasis}, we have a subfamily forming the levels of a level-injective $\omega$-cap-basis $\mathbb{P}$.  By \autoref{SpaceRecovery}, the spectrum $\mathsf{S}\mathbb{P}$ is homeomorphic to the original space $X$.  If necessary, we can replace $\mathbb{P}$ with a subposet consisting of infinitely many levels which is also edge-witnessing and star-refining.  By \autoref{HomeomorphicSubposet}, $\mathsf{S}\mathbb{P}$ will still be homeomorphic to $X$, as will $\mathsf{S}\mathbb{P}_\mathsf{G}$, by \autoref{GradificationSpectrum}.  As each level of $\mathbb{P}_\mathsf{G}$ corresponds to a snake in $\mathsf{S}\mathbb{P}_\mathsf{G}$ and hence $X$, we are done.
\end{proof}

\subsection{Star-Composition}\label{StarComposition}

Let us define the \emph{star} $\sqsupset^*$ of any ${\sqsupset}\subseteq\mathbb{Q}\times\mathbb{P}$ by
\[q\sqsupset^*p\qquad\Leftrightarrow\qquad\exists C\in\mathsf{C}\mathbb{P}\ (Cp\sqsubset q).\]
For example, the star-above relation is the star of both $\geq$ and $\mathrm{id}_\mathbb{P}$, i.e.
${\mathrm{id}_\mathbb{P}^*}={\geq^*}={\vartriangleright}$.

If $\sqsupset_\phi$ is defined by containment relative to some $\phi:\mathsf{S}\mathbb{P}\rightarrow\mathsf{S}\mathbb{Q}$, as in \eqref{sqphi}, its star then corresponds to closed containment.

\begin{prp}\label{sqphistar}
If $\mathbb{P}$ is a prime regular $\omega$-poset and $\phi:\mathsf{S}\mathbb{P}\rightarrow\mathsf{S}\mathbb{Q}$ is continuous,
\[q\sqsupset_\phi^*p\qquad\Leftrightarrow\qquad\phi^{-1}[q^\in]\supseteq\mathrm{cl}(p^\in).\]
\end{prp}

\begin{proof}
    Say $q\sqsupset_\phi^*p$, so we have $C\in\mathsf{C}\mathbb{P}$ with $q\sqsupset_\phi c$, for all $c\in Cp$, and hence
    \[\phi^{-1}[q^\in]\supseteq(Cp)^\in\supseteq p^{\wedge\supseteq}=\mathrm{cl}(p^\in),\]
    by \eqref{Closure} and \eqref{Upni} (if $S\in p^{\wedge\supseteq}$ then $S\subseteq p^\wedge$ so we have $c\in C\cap S\subseteq Cp$ and hence $S\in(Cp)^\in$).  This proves the $\Rightarrow$ part.

    Conversely, assume $\phi^{-1}[q^\in]\supseteq\mathrm{cl}(p^\in)$.  As $\mathbb{P}_\mathsf{S}$ is a basis for $\mathsf{S}\mathbb{P}$, we have a cover $C_S$ of $\mathsf{S}\mathbb{P}$ such that either $c^\in\subseteq\phi^{-1}[q^\in]$ or $c^\in\subseteq\mathsf{S}\mathbb{P}\setminus\mathrm{cl}(p^\in)$, for all $c\in C$.  Thus $C\in\mathsf{C}\mathbb{P}$, by \autoref{SpectrumCompactT1}, and $c^\in\subseteq\phi^{-1}[q^\in]$, whenever $c\in Cp$.  This means $q\sqsupset_\phi c$, for all $c\in Cp$, so $C$ witnesses $q\sqsupset_\phi^*p$.
\end{proof}

Another thing we can note immediately about stars is the following.

\begin{prp}\label{StarWedgePreserving}
    If ${\sqsupset}\subseteq\mathbb{Q}\times\mathbb{P}$ is $\wedge$-preserving then so is $\sqsupset^*$.
\end{prp}

\begin{proof}
    Say $\sqsupset$ is $\wedge$-preserving.  If $q\sqsupset^*p$ and $q'\sqsupset^*p'$ then we have $C,C'\in\mathsf{C}\mathbb{P}$ with $Cp\sqsubset q$ and $C'p'\sqsubset q'$.  If $p\wedge p'$ then \autoref{WedgeSplit} yields $c\in Cp$ with $c\wedge p'$.  Then \autoref{WedgeSplit} again yields $c'\in C'p'$ with $c\wedge c'$.  Thus $q\wedge q'$, as $q\sqsupset c\wedge c'\sqsubset q'$ and $\sqsupset$ is $\wedge$-preserving, showing that $\sqsupset^*$ is also $\wedge$-preserving.
\end{proof}

Also, stars do not change the up-closures of round star-prime subsets.

\begin{prp}
    For any ${\sqsupset}\subseteq\mathbb{Q}\times\mathbb{P}$ and $S\subseteq\mathbb{P}$,
    \begin{equation}\label{RoundStarPrimeStar}
        S\text{ is round and star-prime}\qquad\Rightarrow\qquad S^\sqsubset=S^{\reflectbox{$\scriptstyle\sqsupset^*$}}.
    \end{equation}
\end{prp}

\begin{proof}
    If $S$ is round then $S^\sqsubset=S^{\vartriangleleft\sqsubset}\subseteq S^{\reflectbox{$\scriptstyle\sqsupset^*$}}$.  On the other hand, if $q\in S^{\reflectbox{$\scriptstyle\sqsupset^*$}}$ then we have $s\in S$ with $q\sqsupset^*s$, which means we have $C\in\mathsf{C}\mathbb{P}$ with $Cs\sqsubset q$.  If $S$ is star-prime then we have $c\in Cs\cap S\sqsubset q$ so $q\in S^\sqsubset$, showing that $S^{\reflectbox{$\scriptstyle\sqsupset^*$}}\subseteq S^\sqsubset$.
\end{proof}

Define the \emph{star-composition} of any ${\sqni}\subseteq\mathbb{R}\times\mathbb{Q}$ and ${\sqsupset}\subseteq\mathbb{Q}\times\mathbb{P}$ by
\[{\sqni*\sqsupset}\ =\ (\sqni\circ\sqsupset)^*.\]
This more accurately reflects composition of continuous functions, as we now show.

\begin{prp}
If $\mathbb{P}$, $\mathbb{Q}$ and $\mathbb{R}$ are prime regular $\omega$-posets then, for any continuous maps $\phi:\mathsf{S}\mathbb{P}\rightarrow\mathsf{S}\mathbb{Q}$ and $\psi:\mathsf{S}\mathbb{Q}\rightarrow\mathsf{S}\mathbb{R}$,
\begin{equation*}%\label{sqfunctorStar}
{\sqsupset_\psi^**\sqsupset_\phi^*}={\sqsupset_\psi*\sqsupset_\phi}={\sqsupset_{\psi\circ\phi}^*}.
\end{equation*}
\end{prp}

\begin{proof}
    By \autoref{sqphistar}, ${\sqsupset_\psi^*}\subseteq{\sqsupset_\psi}$ and ${\sqsupset_\phi^*}\subseteq{\sqsupset_\phi}$ so ${\sqsupset_\psi^**\sqsupset_\phi^*}\subseteq{\sqsupset_\psi*\sqsupset_\phi}$.
    On the other hand, if $r\sqsupset_\psi*\sqsupset_\phi p$ then we have $C\in\mathsf{C}\mathbb{P}$ such that, for all $c\in Cp$, we have $q_c\in\mathbb{Q}$ with $r\sqsupset_\psi q_c\sqsupset_\phi c$.  This means
    \[\mathrm{cl}(p^\in)\subseteq\bigcup_{c\in Cp}c^\in\subseteq\bigcup_{c\in Cp}\phi^{-1}[q_c^\in]\subseteq\phi^{-1}[\psi^{-1}[r^\in]]=(\psi\circ\phi)^{-1}[r^\in].\]
    By \autoref{sqphistar}, this implies $r\sqsupset_{\psi\circ\phi}^*p$.

    Now say $r\sqsupset_{\psi\circ\phi}^*p$, i.e. $\mathrm{cl}(p^\in)\subseteq\phi^{-1}[\psi^{-1}[r^\in]]$.  For each $S\in\mathrm{cl}(p^\in)$, the continuity of $\psi$ yields $q\in\phi(S)$ with $\mathrm{cl}(q^\in)\subseteq\psi^{-1}[r^\in]$.  The continuity of $\phi$ then yields $c\in S$ with $\mathrm{cl}(c^\in)\subseteq\phi^{-1}[q^\in]$.  On the other hand, for every $S\in\mathsf{S}\mathbb{P}\setminus\mathrm{cl}(p^\in)$, we have $c\in S$ with $c^\geq\cap p^\geq=\emptyset$.  As $\mathsf{S}\mathbb{P}$ is compact, it has a finite cover consisting of $c^\in$ for such $c$.  By \autoref{SpectrumCompactT1}, these form a cap, i.e. we have $C\in\mathsf{C}\mathbb{P}$ such that $r\sqsupset_\psi^*\circ\sqsupset_\phi^*c$, for all $c\in Cp$, showing that $r\sqsupset_\psi^**\sqsupset_\phi^*p$.
\end{proof}

Also, replacing $\circ$ with $*$ in \eqref{StarCircSq} turns $\Rightarrow$ into $\Leftrightarrow$.

\begin{prp}
    If $\mathbb{P}$ and $\mathbb{Q}$ are regular prime $\omega$-posets and ${\sqsupset}\subseteq\mathbb{Q}\times\mathbb{P}$ is a $\wedge$-preserving refiner then, for all $p\in\mathbb{P}$ and $q\in\mathbb{Q}$,
    \begin{equation}\label{StarClosedContainment}
        q\vartriangleright*\sqsupset p\qquad\Leftrightarrow\qquad\phi_\sqsupset^{-1}[q^\in]\supseteq\mathrm{cl}(p^\in).
    \end{equation}
\end{prp}

\begin{proof}
    If $q\vartriangleright*\sqsupset p$ then we have $C\in\mathsf{C}\mathbb{P}$ with $Cp\sqsubset\circ\vartriangleleft q$ so \eqref{Upni} and \eqref{StarCircSq} yield
    \[\mathrm{cl}(p^\in)\subseteq(Cp)^\in\subseteq\mathrm{cl}((Cp)^\in)\subseteq\phi_\sqsupset^{-1}[q^\in].\]

    Conversely, if $q\vartriangleright*\sqsupset p$ fails then $Cp\nsubseteq q^{\vartriangleright\sqsupset}$, for all $C\in\mathsf{C}\mathbb{P}$.  Put another way, $p^\wedge\setminus q^{\vartriangleright\sqsupset}$ is a selector and hence contains a minimal selector $S$.  Then $S\in\mathrm{cl}(p^\in)$, by \eqref{Closure} and \eqref{Upni}, but $q\notin S^{\sqsubset\vartriangleleft}=\phi_\sqsupset(S)$, i.e. $\phi_\sqsupset(S)\notin q^\in$ so $S\in\mathrm{cl}(p^\in)\setminus\phi_\sqsupset^{-1}[q^\in]$.
\end{proof}

Next let us make some simple observations about $*$.  For example,
\begin{equation}\label{CircStar}
    {\sqni*\sqsupset}\ \supseteq\ {\sqni\circ\sqsupset^*}.
\end{equation}
Indeed, if $r\sqni q\sqsupset^*p$ then we have $C\in\mathsf{C}\mathbb{P}$ with $Cp\sqsubset q\sqin r$ so $r\mathrel{\sqni*\sqsupset}p$.  Thus
\begin{equation}\label{StarTriangle}
    {\sqsupset\circ\vartriangleright}\ \subseteq\ {\sqsupset*\geq}\ =\ {\sqsupset^*\circ\geq}\ =\ {\sqsupset^*}.
\end{equation}
Indeed, the first inclusion is just a special case of \eqref{CircStar} where $\sqni$ and $\sqsupset$ are replaced by $\sqsupset$ and $\geq$ respectively.  On the other hand, if $q\sqsupset^*p\geq r$ then we have $C\in\mathsf{C}\mathbb{P}$ with $Cr\subseteq Cp\leq q$ and hence $q\sqsupset^*r$.  This shows that ${\sqsupset^*\circ\geq}={\sqsupset^*}$.  Also certainly ${\sqsupset^*}\subseteq(\sqsupset\circ\geq)^*={\sqsupset*\geq}$.  Conversely, if $q\sqsupset*\geq p$ then we have $C\in\mathsf{C}\mathbb{P}$ such that, for all $c\in Cp$, we have $q_c\in\mathbb{P}$ with $c\leq q_c\sqsubset q$.  Setting $D=(C\setminus Cp)\cup\{q_c:c\in Cp\}$, note $C\leq D\in\mathsf{C}\mathbb{P}$ and $Dp\sqsubset q$, i.e. $D$ witnesses $q\sqsupset^*p$, showing ${\sqsupset*\geq}\subseteq{\sqsupset^*}$ too.

\begin{prp}
    If $\mathbb{P}$ is a regular and ${\sqsupset}\subseteq\mathbb{Q}\times\mathbb{P}$ is a refiner then so is $\sqsupset^*$.
\end{prp}

\begin{proof}
    As $\mathbb{P}$ is regular, $\vartriangleright$ is a refiner.  As $\sqsupset$ is a refiner too, so is $\sqsupset\circ\vartriangleright$ and hence so too is ${\sqsupset^*}\supseteq{\sqsupset\circ\vartriangleright}$, by \eqref{StarTriangle}.
\end{proof}

Combined with \autoref{StarWedgePreserving}, this means ${\sqsupset^*}\in\mathbf{P}^\mathbb{Q}_\mathbb{P}$ whenever ${\sqsupset}\in\mathbf{P}^\mathbb{Q}_\mathbb{P}$.  Moreover, $\phi_\sqsupset=\phi_{\sqsupset^*}$, by \eqref{RoundStarPrimeStar}.  We can further characterise when $\phi_\sqsupset=\phi_{\sqni}$ as follows.

\begin{cor}\label{StarEquivalence}
    For any ${\sqsupset}\in\mathbf{P}_\mathbb{P}^\mathbb{Q}$,
    \[\phi_\sqsupset=\phi_{\sqni}\qquad\Leftrightarrow\qquad{\vartriangleright*\sqsupset}={\vartriangleright*\sqni}.\]
\end{cor}

\begin{proof}
    If $\phi_\sqsupset=\phi_{\sqni}$ then \eqref{StarClosedContainment} yields 
    \[q\vartriangleright*\sqsupset p\quad\Leftrightarrow\quad\phi_\sqsupset^{-1}[q^\in]\supseteq\mathrm{cl}(p^\in)\quad\Leftrightarrow\quad\phi_{\sqni}^{-1}[q^\in]\supseteq\mathrm{cl}(p^\in)\quad\Leftrightarrow\quad q\vartriangleright*\sqni p.\]
    Conversely, if ${\vartriangleright*\sqsupset}={\vartriangleright*\sqni}$ then \eqref{RoundStarPrimeStar} yields
    \[\phi_\sqsupset(S)=S^{\sqsubset\vartriangleleft}=S^{\sqsubset\circ\vartriangleleft}=S^{\reflectbox{$\scriptstyle(\vartriangleright\circ\sqsupset)^*$}}=S^{\reflectbox{$\scriptstyle(\vartriangleright\circ\sqni)^*$}}=S^{\sqin\circ\vartriangleleft}=S^{\sqin\vartriangleleft}=\phi_{\sqni}(S).\qedhere\]
\end{proof}

Here are some further simple combinatorial properties of star-composition.

\begin{prp}
    If $\mathbb{P}$ is a regular $\omega$-poset, ${\sqni}\subseteq\mathbb{R}\times\mathbb{Q}$ and ${\sqsupset}\subseteq\mathbb{Q}\times\mathbb{P}$ then
    \begin{equation}\label{StarEquivalents}
        {\sqni*\sqsupset}={\sqni*\sqsupset^*}=(\sqni*\sqsupset)^*.
    \end{equation}
\end{prp}

\begin{proof}    
    First we claim that
    \begin{equation}\label{DoubleStar}
        {\sqsupset^{**}}={\sqsupset^*}={\sqsupset*\vartriangleright}.
    \end{equation}
    Indeed, if $q\sqsupset^{**}p$ then we have $C\in\mathsf{C}\mathbb{P}$ such that, for all $c\in Cp$, $q\sqsupset^*c$ and hence $B_cc\sqsubset q$, for some $B_c\in\mathsf{C}\mathbb{P}$.  Replacing $C$ with a finite subcap if necessary, we can then take $A\in\mathsf{C}\mathbb{P}$ refining $B_c$, for all $c\in Cp$, as $\mathbb{P}$ is an $\omega$-poset.  We claim that $Ap\leq\circ\sqsubset q$.  Indeed, if $a\in Ap$ then \autoref{WedgeSplit} yields $c\in Cp$ with $a\wedge c$.  As $A\leq B_c$, we then have $b\in B_c$ with $c\wedge a\leq b$.  Thus $b\in B_cc\sqsubset q$ so $a\leq b\sqsubset q$, proving the claim.  In particular, $A$ witnesses $q\sqsupset*\geq p$, showing that ${\sqsupset^{**}}\subseteq{\sqsupset*\geq}={\sqsupset^*}$.

    Conversely, if $q\sqsupset^*p$ then we have $C\in\mathsf{C}\mathbb{P}$ with $Cp\sqsubset q$.  Regularity then yields $B\in\mathsf{C}\mathbb{P}$ with $B\vartriangleleft C$ so $Bp\vartriangleleft Cp\sqsubset q$ and hence $q\mathrel{\sqsupset\circ\vartriangleright}b$, for all $b\in Bp$.  Thus $B$ witnesses $q\mathrel{(\sqsupset\circ\vartriangleright)^*}p$, showing that ${\sqsupset^*}\subseteq(\sqsupset\circ\vartriangleright)^*\subseteq{\sqsupset^{**}}$, by \eqref{StarTriangle}, completing the proof of \eqref{DoubleStar}.

    In particular, ${(\sqni\circ\sqsupset)^{**}}={(\sqni\circ\sqsupset)^*}=(\sqni\circ\sqsupset\circ\vartriangleright)^*\subseteq(\sqni\circ\sqsupset^*)^*$,
    by \eqref{StarTriangle}.  In terms of star-composition, this means that $(\sqni*\sqsupset)^*={\sqni*\sqsupset}\subseteq{\sqni*\sqsupset^*}$.  But ${\sqni*\sqsupset^*}=(\sqni\circ\sqsupset^*)^*\subseteq(\sqni*\sqsupset)^*$, by \eqref{CircStar}, completing the proof of \eqref{StarEquivalents}.
\end{proof}

\begin{prp}\label{StarCircSub}
    For any ${\sqni}\subseteq\mathbb{R}\times\mathbb{Q}$ and $\wedge$-preserving refiner ${\sqsupset}\subseteq\mathbb{Q}\times\mathbb{P}$,
    \[{\sqni^*\circ\sqsupset}\subseteq{\sqni*\sqsupset}.\]
\end{prp}

\begin{proof}
    If $r\sqni^*q\sqsupset p$ then we have $C\in\mathsf{C}\mathbb{Q}$ with $Cq\sqni r$.  As $\sqsupset$ is a $\wedge$-preserving refiner, we then have $B\in\mathsf{C}\mathbb{P}$ with $B\sqsubset C$ and hence $Bp\sqsubset Cq\sqin r$. So $B$ witnesses $r\sqni*\sqsupset p$, showing that ${\sqni^*\circ\sqsupset}\subseteq{\sqni*\sqsupset}$.
\end{proof}

Under suitable conditions, we can now show that star-composition is associative.

\begin{prp}
    If $\mathbb{P}$ is a regular $\omega$-poset, ${\sqsupset}\subseteq\mathbb{Q}\times\mathbb{P}$ is a $\wedge$-preserving refiner, ${\sqni}\subseteq\mathbb{R}\times\mathbb{Q}$ satisfies ${\sqni}\subseteq{\sqni^*}$ and ${\sqnii}\subseteq\mathbb{S}\times\mathbb{R}$ then
    \[{{\sqnii}*(\sqni*\sqsupset)}\ =\ (\sqnii\circ\sqni\circ\sqsupset)^*\ =\ {(\sqnii*\sqni)*{\sqsupset}}.\]
\end{prp}

\begin{proof}
    First note \eqref{StarEquivalents} immediately yields
    \[{{\sqnii}*(\sqni*\sqsupset)}={{\sqnii}*(\sqni\circ\sqsupset)^*}={{\sqnii}*(\sqni\circ\sqsupset)}=(\sqnii\circ\sqni\circ\sqsupset)^*.\]
    Likewise, \eqref{StarEquivalents} and \autoref{StarCircSub} yield
    \[{(\sqnii*\sqni)*{\sqsupset}}=((\sqnii\circ\sqni)^*\circ{\sqsupset})^*\subseteq((\sqnii\circ\sqni)*{\sqsupset})^*={(\sqnii\circ\sqni)*{\sqsupset}}=(\sqnii\circ\sqni\circ\sqsupset)^*.\]
    Conversely, as ${\sqni}\subseteq{\sqni^*}$, \eqref{CircStar} yields
    \[(\sqnii\circ\sqni\circ\sqsupset)^*\subseteq(\sqnii\circ\sqni^*\circ\sqsupset)^*\subseteq((\sqnii*\sqni)\circ{\sqsupset})^*=(\sqnii*\sqni)*{\sqsupset}.\qedhere\]
\end{proof}

Let us call ${\sqsupset}\in\mathbf{P}_\mathbb{P}^\mathbb{Q}$ a \emph{strong refiner} if
\[\tag{Strong Refiner}{\sqsupset}={\vartriangleright*\sqsupset}.\]
For any other ${\sqni}\in\mathbf{P}_\mathbb{R}^\mathbb{P}$, we immediately see that ${\sqsupset*\sqni}={\vartriangleright*\sqsupset*\sqni}$.  In particular, strong refiners are closed under star-composition.  They are also star-invariant, as
\[{\sqsupset^*}=(\vartriangleright*\sqsupset)^*=(\vartriangleright\circ\sqsupset)^{**}=(\vartriangleright\circ\sqsupset)^*={\sqsupset}.\]
Moreover, ${\vartriangleright*\vartriangleright}={\geq^**\geq^*}={\geq*\geq}={(\geq\circ\geq)^*}={\geq^*}={\vartriangleright}$, showing that $\vartriangleright$ is also a strong refiner on any $\mathbb{P}\in\mathbf{P}$.  Furthermore, ${\vartriangleright*\sqsupset}={\sqsupset}={\sqsupset^*}={\sqsupset*\vartriangleright}$ showing that each $\vartriangleright_\mathbb{P}$ is an identity with respect to star-composition.  In other words, we have a category $\mathbf{S}$ with the same objects as $\mathbf{P}$ (prime regular $\omega$-posets) but with strong refiners as morphisms under star-composition.

In fact, $\mathbf{S}$ is equivalent to $\mathbf{K}$, as witnessed by the map $\mathsf{S}$ from \autoref{SFunctor}.
    
\begin{thm}
    $\mathsf{S}|_\mathbf{S}:\mathbf{S}\rightarrow\mathbf{K}$ is a fully faithful essentially surjective functor such that $\mathsf{S}=\mathsf{S}|_\mathbf{S}\circ\mathsf{Q}$, where $\mathsf{Q}:\mathbf{P}\rightarrow\mathbf{S}$ is the functor defined by $\mathsf{Q}(\sqsupset)={\vartriangleright*\sqsupset}$.
\end{thm}

\begin{proof}
	For any ${\sqsupset}\in\mathbf{P}_\mathbb{P}^\mathbb{Q}$, \eqref{DoubleStar} yields
	\[{\vartriangleright*\sqsupset*\vartriangleright}={(\vartriangleright*\sqsupset)^*}={(\vartriangleright\circ\sqsupset)^{**}}={(\vartriangleright\circ\sqsupset)^{*}}={\vartriangleright*\sqsupset}.\]
	For any other ${\sqni}\in\mathbf{P}_\mathbb{Q}^\mathbb{R}$, it follows that
	\[{\vartriangleright*\sqsupset*\vartriangleright*\sqni}={\vartriangleright*\sqsupset*\sqni}=(\vartriangleright\circ\sqsupset\circ\sqni)^*={{\vartriangleright}*(\sqsupset\circ\sqni)}.\]
	This shows that $\mathsf{Q}$ defined by $\mathsf{Q}(\sqsupset)={\vartriangleright*\sqsupset}$ preserves the product.  Moreover,
	\[{{\vartriangleright_\mathbb{P}}*\mathrm{id}_\mathbb{P}}=({\vartriangleright_\mathbb{P}}\circ\mathrm{id}_\mathbb{P})^*={\vartriangleright_\mathbb{P}^*}={\geq_\mathbb{P}^{**}}={\geq_\mathbb{P}^*}={\vartriangleright_\mathbb{P}}.\]
	As each ${\vartriangleright_\mathbb{P}}$ is an identity in $\mathbf{S}$, this shows that $\mathsf{Q}$ is a functor.  Also
	\[\phi_{\sqni*\sqsupset}=\phi_{(\sqni\circ\sqsupset)^*}=\phi_{\sqni\circ\sqsupset}=\phi_{\sqni}\circ\phi_\sqsupset\]
	and $\phi_{\vartriangleright_\mathbb{P}}=\mathrm{id}_{\mathsf{S}\mathbb{P}}$ (because $S=S^\vartriangleleft=S^{\vartriangleleft\vartriangleleft}$, for all $S\in\mathsf{S}\mathbb{P}$), so $\mathsf{S}|_\mathbf{S}$ is a functor too.  In particular, this also yields $\phi_{\vartriangleright*\sqsupset}=\phi_{\vartriangleright}\circ\phi_{\sqsupset}=\phi_{\sqsupset}$, showing that $\mathsf{S}=\mathsf{S}|_\mathbf{S}\circ\mathsf{Q}$.  As $\mathsf{S}$ is full and essentially surjective, so is $\mathsf{S}|_\mathbf{S}$.  By \autoref{StarEquivalence}, $\mathsf{S}|_\mathbf{S}$ is also faithful.
\end{proof}

The functor $\mathsf{Q}$ thus replaces any ${\sqsupset}\in\mathbf{P}_\mathbb{P}^\mathbb{Q}$ with a canonical representative in the same equivalence class defined by $\mathsf{S}$, namely the unique representative which corresponds exactly to closed containment, by \eqref{StarClosedContainment}.  The natural topology on strong refiners thus corresponds exactly to the compact-open/uniform convergence topology.  More precisely, the functor $\mathsf{S}|_\mathbf{S}$ is a homeomorphism from each hom-set $\mathbf{S}_\mathbb{P}^\mathbb{Q}$, considered as a subspace of the power-space $\mathsf{P}(\mathbb{Q}\times\mathbb{P})$ (i.e. with the topology generated by sets of the form $\{{\sqsupset}\in\mathsf{S}_\mathbb{P}^\mathbb{Q}:q\sqsupset p\}$, for $p\in\mathbb{P}$ and $q\in\mathbb{Q}$), to the hom-set $\mathbf{K}_{\mathsf{S}\mathbb{P}}^{\mathsf{S}\mathbb{Q}}$ with its compact-open/uniform convergence topology.  We plan to make use of this in future work on dynamical systems constructed from posets and refiners.

\begin{rmk}
    One could also make other choices of representative morphisms.  For example, for any ${\sqsupset}\in\mathbf{P}_\mathbb{P}^\mathbb{Q}$, we could define ${\sqsupseteq}\subseteq\mathbb{Q}\times\mathbb{P}$ by
    \[q\sqsupseteq p\qquad\Leftrightarrow\qquad q^\sqsupset\supseteq p^\vartriangleright.\]
    Then ${\sqsupset}\mapsto\underline{\vartriangleright*\sqsupset}$ again defines a functor selecting a representative in the equivalence class defined by $\mathsf{S}$, this time corresponding to mere containment, i.e.
    \[q\mathrel{\underline{\vartriangleright*\sqsupset}}p\qquad\Leftrightarrow\qquad\phi_\sqsupset^{-1}[q^\in]\supseteq p^\in.\]
    However, the natural topology on such refiners will be different and thus less useful when it comes to considering dynamical systems.
\end{rmk}

\bibliography{maths}{}
\bibliographystyle{abbrvurl}
\end{document}